\documentclass[12pt]{article}
\usepackage{verbatim,amsmath,amssymb,mathrsfs,amsthm,amsfonts,bm}
\usepackage{varwidth}
\usepackage{extarrows}
\usepackage{enumitem}
\usepackage{fancyhdr}
\usepackage{booktabs}
\usepackage[dvipsnames]{xcolor}
\usepackage{graphicx,wrapfig}
\usepackage{subfigure}
\usepackage{caption}
\usepackage{amsmath}
\usepackage{makecell}
\usepackage{multirow}
\usepackage{array}
\usepackage{setspace}
\usepackage{float}
\usepackage{cite}
\usepackage{color,mathtools}
\usepackage{epstopdf}
\usepackage{lipsum}
\usepackage{appendix}
\usepackage{algorithm}
\usepackage{algorithmic}
\usepackage{blindtext}
\usepackage{ulem}

\numberwithin{equation}{section}
\numberwithin{figure}{section}
\numberwithin{table}{section}
\allowdisplaybreaks

\newtheorem{theorem}{Theorem}[section]
\newtheorem{lemma}[theorem]{Lemma}

\newtheorem{property}[theorem]{Property}

\newtheorem{remark}{Remark}[section]
\newtheorem{example}{Example}[section]


\usepackage{hyperref}
\hypersetup{
    colorlinks=true, linkcolor=black,
    urlcolor=cyan, citecolor=black,
}

\begin{document}

\title{Structure-preserving finite volume
\\ arbitrary Lagrangian-Eulerian WENO  schemes \\for the shallow water equations}
\date{}
\author{}
\author{Jiahui Zhang\thanks{School of Mathematical Sciences, University of Science and Technology of China, Hefei, Anhui 230026, P.R. China.  E-mail: zjh55@mail.ustc.edu.cn.}
	\and Yinhua Xia\thanks{School of Mathematical Sciences, University of Science and Technology of China, Hefei, Anhui 230026, P.R. China.  E-mail: yhxia@ustc.edu.cn. 
	}
	\and Yan Xu\thanks{Corresponding author.  School of Mathematical Sciences, University of Science and Technology of China, Hefei, Anhui 230026, P.R. China.  E-mail: yxu@ustc.edu.cn. 
	}
}

\maketitle

\begin{abstract}
This paper develops the structure-preserving finite volume weighted essentially non-oscillatory (WENO) hybrid schemes for the shallow water equations under the arbitrary Lagrangian-Eulerian (ALE) framework, dubbed as ALE-WENO schemes.  The WENO hybrid reconstruction is adopted on moving meshes, which distinguishes the smooth, non-smooth, and transition stencils by a simple smoothness detector. To maintain the positivity preserving and the well-balanced properties of the ALE-WENO schemes,  we adapt the positivity preserving limiter and the well-balanced approaches on static meshes to moving meshes.  The rigorous theoretical analysis and numerical examples demonstrate the high order accuracy and positivity-preserving property of the schemes under the ALE framework. For the well-balanced schemes, it is successful in the unique exact equilibrium preservation and capturing small perturbations of the hydrostatic state well without numerical oscillations near the discontinuity. Moreover, our ALE-WENO hybrid schemes have an advantage over the simulations on static meshes due to the higher resolution interface tracking of the fluid motion.

\end{abstract}
\maketitle

\clearpage

\section{Introduction}\label{se:in}
This paper concentrates on discussing the shallow water equations with source terms owing to the non-flat bottom topography. Deriving from the incompressible Navier-Stokes equations, the shallow water equations have been widely used in coastal areas, reservoirs, estuary, and rivers, such as forecasting tsunamis, floods, atmospheric circulation, etc.
It has the following form in one-dimensional case
\begin{equation}\label{sha}
    \left\{\begin{array}{l}
    h_t+(hu)_{x}=0, \\
    (hu)_t+(hu^2+\dfrac{1}{2}gh^2)_x=-ghb_x,\\
    \end{array}\right.
\end{equation}
where $h,u,b$ stands for the water height, the velocity of the fluid, the bottom topography respectively and $g$ is the gravitational constant.\par

As a typical example of hyperbolic balance laws, many methods have been proposed consisting of finite difference schemes \cite{li2015well,xing2006highfd}, finite volume schemes \cite{klingenberg2019moving,noelle2006well} and discontinuous Galerkin (DG) methods \cite{xing2014exactly,vater2019limiter}.  Finite difference schemes have higher efficiency and less cost in multi-dimensional cases. Finite volume schemes and DG methods have an essential advantage in simulations on unstructured meshes without losing high order accuracy and conservation properties. This paper is concerned with finite volume schemes in the arbitrary Lagrangian-Eulerian (ALE) formulation.
The critical problem we may encounter in numerical simulation is the existence of strong shocks in hyperbolic problems. For this reason, the total variation diminishing (TVD) methods \cite{leveque1992numerical} were constructed. However, TVD schemes are at most second-order and will degenerate to first-order accuracy at critical points.
After this, there has been tremendous interest in developing high-order and high-resolution schemes. In 1987, Harten et al. \cite{harten1987uniformly} first designed
essentially non-oscillatory (ENO) schemes in the finite volume framework. In 1988, Shu and Osher \cite{shu1988efficient} proposed finite difference ENO schemes.
In 1994, the WENO reconstruction procedure, which adopted a convex combination of interpolations from sub-stencils, was first introduced by  Liu et al. \cite{liu1994weighted} to maintain the properties of the ENO schemes. 
Soon after, by proposing a new smoothness indicator to calculate the nonlinear weights, Jiang and Shu \cite{shu1998essentially} generalized WENO schemes to higher-order accuracy, which was called the classic WENO-JS scheme. Based on it, numerous improvements and modifications to the WENO procedure were designed.
For example, the WENO-M schemes \cite{henrick2005mapped} and the WENO-Z scheme \cite{castro2011high} were developed to improve accuracy at critical points of smooth solutions.
Without loss of robustness, Levy et al. \cite{levy1999central} proposed an efficient and straightforward central WENO scheme for solving hyperbolic systems of conservation laws that requires no characteristic decomposition. Recently, WENO hybrid scheme \cite{li2010hybrid} has been developed to save computational cost and improve the resolution considerably by using different discontinuity indicators to identify troubled elements. In \cite{wan2021new}, a new finite difference WENO hybrid scheme was proposed for hyperbolic conservation laws.
More recent developments of WENO schemes have been introduced in a review paper \cite{shu2020essentially}. \par

When the solution contains discontinuities or multidimensional problems, refined uniform meshes will be required to capture extremely subtle features. The computational cost is prohibitively huge. As a result, adaptive mesh refinement and deformable moving meshes have attracted considerable attention owing to higher efficiency under the same degree of freedom. Beisiegel et al. \cite{beisiegel2021metrics} introduced a physics-based indicator to determine the target refinement area for the shallow water equations, and the meshes are derived from the computed solution. This paper is devoted to the arbitrary Lagrangian-Eulerian (ALE) method, in which the mesh is no longer fixed and can move with a chosen mesh velocity that may not depend on fluid particles.
In \cite{dumbser2014arbitrary}, arbitrary high order derivatives (ADER) WENO finite volume method for hyperbolic conservation laws was presented with time-accurate local time stepping under the ALE framework. In \cite{klingenberg2017arbitrarya}, Klingenberg et al. developed the arbitrary Lagrangian-Eulerian discontinuous Galerkin (ALE-DG) method for one-dimensional conservation laws. However, it cannot extend to multi-dimensional cases directly due to the unsatisfactory of the geometric conservation law (GCL) property, and the modified strong stability preserving Runge-Kutta (SSP-RK) time discretization can be adopted in \cite{fu2019arbitrary}. For finite volume schemes, Guillard and  Farhat \cite{guillard2000significance} have proved that it is of significance to satisfy the discrete geometric conservation law (D-GCL) on moving mesh grids. For Hamilton-Jacobi equations,  arbitrary Lagrangian-Eulerian local discontinuous Galerkin method and finite difference WENO scheme can be applied, see \cite{klingenberg2017arbitraryh,li2019high,li2020moving}.
In \cite{zhang2021positivity}, the high order well-balanced ALE-DG method has been developed for the shallow water equations.\par

For solving the shallow water equations, we need our methods to be capable of structure-preserving. The first is the property of positivity preservation. Dealing with problems containing both wet and dry areas is inevitable, such as flood waves and dam breaks. Therefore the positivity preserving of water height is of necessity. Many positivity-preserving schemes have been designed in \cite{bermudez1994upwind,bokhove2005flooding,bryson2011well}, and in recent years the most commonly used is a simple linear scaling limiter proposed by Xing and Zhang \cite{xing2011high,xing2010positivity}. It works for finite volume schemes and DG methods in one and two dimensions with a suitable Courant Friedrichs Lewy (CFL) condition number. It can be generalized to unstructured meshes as well \cite{xing2013positivity}. Another important feature is maintaining hydrostatic state solutions, in which the non-zero flux gradient exactly balances the source term. The well-known hydrostatic state solutions, which are so-called moving water equilibrium in one space dimension, are given by
\begin{equation}\label{moving_water}
  m := hu = \text{constant}, \quad E := \dfrac{1}{2}u^2 + g(h + b) = \text{constant}.
\end{equation}
Here $m$, $E$ are the moving water equilibrium variables. When the velocity reduces to zero, we get a special case of the hydrostatic state (\ref{moving_water})
\begin{equation}\label{still_water}
  u = 0, \quad h + b = \text{constant},
\end{equation}
which is called still water equilibrium (referred to as a lake at rest as well).
In the past two decades, a great deal of mathematical effort, e.g., \cite{audusse2004fast,bermudez1994upwind,leveque1998balancing}, has been devoted to the study of constructing well-balanced schemes maintaining the well-balanced property which was first proposed by Bermudez and Vazquez \cite{bermudez1994upwind},
while they only have the first and second order accuracy at most. Recently, extensive numerical schemes have been developed to overcome the problem when obtaining the desired still water equilibrium solution, referring to the review paper \cite{xing2014survey}.
The critical approach in \cite{noelle2006well} is a new quadrature rule for the source term used to extend well-balanced schemes to any desired order of accuracy. A special source term splitting is used in \cite{xing2006highfd,xing2006highfv}.
In \cite{xing2006new} the numerical flux is modified to balance the source term.
Some researchers \cite{li2014high,li2015well} turn their attention to an alternative formulation called pre-balanced equations, which is derived from the hydrostatic state and equivalent to the classic form (\ref{sha}).
However, some limitations exist since the schemes developed for still water preserving can not be generalized for moving water equilibrium. Noelle et al. \cite{noelle2007high} developed high order well-balanced finite volume WENO schemes by defining a reference equilibrium state and an equilibrium limiter, together with the discretization of the source term. Several other methods have been designed for moving water equilibrium; see \cite{xing2014exactly,xing2011advantage} for further information.\par

The main target of this paper is to combine the ALE method with the finite volume WENO scheme to develop the structure-preserving finite volume WENO schemes on moving meshes for the shallow water equations.  The WENO hybrid reconstruction procedure \cite{wan2021new} generalized to moving meshes is used for our spatial discretization. Unlike general hybrid schemes,  the smoothness detector divides the stencils into smooth, transition, or non-smooth regions. A modified total variation diminishing Runge-Kutta (TVD-RK) time discretization method \cite{fu2019arbitrary} is adopted to obtain the geometric conservation law for the ALE scheme. We can also develop well-balanced WENO hybrid schemes based on hydrostatic construction and special source term treatment on moving meshes. The strictly theoretical analysis and the numerical tests verify that the proposed methods for moving meshes are positivity preservation, well balanced, high order accuracy, and essentially non-oscillatory. We highlight that the positivity-preserving limiter works well with both the non well-balanced ALE-WENO scheme and two well-balanced ALE-WENO hybrid schemes we developed on moving meshes.\par

An outline of this paper is as follows.
Section \ref{se:no} presents some notations about the grid movements and the  WENO hybrid reconstruction procedure on moving meshes.
In Section \ref{se:ale}, the finite volume WENO  scheme and the positivity-preserving property both for the semi-discrete and fully-discrete scheme are given under the arbitrary Lagrangian-Eulerian framework.
In Section \ref{se:wb}, by introducing the technique of the bottom topography reconstruction, we get the high order well-balanced finite volume ALE-WENO hybrid schemes with hydrostatic reconstruction and special source term treatment, which maintain the well-balanced property and positivity-preserving property as well.
In Section \ref{se:nu}, numerical results in different circumstances are illustrated to verify all the mathematical properties.
In the end, a summary of the main works is listed in the form of concluding remarks.

\section{Notations  and WENO reconstruction}\label{se:no}
We first consider the grid motion and the WENO hybrid reconstruction procedure to describe the specific schemes under the ALE framework.
\subsection{Grid definitions}
Assuming that the computational domain $\Omega $ in one-dimension is divided into $N$ cells, i.e.
$$\Omega = \bigcup \limits  _{j=1}^{N}[x_{j-\frac{1}{2}}^n, x_{j+\frac{1}{2}}^n] = \bigcup\limits _{j=1}^{N}[x_{j-\frac{1}{2}}^{n+1}, x_{j+\frac{1}{2}}^{n+1}],$$
where $\left\{x_{j+\frac{1}{2}}^n \right\}_{j=0}^N$ are grid points at time $t_n$  and  $\left\{x_{j+\frac{1}{2}}^{n+1} \right\}_{j=0}^N$ at time $t_{n+1}$.
We define the moving speed of the $j$-th grid point
\begin{equation}\label{w1D}
\begin{aligned}
w_{j+\frac{1}{2}}^n = \dfrac{x_{j+\frac{1}{2}}^{n+1}-x_{j+\frac{1}{2}}^{n}}{t_{n+1}-t_n},
 \end{aligned}
 \end{equation}
and the rays connecting $x_{j+\frac{1}{2}}^n$ and $x_{j+\frac{1}{2}}^{n+1}$
\begin{equation}\label{x1D}
\begin{aligned}
x_{j+\frac{1}{2}}(t) = x_{j+\frac{1}{2}}^n + w_{j+\frac{1}{2}}^n(t-t_n), \ \text{for all}\ t\in[t_n,t_{n+1}].
 \end{aligned}
 \end{equation}
Therefore for any  time-dependent interval
\begin{align*}
\mathcal T_j(t) = [x_{j-\frac{1}{2}}(t),x_{j+\frac{1}{2}}(t)],
 \end{align*}
and the length of a cell
\begin{align*}
\Delta_j(t) = x_{j+\frac{1}{2}}(t)-x_{j-\frac{1}{2}}(t),
 \end{align*}
we can define the global moving grid velocity $w:\Omega\times[0,T] \rightarrow R$
\begin{equation}\label{ww1D}
\begin{aligned}
w(x,t) = w_{j+\frac{1}{2}}\dfrac{x-x_{j-\frac{1}{2}}(t)}{\Delta_j(t)} + w_{j-\frac{1}{2}}\dfrac{x_{j+\frac{1}{2}}(t)-x}{\Delta_j(t)}.
 \end{aligned}
 \end{equation}
Moreover, we define the global cell length
\begin{align*}
\Delta x =\max\limits_{t\in[0,T]}\max\limits_{1\leq j\leq N} \Delta_j(t).
 \end{align*}
In order to ensure the regularity and rationality of the deformed mesh, we assume
\begin{itemize}
\setlength{\parskip}{0.0em}
  \item [(A1)] For all $j=1,\cdots,N$ and $t\in[t_n,t_{n+1}],$ we have $\Delta_j(t) > 0.$
  \item [(A2)] The global moving grid velocity $w(x,t)$ and its derivative $\partial_x w(x,t)$ are bounded in  $\Omega\times[0,T]$.
  \item [(A3)] There exists a constant $M>0$, independent of $\Delta x$, such that $$\Delta_j(t) \geq M \Delta x, \ \forall j=1,\cdots,N. $$
\end{itemize}

Analogous to one-dimensional case,  denoting two-dimensional computational domain  $\Omega$ as
$$\overline\Omega = \bigcup \limits  _{i=1}^{Nx}\bigcup \limits  _{j=1}^{Ny}\overline {\mathcal T_{ij}^n} = \bigcup \limits  _{i=1}^{Nx}\bigcup \limits  _{j=1}^{Ny}\overline{\mathcal T_{ij}^{n+1}},$$
where $\mathcal K^n = \bigcup\limits _{i=1}^{Nx}\bigcup\limits _{j=1}^{Ny} \mathcal T_{ij}^n$ are regular quadrilateral meshes which cover the convex polyhedron domain exactly at time $t_n$ and the same mesh topology for $\mathcal K^{n+1}$.
Let $\boldsymbol w_{ij}=(w_{ij}^x,w_{ij}^y)^T$  be the mesh moving speed of the vertice $(x_{ij},y_{ij})$.
They are defined in  the $x$ and $y$  direction  as follows:
\begin{equation}\label{w2D}
\begin{aligned}
&w_{ij}^{x,n}= \dfrac{x_{ij}^{n+1}-x_{ij}^{n}}{t_{n+1}-t_n}, \ x_{ij}(t) = x_{ij}^n + w_{ij}^{x,n}(t-t_n),\\
&w_{ij}^{y,n}=\dfrac{y_{ij}^{n+1}-y_{ij}^{n}}{t_{n+1}-t_n},\ y_{ij}(t) = y_{ij}^n + w_{ij}^{y,n}(t-t_n).
 \end{aligned}
 \end{equation}
We can define the time-dependent domain $\mathcal T_{ij}(t)$, the area of corresponding cell $\Delta_{ij}(t)$, and also the global grid velocity
\begin{equation}\label{wglobal2D}
\boldsymbol w(x,y,t)  = ( w^x(x,y,t), w^y(x,y,t))^T
\end{equation}
in a similar way.

\subsection{The WENO hybrid reconstruction}
\label{subse:weno}
In a finite volume scheme, the computational variables are cell averages. Our goal is to recover a high order accuracy point value of the computational variables $u$ located at node $x$ through the neighboring cell averages $\bar u_j$.
We denote the reconstruction procedure at time $t_n$ as
\begin{equation}\label{recon1D}
 u_j^n(x) = \mathcal R \left( x,\left\{ \bar u_i^n\right\}_{i\in S_j^n}\right),
\end{equation}
where $S_j^n$ is the big stencil namely the set of all index tuples of cells in the reconstruction.\par

In this paper, we use the WENO hybrid reconstruction procedure developed in \cite{wan2021new} for our spatial discretization. It is simple and effective to increase the spectral resolution and maintain the high order accuracy  and oscillatory-free performance  by adopting a smoothness detector. Here we extend the  reconstruction procedure to moving meshes. For simplicity of presentation, we only describe the fifth-order WENO hybrid reconstruction procedure for one-dimensional scalar conservation laws. Hybrid WENO approximations are given in the following way.\par
\textbf {Step 1.\ }
Based on the small stencils at time $t_n$
$$s_r^n(j) = \{\mathcal T_{j-r}^n, \cdots, \mathcal T_{j-r+k-1}^n\}, \ r=0,1,2,$$
we obtain the reconstructed polynomial of third order accuracy
$$u^n_r(x) = \sum \limits _{i=0}^2 c_{ri}^n(x) \bar u _{j-r+i}, \ r=0,1,2,$$
and then use the big stencil $S_j^n=\{\mathcal T_{j-2}^n, \mathcal T_{j-1}^n,\mathcal T_{j}^n,\mathcal T_{j+1}^n, \mathcal T_{j+2}^n\}$ to find the linear weights $d_r^n(x)$ such that
\begin{align*}
 & u^{n}(x)=\sum \limits _{r=0}^2 d_r^n(x) u^n_r(x) =  u(x,t_n)+O((\Delta x^{n})^5).
 \end{align*}
Here $\Delta x^{n}: =\max\limits_{1\leq j\leq N} \Delta_j^n$ is the cell length at time $t_n$.

\textbf {Step 2.\ }
We compute the smoothness indicators
\begin{align*}
\beta_r^n=\sum \limits _{l=1}^2 (\Delta x^{n})^{2l-1} \displaystyle{\int_{x_{j-{1}/{2}}^n}^{x_{j+{1}/{2}}^n} \  \left( \dfrac{\partial^l u_r^n(x)}{\partial x^l} \right)^2 \ dx}, \ r=0,1,2,
\end{align*}
and define the smoothness detector
\begin{equation}\label{omegaj}
 \gamma = \dfrac{\tau}{\beta^A+\varepsilon},\ \text{where} \  \tau = |\beta_0^n - \beta_2^n|,\ \beta^A = \dfrac{1}{3}(\beta_0 + \beta_1 + \beta_2),
\end{equation}
here $\varepsilon>0$ is a small constant in order to prevent the denominator from being zero and taken as $10^{-6}$ in our numerical tests unless otherwise specified.
It is the critical component to classify the domain into three parts: the smooth region if $\gamma\leq \frac{1}{2}$, the transition region if $\frac{1}{2} \leq \gamma \leq 1$ and the non-smooth region if $\gamma\geq 1$.

\textbf {Step 3.\ }
We use linear reconstruction for the smooth region to reduce the computational cost and WENO-Z  reconstruction \cite{castro2011high} for the non-smooth region. More importantly, a linear interpolation is applied to blend the linear and WENO-Z reconstruction in the transition region.
To be specific, denoting $\varphi^L,\varphi^N,\varphi^T$ to be the nonlinear weights in the smooth, non-smooth and transition region respectively and they are chosen by:\begin{equation}\label{omegajjj}
 \begin{aligned}
 & \varphi^L_r(x) = d_r^n(x) \\
  & \varphi^N_r(x) = \dfrac{\alpha_r^n(x)}{ \sum  _{s=0}^2\alpha_s ^n(x)},\ \alpha_r^n(x)=d_r^n(x)\left(1+\left(\dfrac{\tau}{\beta_r^n+\varepsilon}\right)^2\right),\\
   & \varphi^T_r(x)= M(\gamma)\ \varphi^N_r(x) + (1-M(\gamma))\ \varphi^L_r(x),
   \end{aligned}
 \end{equation}
where the weight function is taken as $M(\gamma) = 2\gamma-1$. Note that we only take the WENO-Z reconstruction as an example, it can be replaced by other reconstruction procedure.

\textbf {Step 4.\ }
The final fifth-order WENO hybrid reconstruction is written as a combination of all the reconstructed polynomials
\begin{equation}\label{f-}
\mathcal R \left( x,\left\{ \bar u_i^n\right\}_{i\in S_j^n}\right) =u_j^{n}(x)=\sum \limits _{r=0}^2\varphi_r^n(x) u^{n}_r(x).
\end{equation}

Notice that all the coefficients  $ c_{ri}^n$, $d_r^n$ and the smoothness indicators $\beta_r^n$ are calculated on non-uniform grids and change at any moment. Furthermore,  the linear weights on moving meshes or two-dimensional cases may become negative. Therefore a special treatment \cite{shi2002technique} is of necessity.

For  systems of  hyperbolic conservation laws, we prefer to use the local characteristic decomposition \cite{jiang1996efficient,roe1981approximate}, which is more robust and resulting less oscillations for strong discontinuities than a component-wise version. Additionally, for well-balanced schemes, $\alpha_r^n(x)$ in (\ref{omegajjj}) is modified by
\begin{equation}\label{alphamodified}
 \alpha_r^n(x) = d_r^n(x)\left(1+\left(\dfrac{\tau}{\sum\nolimits_{i=1}^{d}\beta_r^{(i)}+\varepsilon}\right)^2\right), \ r=0,1,2,
\end{equation}
in order to make sure that the coefficient $a_k^n(x)$ in (\ref{f-}) is the same for each of the conservation variable. Here $\beta_r^{(i)}$ is the $i$-th component of the smoothness indicator and $d$ is the number of  equations.

\subsection{The geometric conservation law}
The geometric conservation law (GCL) and the discrete geometric conservation law (D-GCL) play a significant role in the stability and high-order accuracy of numerical schemes in mesh deformation methods.
If a semi-discrete method has the  GCL property, it has the following differential identity
\begin{equation}\label{gclf}
     \dfrac{d}{dt}\displaystyle{\iint_{ \mathcal  T_{ij}(t)} \ dxdy} = \displaystyle{\iint_{ { \mathcal  T_{ij}(t)}} \nabla \cdot \boldsymbol w \ dxdy},
\end{equation}
where $\boldsymbol w$ is the global  grid velocity defined in (\ref{wglobal2D}).
In other words, if the initial condition of the equation is a constant, the numerical solution will be exactly equal to the constant within machine error. The D-GCL property which can  maintain constant states as well is aimed at fully-discrete schemes. More details can be referred in \cite{guillard2000significance}.
Klingenberg et al. \cite{klingenberg2017arbitrarya}  proved that the ALE-DG method satisfies the GCL and D-GCL property for one-dimensional conservation law under the same mesh velocity setup, and a modified SSP-RK method which satisfies the  D-GCL property was developed in \cite{fu2019arbitrary} for multi-dimensional cases. The D-GCL property is also essential for the well-balanced preservation of the ALE-DG schemes for the shallow water equations \cite{zhang2021positivity}. The same result holds for our ALE-WENO scheme since the piecewise constant ALE-DG scheme is precisely the first-order finite volume scheme. Thus, we omit the details here.

\section{The finite volume ALE-WENO  scheme}\label{se:ale}
In this section, we introduce the fifth-order finite volume ALE-WENO  scheme for the shallow water equations. Here we restrict ourselves to two dimensions.
\begin{equation}\label{2Dsha}
    \left\{\begin{array}{l}
    h_t+(hu)_{x}+(hv)_{y}=0, \\
    (hu)_t+(hu^2+\dfrac{1}{2}gh^2)_x+(huv)_y = -ghb_x,\\
    (hv)_t+(huv)_x+(hv^2+\dfrac{1}{2}gh^2)_y = -ghb_y.
    \end{array}\right.
 \end{equation}
Taking the GCL property into account, we follow the same  idea as \cite{arpaia2018r,zhang2021positivity} to transform conservation variables $\boldsymbol u=(h,hu,hv)^T$  into equilibrium variables $\boldsymbol q =(\eta,hu,hv)^T$ with $\eta = h+b$ for  better computation on moving meshes. Here the superscript $T$ denoting the transpose.
Following the notation above, we rewrite the shallow water equations  (\ref{2Dsha}) as follows:
 \begin{equation}\label{resha2D}
\begin{aligned}
\boldsymbol q_t + \boldsymbol f_1(\boldsymbol q)_x + \boldsymbol f_2(\boldsymbol q)_y = \boldsymbol r(\boldsymbol q,b),
 \end{aligned}
 \end{equation}
where
$$\boldsymbol f_1(\boldsymbol q) = \left[\begin{array}{c}
                 hu    \\
                 \frac{(hu)^2}{\eta-b}+\frac{1}{2}g(\eta-b)^2\\
                 \frac{(hu)(hv)}{\eta-b}
                 \end{array} \right],
\boldsymbol f_2(\boldsymbol q) = \left[\begin{array}{c}
                 hv    \\
                 \frac{(hu)(hv)}{\eta-b}\\
                 \frac{(hv)^2}{\eta-b}+\frac{1}{2}g(\eta-b)^2
                 \end{array} \right],
\boldsymbol r(\boldsymbol q,b) =\left[\begin{array}{c}
                 0    \\
                 -g(\eta-b)b_x\\
                 -g(\eta-b)b_y
                 \end{array} \right]. $$
\subsection{The semi-discrete ALE-WENO  scheme}
For any time-dependent  quadrilateral cells,  a bilinear mapping can be defined to transform it into a time-independent reference cell
\begin{equation}\label{mappingbilinear}
  \mathcal{X}_{\mathcal T(t)}: \widetilde{\mathcal T}_{\text{ref}}  \rightarrow \overline{\mathcal T(t)}, \quad  (\xi,\zeta) \rightarrow  (x(t),y(t)) = \mathcal{X}_{\mathcal T(t)}(\xi,\zeta,t).
\end{equation}
Denoting $\mathcal{A}_{\mathcal T(t)}$ to be the Jacobian matrix of the mapping $\mathcal{X}_{\mathcal T(t)}$ and
$$ \mathcal{J}_{\mathcal T(t)} = |\text{det}(\mathcal{A}_{\mathcal T(t)})|\geq 0  $$
to be the absolute value of the corresponding determinant. We notice that  $\mathcal{A}_{\mathcal T(t)}$ and $\mathcal{J}_{\mathcal T(t)}$ here  depend on spatial and temporal variables. Some tedious manipulation yields
\begin{equation}\label{c0}
  c_0\Delta_{ij}(t) \leq\mathcal{J}_{\mathcal T_{ij}(t)}\leq c_1\Delta_{ij}(t),
\end{equation}
where the constants $c_0,c_1$ depend on the deformation of mesh grids. Moreover, the bilinear mapping gives a characterization of the grid velocity
\begin{equation}\label{velcha}
\boldsymbol w(x,y,t) = \partial_t\left(\mathcal{X}_{\mathcal T(t)}(\xi,\zeta,t)\right).
\end{equation}

Based on it, we present a  derivation of our semi-discrete ALE-WENO scheme briefly. Let $\widetilde { \boldsymbol q} =  \boldsymbol q\circ \mathcal{X}_{\mathcal T(t)}$ be the variables in the reference element.
The chain rule formula provides
\begin{equation}\label{dustar}
 \dfrac{\partial} {\partial t}\widetilde { \boldsymbol q} =\dfrac{\partial} {\partial t}{ \boldsymbol q} + {\boldsymbol w} \cdot \nabla  \boldsymbol q,
\end{equation}
and from the  GCL identity (\ref{gclf}) we have
\begin{equation}\label{djaco}
  \dfrac{\partial}{\partial t}\mathcal{J}_{\mathcal T(t)} = \mathcal{J}_{\mathcal T(t)} \nabla \cdot  {\boldsymbol w}.
\end{equation}
Summing (\ref{dustar}) which pre-multiplied by  $\mathcal{J}_{\mathcal T(t)}$  to (\ref{djaco}) pre-multiplied by  $\widetilde { \boldsymbol q} $ gives
\begin{equation}\label{differential}
  \dfrac{\partial}{\partial t}\left(\mathcal{J}_{\mathcal T(t)} \widetilde { \boldsymbol q}\right) = \mathcal{J}_{\mathcal T(t)}\boldsymbol r(\boldsymbol q,b) -   \mathcal{J}_{\mathcal T(t)} \nabla \cdot \boldsymbol {\mathcal H}(\boldsymbol w,\boldsymbol q),
\end{equation}
where the flux term $$\boldsymbol {\mathcal H}(\boldsymbol w,\boldsymbol q) = \boldsymbol F(\boldsymbol q) - \boldsymbol w\boldsymbol q = \left(\boldsymbol f_1(\boldsymbol q)-  w^x \boldsymbol q, \boldsymbol f_2(\boldsymbol q)-  w^y \boldsymbol q\right).$$
Integrating (\ref{differential})  over a reference element $\widetilde{\mathcal T}_{\text{ref}}$ leads to
\begin{align*}
 \displaystyle{\iint_ {\widetilde{\mathcal T}_{\text{ref}}} \dfrac{\partial}{\partial t}\left(\mathcal{J}_{\mathcal T(t)} \widetilde { \boldsymbol q}\right) \ d\xi d\zeta } = -   \displaystyle{\iint_ {\widetilde{\mathcal T}_{\text{ref}}} \nabla \cdot \boldsymbol {\mathcal H}(\boldsymbol w,\boldsymbol q) \mathcal{J}_{\mathcal T(t)}  \ d\xi d\zeta } + \displaystyle{\iint_ {\widetilde{\mathcal T}_{\text{ref}}}\boldsymbol r(\boldsymbol q,b)\mathcal{J}_{\mathcal T(t)}  \ d\xi d\zeta },
\end{align*}
which can be eventually transformed into the  following ALE integral formulation
\begin{equation}\label{integralphy}
 \dfrac{d}{dt}\displaystyle{\iint_ {{\mathcal T_{ij}(t)}}{ \boldsymbol q} \ dxdy } = -   \displaystyle{\iint_ {{\mathcal T_{ij}(t)}} \nabla \cdot \boldsymbol {\mathcal H}(\boldsymbol w,\boldsymbol q)  \ dxdy } + \displaystyle{\iint_ {{\mathcal T_{ij}(t)}} \boldsymbol r(\boldsymbol q,b)  \ dxdy }.
\end{equation}
Denoting $ \overline{\boldsymbol  q}_{ij}(t) $ to be the approximation of the cell average of the function ${\boldsymbol q}(\cdot,\cdot,t)$ over the time-dependent domain $\mathcal T_{ij}(t)$, that is
\begin{equation}\label{cellave2D}
\overline {\boldsymbol q}_{ij}(t) \approx \overline {\boldsymbol q}(x_i,y_j,t)= \dfrac{1}{\Delta_{ij}(t)}\displaystyle{\iint_{\mathcal T_{ij}(t)}  \boldsymbol q(x,y,t)  \ dxdy}.
\end{equation}
For  each quadrilateral cells $\mathcal T_{ij}(t)$, the boundary can be expressed in terms of
$$\partial\mathcal T_{ij}(t) = \bigcup\nolimits  _{l=1}^{4} {e_{ij}^{l}(t)}.$$
Hence we have the following conservative semi-discrete finite volume ALE-WENO  scheme:
\begin{equation}\label{se2D}
\begin{aligned}
  &\dfrac{d}{dt} \left(\Delta_{ij}(t) \overline{\boldsymbol  q}_{ij}(t) \right)  = \text{RHS}_{ij}(\boldsymbol w,\overline {\boldsymbol q},t)
\end{aligned}
\end{equation}
with
\begin{equation}\label{s2D}
\begin{aligned}
   &\text{RHS}_{ij}(\boldsymbol w,\overline {\boldsymbol q},t) = -\sum _{l=1}^4 \displaystyle{\int_{e_{ij}^{l}(t)} \widehat {\boldsymbol {\mathcal H}}(\boldsymbol w,\boldsymbol q^{-},\boldsymbol q^{+}, \boldsymbol n_l) \ ds } + \displaystyle{\iint_{ \mathcal  T_{ij}(t)} \boldsymbol r(\boldsymbol q,b)  \ dxdy},
\end{aligned}
\end{equation}
where ${\boldsymbol q}^{-}$ is the reconstructed value obtained by the high order WENO hybrid reconstruction procedure in Section \ref{subse:weno} on the cell $\mathcal T_{ij}(t)$ at a gauss point
and ${\boldsymbol  q}^{+}$ is acquired at the same point on the cell adjacent to the associated edge,  $\boldsymbol n_l$ is the outer normal vector. Considering  the stability of the scheme, we choose the simplest monotone Lax-Friedrichs numerical flux for  $\widehat {\boldsymbol {\mathcal H}}(\boldsymbol w,\boldsymbol q^{-},\boldsymbol q^{+}, \boldsymbol n_l)$, that is
\begin{equation}\label{alpha2D}
\begin{aligned}
  \widehat {\boldsymbol {\mathcal H}}(\boldsymbol w,\boldsymbol q^{-},\boldsymbol q^{+}, \boldsymbol n_l) &= \dfrac{1}{2}\left(\boldsymbol  {\mathcal H}(\boldsymbol w,\boldsymbol q^-)\cdot \boldsymbol n_l+ \boldsymbol {\mathcal H}(\boldsymbol w,\boldsymbol q^+)\cdot \boldsymbol n_l \right) -\dfrac{\alpha}{2}(\boldsymbol q^+ - \boldsymbol q^-),\\
  \alpha &= \max\limits_{u,v,l} \left(\left|(u-w^x,v-w^y)^T \cdot\boldsymbol n_l \right|+\sqrt{gh}\right),
\end{aligned}
\end{equation}
the maximum is taken globally in our computation.

\subsection{The fully-discrete ALE-WENO  scheme}
For hyperbolic conservation law, we adopt high order SSP-RK and multistep time discretization method \cite{gottlieb2001strong, shu1988total} to resolve the semi-discrete scheme. Herein for the satisfaction of the D-GCL property  in  multi-dimensional situation, we adopt the  modified third-order SSP-RK time discretization in  \cite{fu2019arbitrary} for  (\ref{se2D}):
 \begin{equation}\label{RK32D}
  \begin{aligned}
  & J_{\mathcal{T}_{ij}^{n,1}} \overline {\boldsymbol q}_{ij}^{n,1} = J_{\mathcal{T}_{ij}^{n}}\left( \overline {\boldsymbol q}_{ij}^{n} +  \dfrac{1}{\Delta_{ij}^{n}}\Delta t \ \text{RHS}_{ij}(\boldsymbol w^n,\overline {\boldsymbol q}^n,t_n)\right),\\
  & J_{\mathcal{T}_{ij}^{n,2}} \overline {\boldsymbol q}_{ij}^{n,2} =\dfrac{3}{4}  J_{\mathcal{T}_{ij}^{n}} \overline {\boldsymbol q}_{ij}^{n} +\dfrac{1}{4}  J_{\mathcal{T}_{ij}^{n,1}} \left( \overline {\boldsymbol q}_{ij}^{n,1} +  \dfrac{1}{\Delta_{ij}^{n,1}}\Delta t \ \text{RHS}_{ij}(\boldsymbol w^{n,1},\overline {\boldsymbol q}^{n,1},t_{n,1})\right),\\
  & J_{\mathcal{T}_{ij}^{n+1}} \overline {\boldsymbol q}_{ij}^{n+1} =\dfrac{1}{3}  J_{\mathcal{T}_{ij}^{n}} \overline {\boldsymbol q}_{ij}^{n} +\dfrac{2}{3}  J_{\mathcal{T}_{ij}^{n,2}} \left( \overline {\boldsymbol q}_{ij}^{n,2} +  \dfrac{1}{\Delta_{ij}^{n,2}}\Delta t \ \text{RHS}_{ij}(\boldsymbol w^{n,2},\overline {\boldsymbol q}^{n,2},t_{n,2})\right),
   \end{aligned}
 \end{equation}
where
  \begin{equation}\label{Jacobi2D}
  \begin{aligned}
  & J_{\mathcal{T}_{ij}^{n}} = \mathcal{J}_{\mathcal{T}_{ij}(t_{n})}, \\
  & J_{\mathcal{T}_{ij}^{n,1}} = J_{\mathcal{T}_{ij}^{n}}+\Delta t \ \nabla\cdot {\boldsymbol w^n}J_{\mathcal{T}_{ij}^{n}},\\
  & J_{\mathcal{T}_{ij}^{n,2}} = \dfrac{3}{4}  J_{\mathcal{T}_{ij}^{n}} + \dfrac{1}{4}J_{\mathcal{T}_{ij}^{n,1}}+\dfrac{1}{4}\Delta t\ \nabla\cdot \boldsymbol w^{n,1}J_{\mathcal{T}_{ij}^{n,1}},\\
  & J_{\mathcal{T}_{ij}^{n+1}} = \dfrac{1}{3}  J_{\mathcal{T}_{ij}^{n}} + \dfrac{2}{3}J_{\mathcal{T}_{ij}^{n,2}}+\dfrac{2}{3}\Delta t \ \nabla\cdot \boldsymbol w^{n,2} J_{\mathcal{T}_{ij}^{n,2}},
   \end{aligned}
 \end{equation}
and
\begin{align*}
 &t_{n,1} = t_{n+1}, t_{n,2} = \dfrac{1}{2}(t_{n}+t_{n+1}), \quad  \boldsymbol w^{n,1} = \boldsymbol w({t_{n,1}}), \boldsymbol w^{n,2} = \boldsymbol w({t_{n,2}}), \\
 &\mathcal{T}_{ij}^{n} = \mathcal{ T}_{ij}(t_{n}), \mathcal{ T}_{ij}^{n,1} = \mathcal{ T}_{ij}^{n+1} =\mathcal{ T}_{ij}(t_{n+1}) , \mathcal{ T}_{ij}^{n,2} = \mathcal{ T}_{ij}\left( \dfrac{t_n+t_{n+1}}{2}\right).
\end{align*}

\subsection{The approximation of the bottom topography}
There are two different choices for the approximation of the bottom topography. The first one is the $L^2$ projection of the exact bottom function $b(x,y)$, this approach requires us to calculate an approximate cell average  $\overline b_{ij}(t)$ on each time-dependent interval $\mathcal{T}_{ij}(t)$, i.e.
\begin{equation}\label{baverage}
  \overline b_{ij}(t)\approx    \dfrac{1}{\Delta_{ij}(t)}\displaystyle{\iint_{\mathcal{T}_{ij}(t)}   b(x,y)  \ dxdy}.
\end{equation}
It can be seen that we have to select a sufficiently accurate numerical quadrature rule.\par

An alternative way is shown in \cite{zhang2021positivity}, similarly, we use the ALE-WENO scheme to approximate the bottom topography on moving meshes. Making use of the equation
$b_t = 0$, we can get the conservation scheme for the cell average $\overline b_{ij}(t)$:
\begin{equation}\label{bottomrk}
  \begin{aligned}
  \dfrac{d}{dt}\left( \Delta_{ij}(t) \overline b_{ij}(t)\right) &= \sum _{l=1}^4 \displaystyle{\int_{e_{ij}^{l}(t)} \widehat { {\mathcal B}}(\boldsymbol w, b^{-},b^{+}, \boldsymbol n_l)  \ ds },
\end{aligned}
\end{equation}
where
$$\widehat {\mathcal B}(\boldsymbol w,b^{-},b^{+}, {\boldsymbol n_l}) = \dfrac{1}{2}\boldsymbol w \cdot  {\boldsymbol n_l}(b^{-}+b^{+}) + \dfrac{1}{2}\alpha (b^{+}-b^{-})$$
is the numerical flux with the same  $\alpha$   in (\ref{alpha2D}). The values $b^{\pm}$ are acquired by the WENO hybrid reconstruction  procedure in (\ref{recon1D}). In addition, the likewise modified third-order SSP-RK method (\ref{RK32D}) can be  used for the semi-discrete scheme (\ref{bottomrk}).\par
Comparing the two approaches above, although the first one is more accurate than the second one theoretically, it may affect the exact mass conservation and the positivity preserving property on moving meshes, especially when encountering discontinuous bottom topography.  We will use the second one to approximate the bottom topography unless otherwise specified in the following numerical examples.

\subsection{Positivity-preserving property}
The numerical schemes will produce a negative water height if we do not take special treatment at the wet and dry interfaces. The resulting system is no longer hyperbolic, and the initial value problem is ill-posed. Therefore in this subsection, we will discuss the positivity-preserving property of the ALE-WENO scheme for two-dimensional shallow water equations.

\subsubsection{Positivity-preserving property }
Assuming that $\eta_{ij}^n(x,y), b_{ij}^n(x,y)$ are the fifth-order WENO hybrid reconstruction polynomials defined in domain $\mathcal T_{ij}^n$. We denote $h_{ij}=\eta_{ij} - b_{ij}$   for simplicity. Let
$$\widetilde h_{ij}^n(\xi,\zeta) = h_{ij}^n \circ \mathcal{X}_{\mathcal T(t)}$$
be a polynomial defined in  the reference element $\widetilde{\mathcal T}_{\text{ref}} $  (\ref{mappingbilinear}) which is chosen as a unit square.
In our computation, $p$-th order accuracy  Gauss-Lobatto quadrature rule  with $k$ points is used. Here $k$ is chosen to satisfy $k\geq {(p+2)}/{2}$.
We distinguish the Gauss-Lobatto quadrature nodes for the interval $[0,1]$ in the $\xi$ and $\zeta$ direction by denoting
$$ p_i^\xi = \left\{ \xi_{i_\beta}:\beta = 1,\cdots,k, \xi_{i_1}=0, \xi_{i_k}=1 \right\}, \quad p_j^\zeta = \left\{ \zeta_{j_\nu}:\nu = 1,\cdots,k, \zeta_{j_1}=0, \zeta_{j_k}=1 \right\}, $$
and the corresponding weights $\sigma_\beta$ and $ \sigma_\nu$  which satisfies $\sum_{\beta=1}^k  \sigma_\beta = 1, \sum_{\nu=1}^k  \sigma_\nu = 1.$
Let the variables
\begin{equation}\label{deltajn}
\begin{aligned}
\delta_{ij}^1& =\dfrac{1}{\sum_{\nu=2}^{k-1}  \sigma_\nu}\left( \dfrac{1}{\Delta_{ij}^n} \sum\limits_{\nu=2}^{k-1}  \sigma_\nu \left( \sum\limits_{\beta=1}^k  \sigma_\beta \widetilde h^-_{i_\beta,j_\nu} \mathcal{J}(\xi_{i_\beta},\zeta_{j_\nu}) \right) \right),\\
\delta_{ij}^2 & =  \dfrac{1}{\sum_{\beta=2}^{k-1} \sigma_\beta } \left(\dfrac{1}{\Delta_{ij}^n} \sum\limits_{\beta=2}^{k-1}  \sigma_\beta \left( \sum\limits_{\nu=1}^k  \sigma_\nu \widetilde h^-_{i_\beta,j_\nu} \mathcal{J}(\xi_{i_\beta},\zeta_{j_\nu}) \right)\right).
 \end{aligned}
\end{equation}

All the quadrature and the positivity-preserving limiter will be performed in a time-independent reference cell. We only show the positivity preservation property for the modified  Euler forward time discretization. The validity of the proposition with the modified third-order SSP-RK method can be acquired similarly.
\begin{property}[\textbf{Cell average positivity preserving}]\label{propp1}
We consider the  semi-discrete ALE-WENO  scheme (\ref{se2D}), (\ref{bottomrk}) satisfied by the cell average of the surface level $\eta$ and the bottom topography $b$, combined with the modified Euler Forward time discretization, see the first equation in (\ref{RK32D}). Let $\delta_{ij}^1,\delta_{ij}^2$ be defined in (\ref{deltajn}), if $\widetilde h^-_{i_\beta,j_1}$, $\widetilde h^-_{i_\beta,j_k}$, $\widetilde h^-_{i_1,j_\nu}$, $\widetilde h^-_{i_k,j_\nu}$, $\delta_{ij}^1,\delta_{ij}^2$ are all non-negative for all the $\beta,\nu,i,j$ at time $t_n$, then the cell average of $\bar \eta_{ij}^{n+1}-\bar b_{ij}^{n+1}$ is  non-negative under the CFL condition
\begin{equation}\label{CFLcondition}
  \max\left\{ \sigma_1 \dfrac{J_{\mathcal{T}_{ij}(t)}}{\Delta _{ij}(t)} |\nabla\cdot {\boldsymbol w(t)}| +  \dfrac{1}{\Delta _{ij}(t)}\sum_{l=1}^4 |e_{ij}^l(t)|\alpha: t\in [t_n,t_{n+1}]\right\} \Delta t \leq c_0 \sigma_1,
\end{equation}
where $\alpha$ is chosen the same as (\ref{alpha2D}) and  $c_0$ given in (\ref{c0}).
\end{property}
\begin{proof}
The proof of Proposition \ref{propp1} is shown in  Appendix \ref{a1}.
\end{proof}

\subsubsection{Bound preserving limiter}
Assuming $\bar h_{ij}^n  = \bar \eta_{ij}^{n}-\bar b_{ij}^{n}$ at time $t_n$ is nonnegative, then the modified polynomial defined in reference domain is given in the following way:
\begin{equation}\label{modifiedhjn}
  {\widetilde h}_{ij}^{n,\text{new}}(\xi,\zeta) = \Phi (\widetilde h_{ij}^n(\xi,\zeta)-\bar h_{ij}^n)+ \bar h_{ij}^n, \quad \Phi = \min\left\{1,\dfrac{\bar h_{ij}^n-\epsilon}{\bar h_{ij}^n-m_{ij}^n}\right\},
\end{equation}
where
 $$ m_{ij}^n =\min_{(\xi,\zeta)\in p_{ij}}\widetilde h_{ij}^n(\xi,\zeta) , \ p_{ij} = \left\{(\xi,\zeta): \xi \in p_{i}^\xi, \zeta \in p_{j}^\zeta \right\},\  \epsilon =\min\left\{10^{-11},\bar h_{ij}^n\right\}.$$
Noticing that in finite volume WENO scheme, if all the values at the Gauss quadrature points $\widetilde h_{i_\beta,j_\alpha}^- $ are obtained by the WENO hybrid reconstruction procedure, the cost is very high. An alternative approach proposed by \cite{xing2011high}  is  separating the limiter (\ref{modifiedhjn}) into two direction, that means in $\xi$-direction:
\begin{equation}\label{xd}
\begin{aligned}
  & {\widetilde h}_{i_\beta,j_1}^{-,\text{new}} = \Phi (\widetilde h_{i_\beta,j_1}^- -\bar h_{ij}^n)+ \bar h_{ij}^n,  \ {\widetilde h}_{i_\beta,j_k}^{-,\text{new}} = \Phi (\widetilde h_{i_\beta,j_k}^--\bar h_{ij}^n)+ \bar h_{ij}^n, \\
  & \text{with} \quad m_{ij} =\min\left\{ \widetilde h_{i_\beta,j_1}^- ,\delta_{ij}^1,\widetilde h_{i_\beta,j_k}^- \right\},
\end{aligned}
\end{equation}
and in $\zeta$-direction:
\begin{equation}\label{yd}
\begin{aligned}
 & {\widetilde h}_{i_1,j_\nu}^{-, \text{new}} = \Phi (\widetilde h_{i_1,j_\nu}^- -\bar h_{ij}^n)+ \bar h_{ij}^n,  \  {\widetilde h}_{i_k,j_\nu}^{-,\text{new}} = \Phi (\widetilde  h_{i_k,j_\nu}^--\bar h_{ij}^n)+ \bar h_{ij}^n, \\
  &\text{with} \quad m_{ij} =\min\left\{ \widetilde h_{i_1,j_\nu}^- ,\delta_{ij}^2,\widetilde h_{i_k,j_\nu}^- \right\}.
  \end{aligned}
\end{equation}
Similar to the bound-preserving limiter for hyperbolic conservation laws, (\ref{xd}),(\ref{yd}) can enforce the conditions of Proposition \ref{propp1}  without destroying accuracy and keeping the water height non-negative.

\section{The well-balanced finite volume ALE-WENO hybrid scheme}\label{se:wb}

In this section, we present the well-balanced finite volume ALE-WENO hybrid scheme which can possess the well-balanced property and the positivity-preserving property as well. The hydrostatic equilibrium state we are interested in is
\begin{equation}\label{still_water2D}
  u =v =  0, \quad \eta = \text{constant}.
\end{equation}

\subsection{The scheme with hydrostatic reconstruction}
\label{subse:hy}
Since the same modified third-order SSP-RK time discretization  (\ref{RK32D}) can be used for the well-balanced fully-discrete ALE-WENO hybrid scheme, we only give the well-balanced  semi-discrete ALE-WENO hybrid scheme. Before  proceeding further, we introduce a technique  on the reconstruction of the bottom topography $b$ briefly.
We adopt the similar approach used in \cite{xing2006highfv} with slightly modification.
\begin{itemize}
\setlength{\parskip}{0.0em}
  \item [1.] Different from the reconstruction values of the  ALE-WENO scheme in Section \ref{se:ale}, here we apply the WENO hybrid reconstruction procedure (\ref{f-}) to the conservation variables  $ {\boldsymbol  u}$ instead of the equilibrium variables $ {\boldsymbol q}$, we have
       \begin{equation}\label{urec}
          \boldsymbol u^{-} = \mathcal{R}^{-}(\bar {\boldsymbol u}), \quad
          \boldsymbol u^{+} = \mathcal{R}^{+}(\bar {\boldsymbol u}),
       \end{equation}
      Noticing that WENO nonlinear weights in (\ref{omegajjj}) depend on the cell average  $\bar{\boldsymbol u}$ nonlinearly.
  \item [2.] We use the same nonlinear weights on $\boldsymbol  b= (b,0,0)^T$ to reconstruct $ \boldsymbol  b^{\pm}$   and then $ \boldsymbol  q^{\pm}=\boldsymbol  u^{\pm} +\boldsymbol  b^{\pm}$,  from which we can easily arrive at 
      \begin{equation}\label{etahubar}
         \eta^+=\eta^-=\text{constant}:=c,\quad  (hu)^{\pm}=(hv)^{\pm}=0.
      \end{equation}
   It is a crucial technique for our construction of well-balanced schemes.
\end{itemize}

The key idea of the well-balanced semi-discrete ALE-WENO hybrid scheme with hydrostatic reconstruction  is a special treatment of the numerical flux. Following the procedure in \cite{xing2006new}, we have
\begin{equation}\label{hydro2D}
\dfrac{d}{dt}\left( \Delta_{ij}(t)\overline {\boldsymbol q}_{ij}(t) \right)
  =  -\sum _{l=1}^4 \displaystyle{\int_{e_{ij}^{l}(t)} \widehat {\boldsymbol {\mathcal H}}^{\text{mod}}(\boldsymbol w,\boldsymbol q^{-},\boldsymbol q^{+}, \boldsymbol n_l)\ ds } + \displaystyle{\iint_{\mathcal T_{ij}(t)} \boldsymbol r(\boldsymbol q,b)  \ dxdy},
\end{equation}
where
\begin{equation}\label{H2D}
\widehat {\boldsymbol {\mathcal H}}^{\text{mod}} = \widehat {\boldsymbol {\mathcal H}}(\boldsymbol w,\boldsymbol q^{*,-},\boldsymbol q^{*,+},\boldsymbol n_l) +\left(\begin{array}{c}
                 \vartheta    \\
                 \left(\dfrac{1}{2}g(\eta^{-} - b^-)^2 - \dfrac{1}{2}g(h^{*,-})^2\right) n_{l_1}\\
                 \left(\dfrac{1}{2}g(\eta^{-} - b^-)^2 - \dfrac{1}{2}g(h^{*,-})^2\right) n_{l_2}
                 \end{array} \right)
\end{equation}
with
\begin{equation}\label{HH2D}
\begin{aligned}
  & \boldsymbol n_l = (n_{l_1},n_{l_2}), \quad  \boldsymbol q^{*,\pm} = \left(\eta^{*,\pm},  hu^{\pm}, hv^{\pm} \right)^T, \\
  & \eta^{*,\pm} =  h^{*,\pm} + b^{\pm}, \ h^{*,\pm} = \max (0, \eta^{\pm} - \max(b^{+}, b^{-})),\\
  & \vartheta = \dfrac{\boldsymbol w\cdot \boldsymbol n_l +\alpha}{2}(\eta^{*,+}-\eta^+)+  \dfrac{\boldsymbol w\cdot \boldsymbol n_l-\alpha}{2}(\eta^{*,-}-\eta^-).
\end{aligned}
\end{equation}
To maintain the hydrostatic equilibrium state for the one-dimensional scheme, the source term's approximation is acquired by interpolating fourth-order polynomials $\eta_h(x)$ and $b_h(x)$ on the cell $\mathcal T_j^n$. Then an accurate enough integration quadrature is adopted, referring \cite{xing2006new} for more details. The two-dimensional case is similar.

\subsection{The scheme with special source term treatment}
The well-balanced semi-discrete ALE-WENO hybrid scheme with a special decomposition of the source term is given by
\begin{equation}\label{special2D}
 \dfrac{d}{dt}\left( \Delta_{ij}(t)\overline {\boldsymbol q}_{ij}(t) \right)
  = -\sum _{l=1}^4 \displaystyle{\int_{e_{ij}^{l}(t)} \widehat {\boldsymbol {\mathcal H}}(\boldsymbol w,\boldsymbol q^{-},\boldsymbol q^{+}, \boldsymbol n_l) \ ds } +  { \boldsymbol r}_{ij},
\end{equation}
where the modified source term is defined as follows:
\begin{equation}\label{source2}
\begin{aligned}
& {\boldsymbol r}_{ij}^{(2)} = \dfrac{g}{2}\sum _{l=1}^4 \displaystyle{\int_{e_{ij}^{l}(t)} (\widehat {b^2} ) n_{l_1}  \ ds } - g \bar \eta_{ij}\sum _{l=1}^4 \displaystyle{\int_{e_{ij}^{l}(t)} \widehat {b}  n_{l_1}  \ ds } + r_{ij}^{(2)},\\
& {\boldsymbol r}_{ij}^{(3)} = \dfrac{g}{2}\sum _{l=1}^4 \displaystyle{\int_{e_{ij}^{l}(t)} (\widehat {b^2} ) n_{l_2}  \ ds } - g \bar \eta_{ij}\sum _{l=1}^4 \displaystyle{\int_{e_{ij}^{l}(t)} \widehat {b}  n_{l_2}  \ ds } + r_{ij}^{(3)},\\
\end{aligned}
\end{equation}
with
\begin{equation}\label{bhat2D}
\begin{aligned}
 & (\widehat {b^2}) = \dfrac{1}{2}\left( (b^-)^2 + (b^+)^2\right), \quad \widehat {b} =  \dfrac{1}{2}\left( b^- + b^+\right).
\end{aligned}
\end{equation}
Here ${\boldsymbol q}^{\pm}$, $b^{\pm}$ are obtained the same as Section \ref{subse:hy} and  $r_{ij}^{(2)},r_{ij}^{(3)}$ are the high order approximation to the integral
$$r_{ij}^{(2)} = \displaystyle{\iint_{\mathcal T_{ij}(t)} -g(\eta - \bar \eta_{ij})b_x  \ dxdy},\  r_{ij}^{(3)} = \displaystyle{\iint_{\mathcal T_{ij}(t)} -g(\eta - \bar \eta_{ij})b_y  \ dxdy}.$$
Taking $r_{ij}^{(2)}$ as an example, by choosing the Gauss-Lobatto points $( x_{i_\mu}, y_{j_\upsilon}) $ and the associated weights $\omega_{\mu\upsilon}$ in the quadrilateral cell $\mathcal T_{ij}$, we have\begin{equation}\label{gausssij2}
  r_{ij}^{(2)} = -\sum\limits _{\mu,\upsilon}\omega_{\mu\upsilon}g(\eta_{i_\mu,j_\upsilon} -  \bar \eta_{ij})b_x( x_{i_\mu},y_{j_\upsilon}).
\end{equation}
For the sake of the well-balanced property,  those values at the Gauss points are obtained by the WENO hybrid reconstruction procedure (\ref{recon1D}). We can deal with $r_{ij}^{(3)}$ in the same way.

\subsection{Well-balanced and positivity-preserving property}
We show that both the well-balanced semi-discrete and the fully-discrete finite volume ALE-WENO hybrid schemes have the well-balanced property, and analogous to the finite volume ALE-WENO  scheme,  they have the positivity-preserving property as well.

\subsubsection{Well-balanced property}
\begin{lemma}
The finite volume semi-discrete ALE-WENO hybrid scheme with hydrostatic reconstruction (\ref{hydro2D}-\ref{HH2D}) and the scheme  with special source term treatment (\ref{special2D}-\ref{gausssij2}) are all exact for the hydrostatic equilibrium state  (\ref{still_water2D})
and possess high order accuracy. What's more,
we have\begin{equation}\label{prowell3}
\dfrac{d}{dt}\left( \Delta_{ij}(t)\overline {\boldsymbol q}_{ij}(t) \right)
   =\displaystyle{\iint_{\mathcal T_{ij}(t)}  \nabla \cdot(\boldsymbol w\otimes \boldsymbol q )\  dxdy},
 \end{equation}
 here
 $\boldsymbol w\otimes \boldsymbol q = (w^x  \boldsymbol q \quad w^y \boldsymbol q )$ and $w^x$, $w^y$ are defined in (\ref{wglobal2D}).
\end{lemma}

\begin{proof}{\textbf{(for the scheme with hydrostatic reconstruction)}}
It is straightforward to show the high order accuracy of the fifth-order WENO hybrid reconstruction procedure. Next, we will focus on proving (\ref{prowell3}).
According to the fact that $\boldsymbol  q$ is in hydrostatic state and the bottom topography reconstruction technique
we have (\ref{etahubar}).
Here the subscript is omitted. We can calculate
\begin{align*}
\widehat {\boldsymbol {\mathcal H}}^{\text{mod},(1)} & = \widehat {\boldsymbol {\mathcal H}}^{(1)}(\boldsymbol w,\boldsymbol q^{*,-},\boldsymbol q^{*,+},\boldsymbol n_l) +  \dfrac{1}{2}(\boldsymbol w\cdot \boldsymbol n_l +\alpha)(\eta^{*,+}-\eta^+)+  \dfrac{1}{2}(\boldsymbol w\cdot \boldsymbol n_l-\alpha)(\eta^{*,-}-\eta^-)\\
 &= \dfrac{1}{2}\left( \widehat {\boldsymbol{\mathcal H}}^{(1)}(\boldsymbol  w,\boldsymbol q^{*,-})\cdot \boldsymbol n_l + \widehat {\boldsymbol{\mathcal H}}^{(1)}(\boldsymbol  w,\boldsymbol q^{*,+})\cdot \boldsymbol n_l  -\alpha(\eta^{*,+}-\eta^{*,-})\right) +\vartheta\\
&= -\dfrac{1}{2}(\boldsymbol w\cdot \boldsymbol n_l +\alpha)\eta^+ - \dfrac{1}{2}(\boldsymbol w\cdot \boldsymbol n_l-\alpha)\eta^- = -(\boldsymbol w\cdot \boldsymbol n_l)\eta^- = -(\boldsymbol w \otimes \eta^-)\cdot \boldsymbol n_l.
\end{align*}
The superscript $(1)$ means the first component of the numerical flux $\widehat {\boldsymbol {\mathcal H}}$.  Some tedious algebraic operations and (\ref{etahubar}) yield  the  third and last equality.  Likewise, due to the fact  $h^{*,-} = h^{*,+}$, we get
\begin{align*}
\widehat {\boldsymbol {\mathcal H}}^{\text{mod},(2)} & = \widehat {\boldsymbol {\mathcal H}}^{(2)}(\boldsymbol w,\boldsymbol q^{*,-},\boldsymbol q^{*,+},\boldsymbol n_l) +  \left(\dfrac{g}{2}(\eta^{-} - b^-)^2 - \dfrac{g}{2}(h^{*,-})^2\right) n_{l_1}\\
&= \dfrac{1}{2}\left( \dfrac{g}{2}(\eta^{*,-} - b^-)^2 n_{l_1} + \dfrac{g}{2}(\eta^{*,+} - b^+)^2 n_{l_1}\right) + \left(\dfrac{g}{2}(\eta^{-} - b^-)^2 - \dfrac{g}{2}(h^{*,-})^2\right) n_{l_1}\\
&= \dfrac{g}{2}(\eta^{-} - b^-)^2n_{l_1},\\
\widehat {\boldsymbol {\mathcal H}}^{\text{mod},(3)}& = \dfrac{g}{2}(\eta^{-} - b^-)^2n_{l_2}.
\end{align*}
Thence
\begin{align*}
\text{RHS}_{ij}^{(1)} &=  -\sum _{l=1}^4 \displaystyle{\int_{e_{ij}^{l}(t)} -(\boldsymbol w \otimes \eta^-)\cdot \boldsymbol n_l \ ds}= \displaystyle{\int_{\partial \mathcal T_{ij}(t)} (\boldsymbol w \otimes \eta^-)\cdot \boldsymbol n \ ds} =\displaystyle{\iint_{\mathcal T_{ij}(t)}  \nabla \cdot(\boldsymbol w\otimes \eta ) \  dxdy}, \\
\text{RHS}_{ij}^{(2)} &= -\sum _{l=1}^4 \displaystyle{\int_{e_{ij}^{l}(t)} \dfrac{g}{2}(\eta^{-} - b^-)^2n_{l_1}  \ ds} + \displaystyle{\iint_{\mathcal T_{ij}(t)} -g(\eta-b)b_x  \ dxdy} \\
&= -\displaystyle{\int_{\partial \mathcal T_{ij}(t)} \dfrac{g}{2}(\eta^{-} - b^-)^2  n_{1} \ ds}+\displaystyle{\iint_{\mathcal T_{ij}(t)} -g(\eta-b)b_x  \ dxdy} \\
& = \displaystyle{\iint_{ \mathcal T_{ij}(t)} \left(-\dfrac{g}{2}(\eta - b)^2\right)_x -g(\eta-b)b_x  \ dxdy}  = 0=\displaystyle{\iint_{\mathcal T_{ij}(t)}  \nabla \cdot(\boldsymbol w\otimes hu )\  dxdy},\\
\text{RHS}_{ij}^{(3)} & = \displaystyle{\iint_{ \mathcal T_{ij}(t)} \left(-\dfrac{g}{2}(\eta - b)^2\right)_y -g(\eta-b)b_y  \ dxdy}  = 0=\displaystyle{\iint_{\mathcal T_{ij}(t)}  \nabla \cdot(\boldsymbol w\otimes hv )\  dxdy}.
\end{align*}
The penultimate equality is owing to the second and third equation of  (\ref{2Dsha}). The last equality follows from the equilibrium state solution. Then the property (\ref{prowell3}) follows.

\bigskip

{\noindent  \textbf{(for the scheme with special source term)}}
Noting that the reconstructed values $\eta_{i_\mu,j_\upsilon}$  at the Gauss point in formula (\ref{gausssij2}) is equal to $c$, it follows that $ r_{ij}^{(2)}, r_{ij}^{(3)}$ are disappearing when  $\boldsymbol  q$ is the hydrostatic state solution. What's more, using the fact that (\ref{etahubar}) we have
\begin{align}\label{hjnew}
\widehat{\boldsymbol { \mathcal H}}(\boldsymbol w,\boldsymbol q^{-},\boldsymbol q^{+}, \boldsymbol n_l) = \dfrac{1}{2}\left(  \boldsymbol { \mathcal H}(\boldsymbol w,\boldsymbol q^-)\cdot \boldsymbol n_l+ {\boldsymbol { \mathcal H}}(\boldsymbol w,\boldsymbol q^+)\cdot \boldsymbol n_l \right).
\end{align}
By substituting (\ref{hjnew}) into (\ref{special2D}), we obtain
\begin{align*}
\text{RHS}_{ij}^{(1)} &=- \sum _{l=1}^4 \displaystyle{\int_{e_{ij}^{l}(t)} \widehat {\boldsymbol {\mathcal H}}^{(1)}(\boldsymbol w,\boldsymbol q^{-},\boldsymbol q^{+}, \boldsymbol n_l)  \ ds}  \\
& = - \sum _{l=1}^4 \displaystyle{\int_{e_{ij}^{l}(t)} \dfrac{1}{2}\left[\left( - w^x \eta^-\right)n_{l_1} +\left( - w^y \eta^-\right)n_{l_2} + \left( - w^x \eta^+\right)n_{l_1} +\left(- w^y \eta^+\right)n_{l_2}\right] \ ds} \\
&= \sum _{l=1}^4 \displaystyle{\int_{e_{ij}^{l}(t)} w^x \eta^-n_{l_1} +  w^y \eta^-n_{l_2}\ ds} = \displaystyle{\int_{\partial \mathcal T_{ij}(t)} (\boldsymbol w\otimes \eta )\cdot \boldsymbol n_l \ ds}= \displaystyle{\iint_{\mathcal T_{ij}(t)}  \nabla \cdot(\boldsymbol w\otimes \eta )\  dxdy}, \\
\text{RHS}_{ij}^{(2)} &=- \sum _{l=1}^4 \displaystyle{\int_{e_{ij}^{l}(t)} \widehat {\boldsymbol {\mathcal H}}^{(2)}(\boldsymbol w,\boldsymbol q^{-},\boldsymbol q^{+}, \boldsymbol n_l) \ ds}  +  \dfrac{g}{2}\sum _{l=1}^4 \displaystyle{\int_{e_{ij}^{l}(t)} (\widehat {b^2} ) n_{l_1}  \ ds} - g \bar \eta_{ij}\sum _{l=1}^4 \displaystyle{\int_{e_{ij}^{l}(t)} \widehat {b}  n_{l_1}  \ ds}\\
& = \sum _{l=1}^4 \displaystyle{\int_{e_{ij}^{l}(t)} -\dfrac{1}{2}\left( \dfrac{g}{2}(\eta^- - b^-)^2 n_{l_1}+ \dfrac{g}{2}(\eta^+ - b^+)^2n_{l_1}\right) +  \dfrac{g}{2} (\widehat {b^2} ) n_{l_1} -   g \bar \eta_{ij}\widehat {b}  n_{l_1}  \ ds}\\
& = \sum _{l=1}^4 \displaystyle{\int_{e_{ij}^{l}(t)} -\left[ \dfrac{g}{4}(\eta^- - b^-)^2 + \dfrac{g}{4}(\eta^+ - b^+)^2\right]n_{l_1}+  \dfrac{g}{2}\dfrac{ (b^-)^2 + (b^+)^2}{2}n_{l_1} -  g \bar \eta_{ij}\dfrac{b^-+b^+}{2}  n_{l_1}   \ ds}\\
& = \sum _{l=1}^4 \displaystyle{\int_{e_{ij}^{l}(t)} -\dfrac{g}{2}c^2 n_{l_1}   \ ds} =
  -\dfrac{g}{2}c^2\displaystyle{\int_{\partial\mathcal T_{ij}(t)}n_{1} \ ds} = 0 = \displaystyle{\iint_{\mathcal T_{ij}(t)}  \nabla \cdot(\boldsymbol w\otimes hu )\  dxdy},\\
\text{RHS}_{ij}^{(3)} &=\sum _{l=1}^4 \displaystyle{\int_{e_{ij}^{l}(t)} -\dfrac{g}{2}c^2 n_{l_2}   \ ds} =
  -\dfrac{g}{2}c^2\displaystyle{\int_{\partial\mathcal T_{ij}(t)}n_{2} \ ds} = 0=\displaystyle{\iint_{\mathcal T_{ij}(t)}  \nabla \cdot(\boldsymbol w\otimes hv )\  dxdy}.
\end{align*}
This completes the proof of property (\ref{prowell3}).
\end{proof}

Following the discrete GCL property, we have the well-balanced property of the fully-discrete ALE-WENO hybrid scheme:
\begin{theorem}
The well-balanced semi-discrete ALE-WENO hybrid scheme together with the modified third-order SSP-RK time discretization  (\ref{RK32D})  are exact for the hydrostatic equilibrium state  (\ref{still_water2D}) and possess high order accuracy.
\end{theorem}

\subsubsection{Positivity-preserving property}
\begin{property}
Considering the well-balanced semi-discrete ALE-WENO hybrid schemes (\ref{hydro2D}) and (\ref{special2D}), combined with the modified third-order SSP-RK time discretization  (\ref{RK32D}) and the positivity-preserving limiter (\ref{xd}),(\ref{yd}), we can get the structure-preserving  fifth-order finite volume ALE-WENO hybrid scheme under the CFL condition  (\ref{CFLcondition}).
\end{property}

An easy mathematical operation of  schemes (\ref{hydro2D})  and (\ref{special2D}) satisfied by the cell averages of the surface level gives
$$ \dfrac{d}{dt}\left( \Delta_{ij}(t)\overline {\eta}_{ij}(t) \right)= -\sum _{l=1}^4 \displaystyle{\int_{e_{ij}^{l}(t)} \widehat {\boldsymbol {\mathcal H}}^{(1)}(\boldsymbol w,\boldsymbol q^{-},\boldsymbol q^{+}, \boldsymbol n_l) \ ds}, $$
which is the same as those in the finite volume ALE-WENO scheme  (\ref{se2D}), herein the proof  will not be reproduced.
\begin{remark}
The same technique in \cite{zhang2021positivity}
\begin{align*}
 b_{ij}^{\text{new}} (x,y) = \eta_{ij} (x,y) -  h^{\text{new}}_{ij} (x,y)
\end{align*} is adopted to  update the bottom function rather than changing the surface level $\eta_{ij} (x,y)$ in order not to disturb the stationary state and $h^{\text{new}}_{ij} (x,y)$ is the modified polynomial with positivity preserving limiter (\ref{modifiedhjn}) in the quadrilateral cell ${\mathcal T_{ij}(t)}$.
\end{remark}

\section{Numerical examples}\label{se:nu}
This section implements our proposed fifth-order finite volume ALE-WENO hybrid schemes on moving meshes for one- and two-dimensional shallow water equations. In all numerical experiments, the CFL number is taken as 0.6 except for the accuracy tests and the positivity-preserving tests given in the following specific examples.
The gravitation constant $g$ is taken as 9.812 $ m/s^2 $. Since we do not concentrate on the approaches of grid movement, unless otherwise specified, the moving meshes chosen in one-dimensional examples is as follows:
\begin{equation}\label{wangge1}
 x_{j+\frac{1}{2}}(t) = x_{j+\frac{1}{2}}(0)+\dfrac{1}{3(x_{\max}-x_{\min})^2}\sin\left(\dfrac{2\pi t}{t_{\text{end}}}\right)\left(x_{j+\frac{1}{2}}(0)-x_{\max}\right)\left(x_{j+\frac{1}{2}}(0)-x_{\min}\right),
 \end{equation}
and in  two-dimensional examples
 \begin{equation}\label{wangge2}
  \begin{aligned}
 & x_{i+\frac{1}{2}}(t) = x_{i+\frac{1}{2}}(0)+0.03\sin\left(\dfrac{2\pi t}{t_{\text{end}}}\right)\sin\left( \dfrac{2\pi  x_{i+\frac{1}{2}}(0)}{x_{\max}-x_{\min}} \right)\sin\left( \dfrac{2\pi  y_{j+\frac{1}{2}}(0)}{y_{\max}-y_{\min}} \right),\\
 & y_{j+\frac{1}{2}}(t) = y_{j+\frac{1}{2}}(0)+0.02\sin\left(\dfrac{4\pi t}{t_{\text{end}}}\right)\sin\left( \dfrac{2\pi  x_{i+\frac{1}{2}}(0)}{x_{\max}-x_{\min}} \right)\sin\left( \dfrac{2\pi  y_{j+\frac{1}{2}}(0)}{y_{\max}-y_{\min}} \right),
  \end{aligned}
 \end{equation}
 where  $x_{\min},x_{\max},y_{\min},y_{\max}$ are the vertices of the computational domain in  one- or two-dimensional and $t_{\text{end}}$ is the final time.

\subsection{One-dimensional tests}

\begin{example}{\bf Accuracy test}\label{accuracy1D}
\end{example}
This example tests the high order accuracy of the ALE-WENO hybrid schemes for a smooth solution on a domain $[0,1]$.
We choose the same bottom function and initial conditions as in \cite{xing2006new}.
 \begin{equation}\label{inismoo}
 b(x) = \sin^2(\pi x),\quad h(x,0)=5+e^{\cos(2\pi x)}, \quad  hu(x,0) = \sin(\cos(2\pi x)).
\end{equation}
We impose periodic boundary conditions and compute until  $t = 0.1$ to avoid the appearance of the shocks in the solution.
Here the CFL number is taken as 0.1 to ensure that spatial errors dominate. We treat the solution obtained by the fifth-order finite volume WENO-JS scheme on a uniform mesh consisting of 12800 grid cells as the reference solution to calculate the  $L^1$ errors and the order of accuracy for the surface level $h+b$ and the discharge $hu$.
The results using both the ALE-WENO scheme (denoted as non well-balanced scheme) (\ref{se2D}-\ref{s2D}), the well-balanced ALE-WENO schemes with hydrostatic reconstruction (\ref{hydro2D}) and special source term treatment (\ref{special2D}) are listed in Tables \ref{WENO5n1D} and \ref{WENO5h1D}. All schemes obtain the expected fifth-order accuracy on the moving mesh (\ref{wangge1}).

\begin{table}[htb]
  \centering
  \caption{  Example \ref{accuracy1D}:  $L^1$ errors for cell averages and numerical orders of accuracy by the non well-balanced ALE-WENO  scheme, with initial condition (\ref{inismoo}), $t = 0.1$.}\label{WENO5n1D}
  \begin{tabular}{ c c c c c  c c c c }
  \toprule
  \multirow{2}{*} {$N$}
    &\multicolumn{2}{c}{$h+b$}&\multicolumn{2}{c}{$ hu $}\\
   \cmidrule(lr){2-3} \cmidrule(lr){4-5}
   ~ &$L^1$ error &order&$L^1$ error &order\\
   \midrule
   25 & 9.27E{-03} &     -- &  6.71E{-02} &  --\\
   50 & 1.39E{-03} &   2.73 &  1.37E{-02} & 2.29\\
   100& 2.66E{-04} &   2.39 &  2.23E{-03} & 2.62\\
   200& 1.73E{-05} &   3.95 &  1.47E{-04} & 3.92\\
   400& 6.81E{-07} &   4.66 &  5.82E{-06} & 4.66\\
   800& 2.24E{-08} &   4.93 &  1.92E{-07} & 4.92\\
  \bottomrule
  \end{tabular}
\end{table}

\begin{table}[htb]
  \centering
  \caption{Example \ref{accuracy1D}: $L^1$ errors for cell averages and numerical orders of accuracy by the well-balanced ALE-WENO hybrid schemes, with initial condition (\ref{inismoo}), $t = 0.1$.}\label{WENO5h1D}
  \begin{tabular}{ c c c c c  c c c c }
  \toprule
   \multirow{3}{*} {$N$} &\multicolumn{4}{c}{hydrostatic reconstruction} &\multicolumn{4}{c}{special source term treatment}\\
   \cmidrule(lr){2-5} \cmidrule(lr){6-9}
    &\multicolumn{2}{c}{$h+b$}&\multicolumn{2}{c}{$ hu $}& \multicolumn{2}{c}{$h+b$}&\multicolumn{2}{c}{$ hu $}\\
   \cmidrule(lr){2-3} \cmidrule(lr){4-5} \cmidrule(lr){6-7} \cmidrule(lr){8-9}
   ~ &$L^1$ error &order&$L^1$ error &order&$L^1$ error &order&$L^1$ error &order\\
   \midrule
   25 & 9.23E{-03} &    -- & 6.48E{-02} &    -- & 1.02E{-02} &    -- & 7.26E{-02} &   -- \\
   50 & 1.36E{-03} &  2.76 & 1.35E{-02} &  2.26 & 1.40E{-03} &  2.86 & 1.35E{-02} & 2.43 \\
   100& 2.65E{-04} &  2.36 & 2.23E{-03} &  2.60 & 2.70E{-04} &  2.37 & 2.27E{-03} & 2.57 \\
   200& 1.73E{-05} &  3.94 & 1.47E{-04} &  3.92 & 1.72E{-05} &  3.97 & 1.47E{-04} & 3.95 \\
   400& 6.81E{-07} &  4.66 & 5.82E{-06} &  4.66 & 6.80E{-07} &  4.66 & 5.81E{-06} & 4.66 \\
   800& 2.24E{-08} &  4.93 & 1.91E{-07} &  4.93 & 2.23E{-08} &  4.93 & 1.91E{-07} & 4.93 \\
  \bottomrule
  \end{tabular}
\end{table}

\begin{example}{\bf Positivity-preserving test for Riemann problems}
\label{flat1D}
\end{example}
To test the positivity-preserving property of our ALE-WENO schemes on moving meshes, we give two examples introduced in \cite{holden2015front} over a flat bottom which means $b(x)=0$. The CFL number is set as 0.08 to satisfy the condition (\ref{CFLcondition}). We test the following cases with transmissive boundary conditions on moving meshes  (\ref{wangge1}) using 200 grid cells at the beginning. Notice that there is no need to use the well-balanced schemes for our simulation due to the flat bottom topography; we only illustrate the results calculated by the non well-balanced ALE-WENO scheme.

\smallskip
{\noindent \textbf{Moses's first problem}}
\smallskip

The first one is computed on domain $[-200,400]$.  The initial condition for the water height does not contain dry area originally and the initial velocity is nonzero,  given in the following way
 \begin{equation}\label{flat2}
 h(x,0)=\left\{\begin{array}{lll}
    5,&\text{if} \ \ x\leq0, \\
    10,&\text{otherwise},\\
    \end{array}\right.  \quad  u(x,0) = \left\{\begin{array}{lll}
    0,&\text{if} \ \ x\leq0 ,\\
    40,&\text{otherwise}.\\
    \end{array}\right.
\end{equation}
We can find the analytical solution in \cite{bokhove2005flooding}.
The dry region will appear when the constant initial values satisfy the criterion $\sqrt{gh_l}+\sqrt{gh_r}+u_l - u_r<0$, and then two expansion waves propagate away from each other.
The numerical solutions computed at times $t=2,4,6$ and using the analytic solutions for comparison are shown in Fig. \ref{WENOsecond}.

\smallskip
{\noindent \textbf {Moses's second problem}}
\smallskip

Another case is  Moses's second problem which involving multiple Riemann problem. On computational domain $[-300,300]$, initial conditions are given by
\begin{equation}\label{flat3}
 h(x,0)=\left\{\begin{array}{lll}
    10,&\text{for} \ \ x\leq-70, \\
    0, &\text{for} \ \ -70< x< 70,\\
    10 &\text{for} \ \  x\geq 70,\\
    \end{array}\right.  \quad  u(x,0) = 0.
\end{equation}
Before the two rarefaction waves interact, the analytic solution patches the solution of two dam-breaking problems together. The explicit expressions become difficult at some positive time when the rarefaction waves interact. Herein the results of 2500 uniform grid cells are calculated as reference solutions.
Fig. \ref{WENOmoses}  presents the solutions of the water height and the discharge at times $t=0,1,12,18$. The minimum of the water height at time $t=1$ are listed in  Table \ref{minwater}.
The two tests above show that the numerical results match the exact ones well on moving meshes with the positivity-preserving limiter. It can realize non-oscillatory and high-resolution interface tracking of fluid motion.

\begin{figure}[htbp]
  \centering
  \subfigure[water height]{
  \centering
     \includegraphics[width= 6.5cm,scale=1]{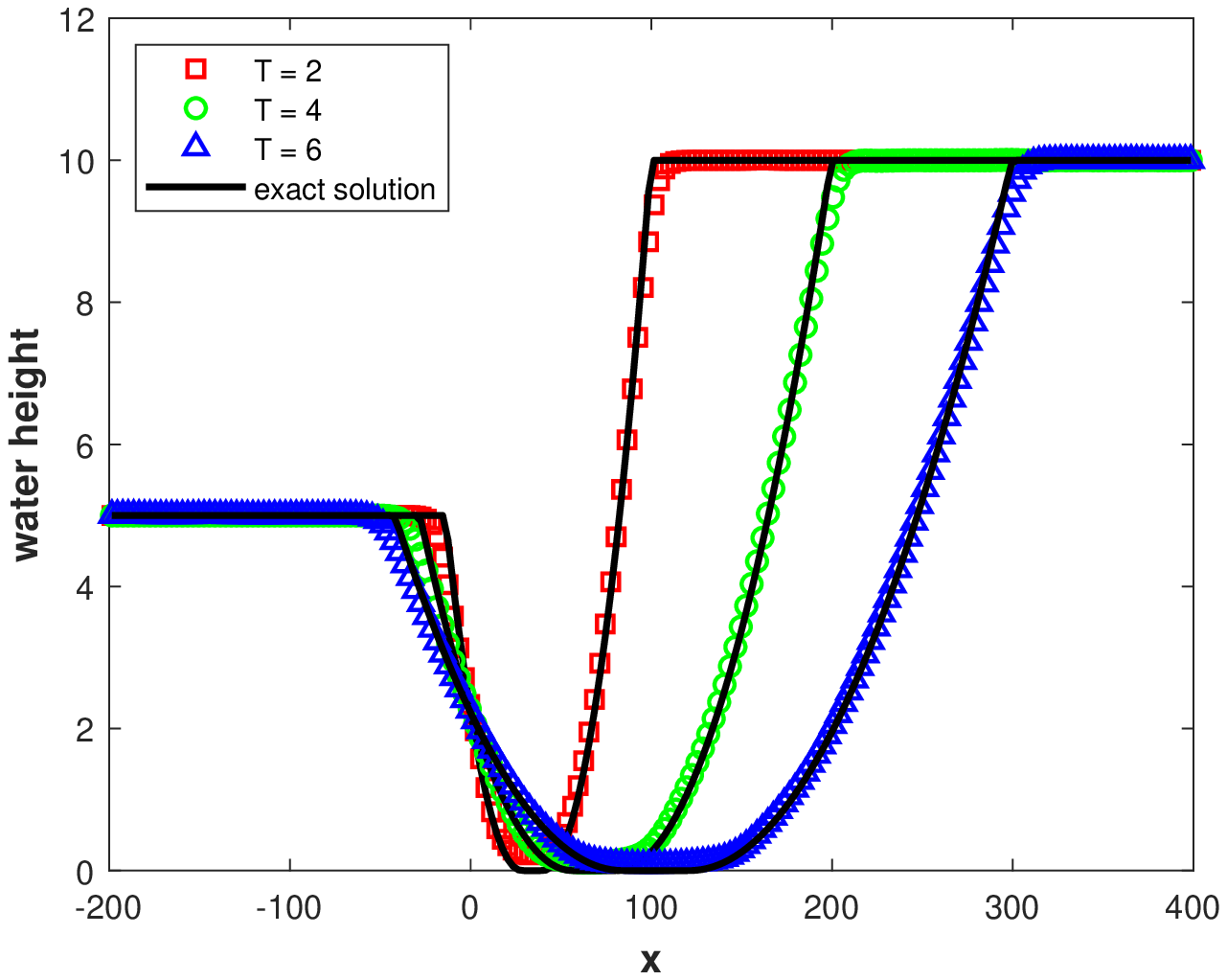}
  }
  \subfigure[the discharge]{
  \centering
     \includegraphics[width= 6.5cm,scale=1]{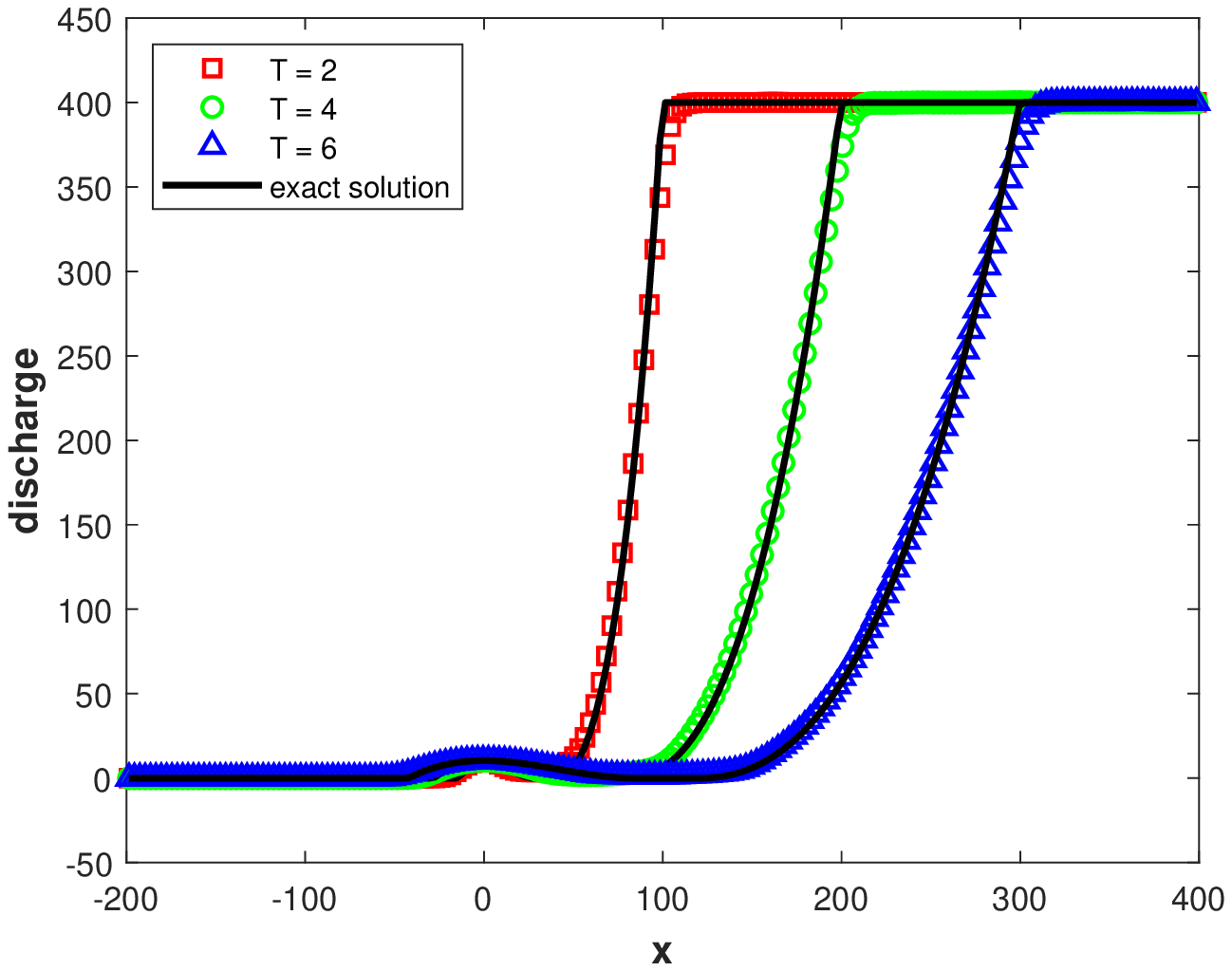}
  }
   \caption{Example \ref{flat1D}: The non well-balanced ALE-WENO scheme for Moses's first problem with initial condition (\ref{flat2}). The numerical and analytical solutions of the water height (left) and the discharge (right) with 200 moving grid cells at different times. }\label{WENOsecond}
\end{figure}

\begin{table}[htb]
  \centering
  \caption{Example \ref{flat1D}: The minimum of the water height at $t=1$ with  initial condition (\ref{flat3}).}\label{minwater}
  \begin{tabular}{c c c c c c c c }
   \toprule
   \multirow{1}{*} {Scheme} &{ non well-balanced scheme}\\
    \midrule
    $\min h$   & 1.00E{-11} \\
   \bottomrule
  \end{tabular}
\end{table}

\begin{figure}[htbp]
\setlength{\abovecaptionskip}{0.1cm}
\setlength{\belowcaptionskip}{0.0cm}
  \centering
  \subfigure[water height, $t = 0$]{
  \centering
     \includegraphics[width= 5.5cm,scale=1]{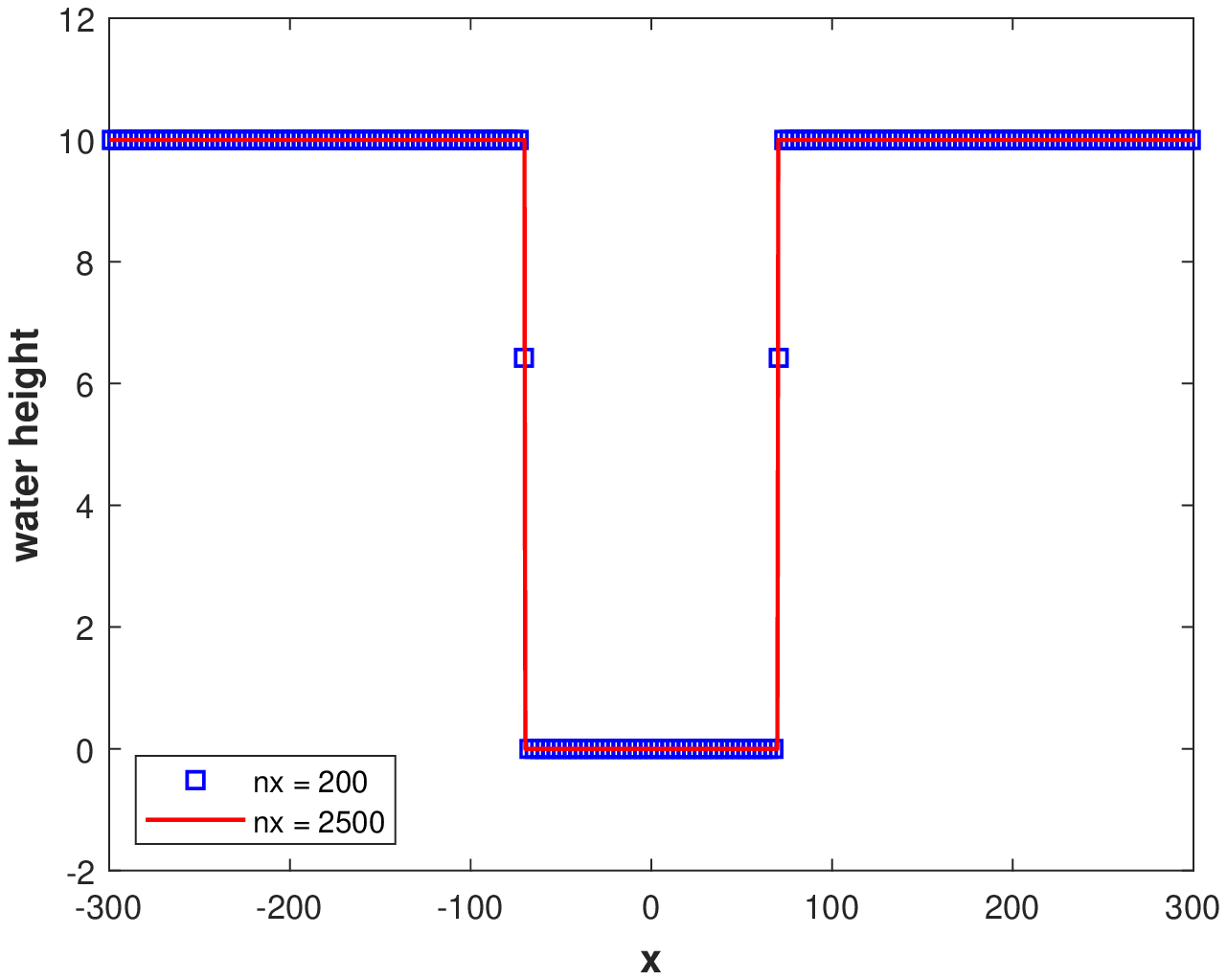}
  }
  \subfigure[the discharge, $t = 0$]{
  \centering
     \includegraphics[width= 5.5cm,scale=1]{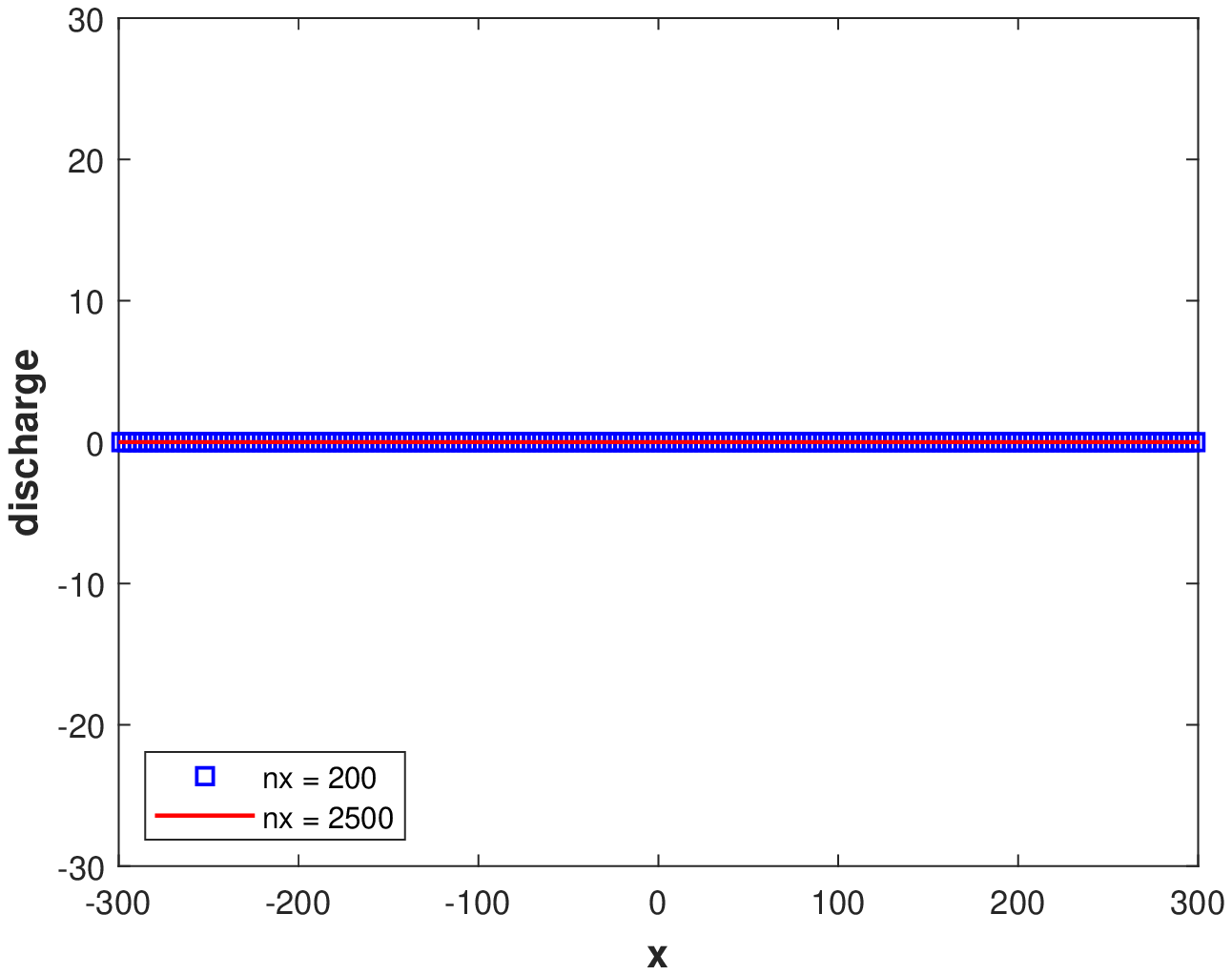}
  }
  \subfigure[water height, $t = 1$]{
  \centering
     \includegraphics[width= 5.5cm,scale=1]{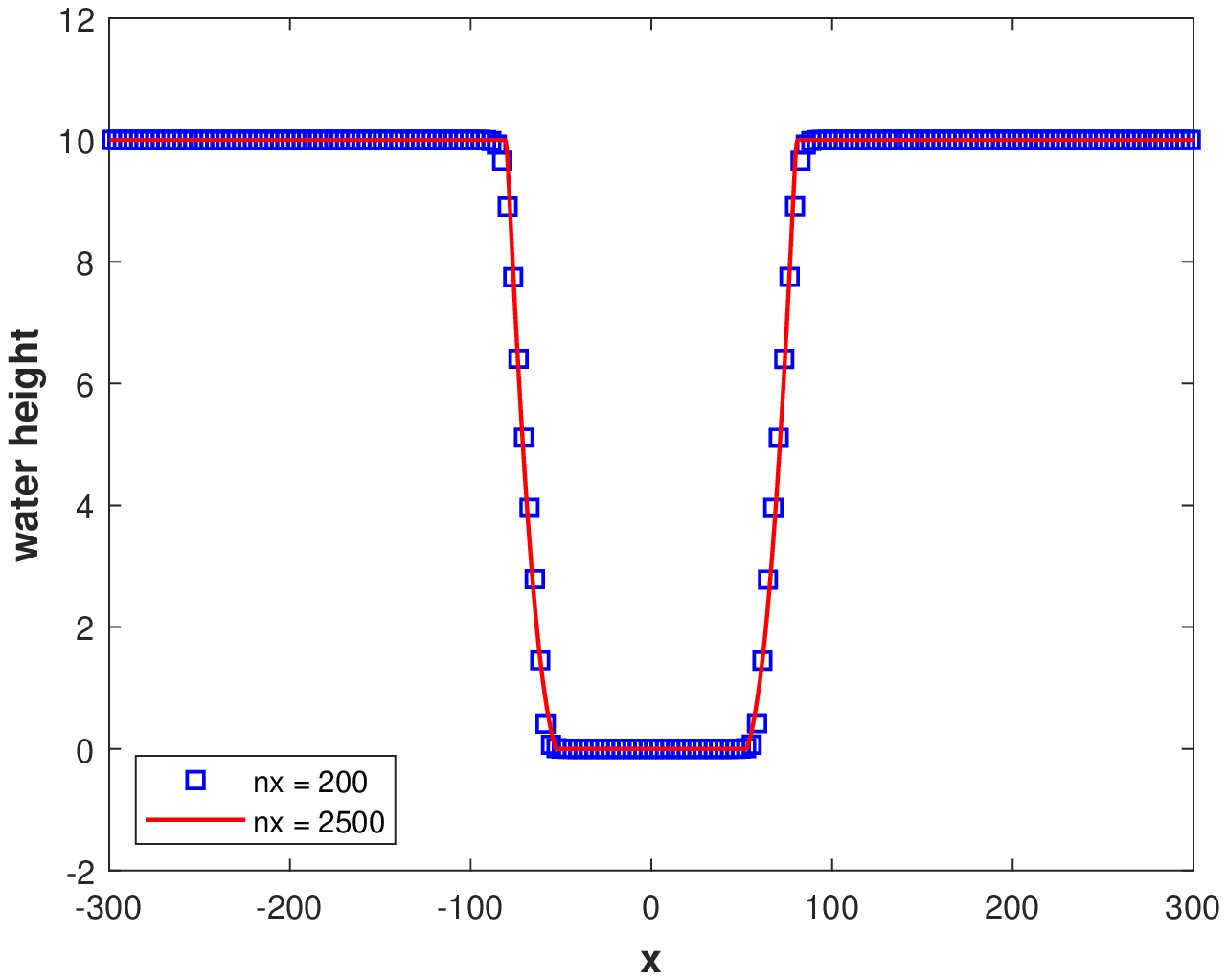}
  }
  \subfigure[the discharge, $t = 1$]{
  \centering
     \includegraphics[width= 5.5cm,scale=1]{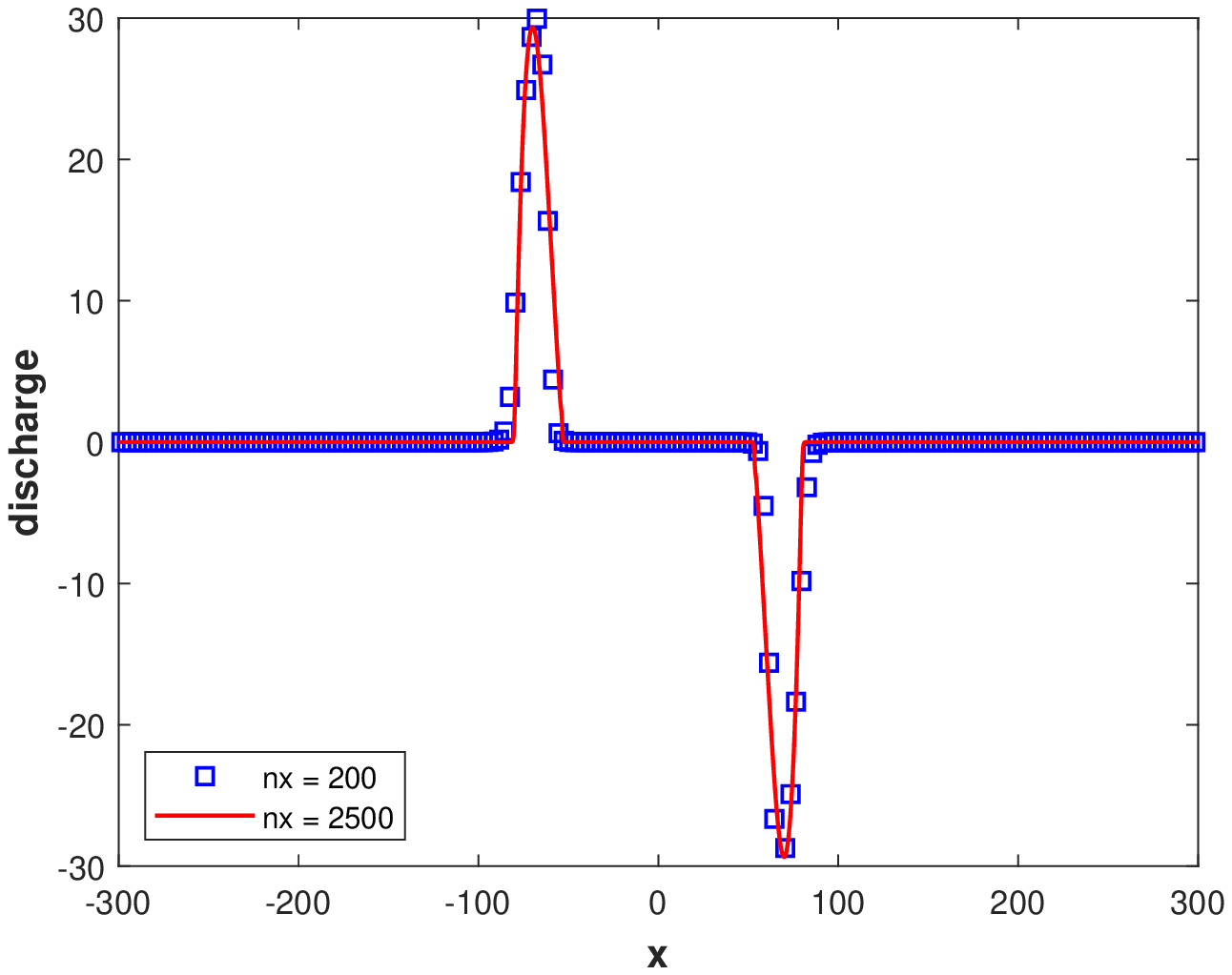}
  }
  \subfigure[water height, $t = 12$]{
  \centering
     \includegraphics[width= 5.5cm,scale=1]{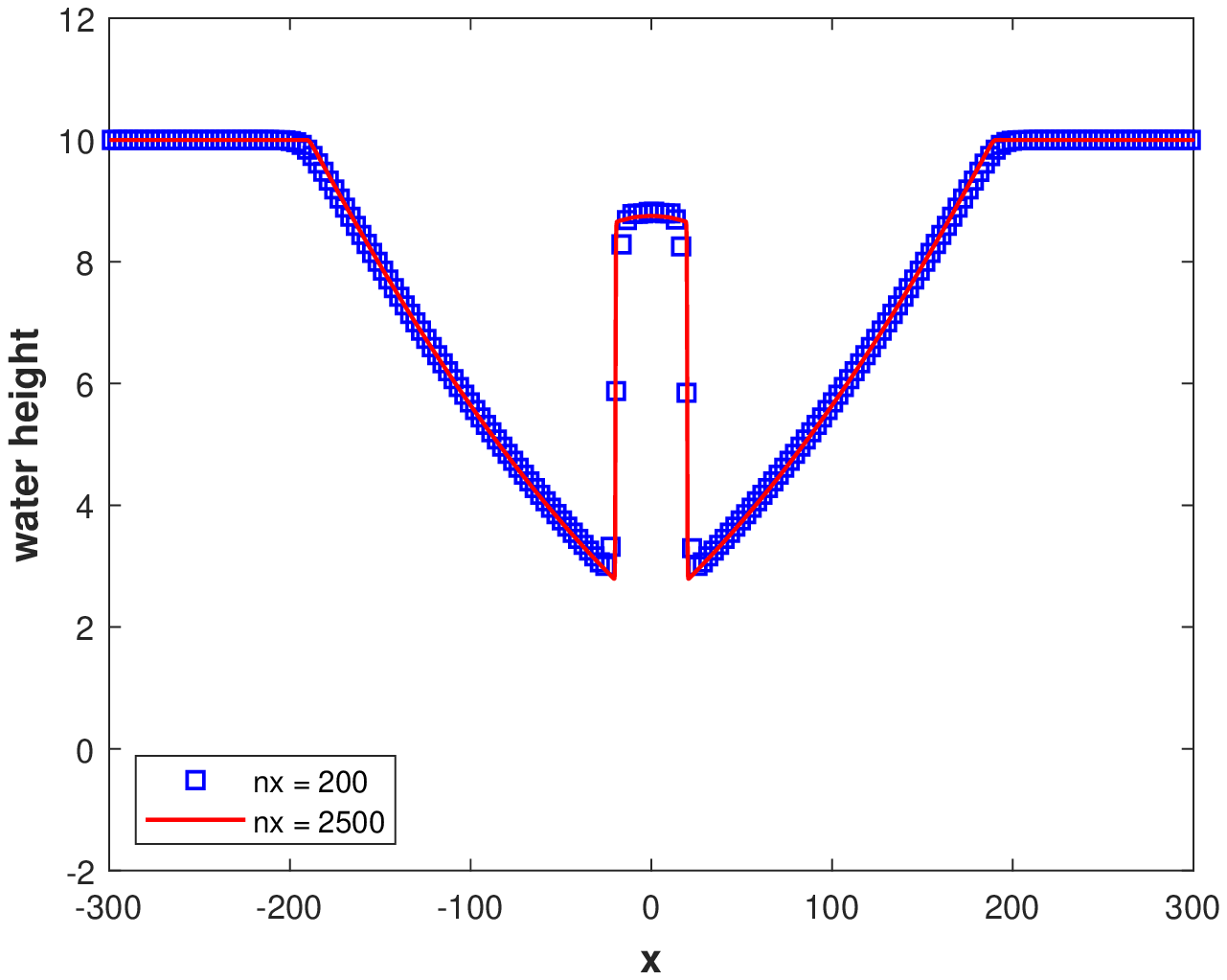}
  }
  \subfigure[the discharge, $t = 12$]{
  \centering
     \includegraphics[width= 5.5cm,scale=1]{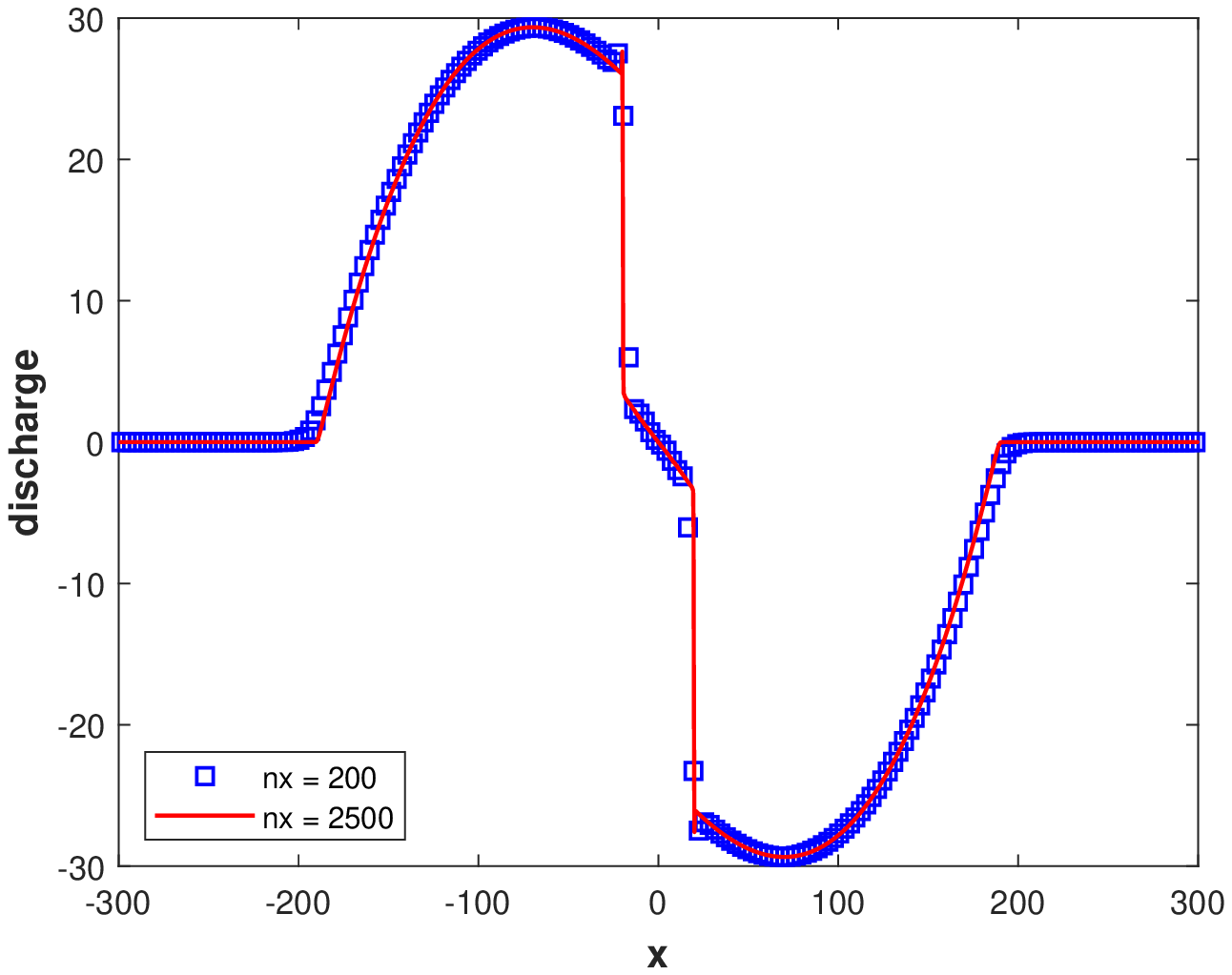}
  }
  \subfigure[water height, $t = 18$]{
  \centering
     \includegraphics[width= 5.5cm,scale=1]{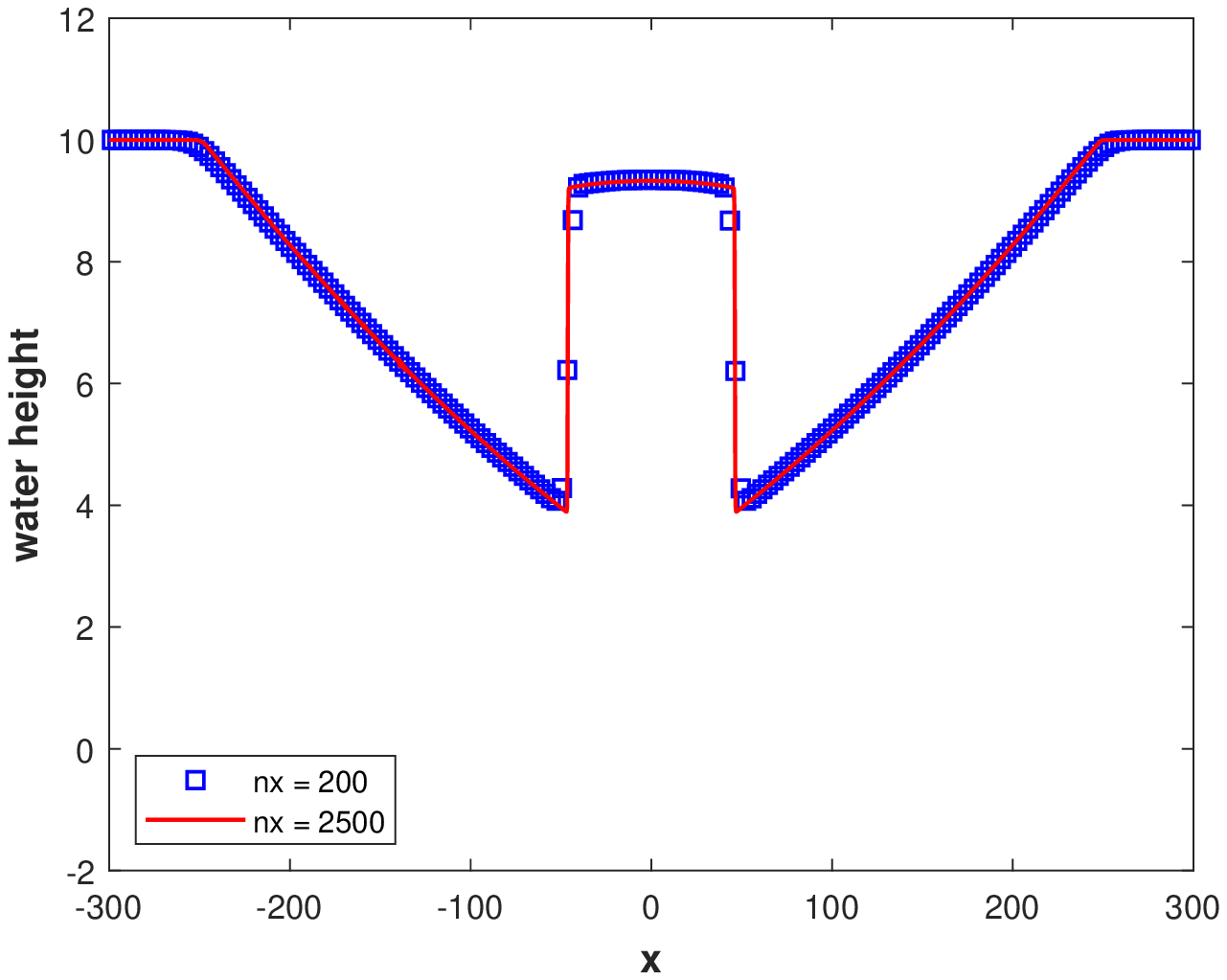}
  }
  \subfigure[the discharge, $t = 18$]{
  \centering
     \includegraphics[width= 5.5cm,scale=1]{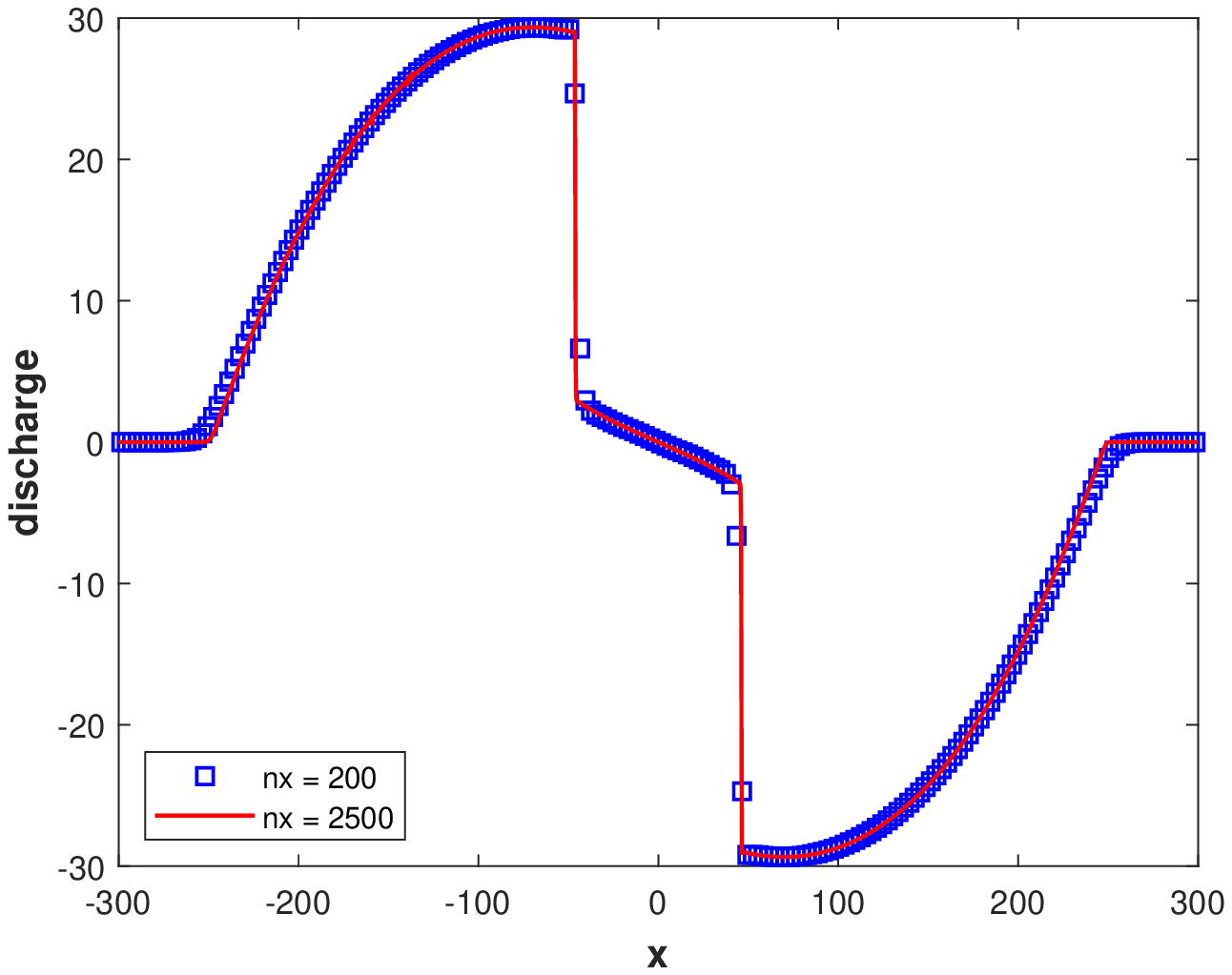}
  }
   \caption{Example \ref{flat1D}: The non well-balanced ALE-WENO scheme for Moses's second problem with initial condition (\ref{flat3}). The numerical solutions of the water height (left) and the discharge (right) with 200 moving and 2500 uniform grid cells at different times $t=0,1,12,18$,  from top to bottom.}\label{WENOmoses}
\end{figure}

\begin{example}{\bf Test for the well-balanced property}\label{exactC1D}
\end{example}
We use this example to show that compared to the non well-balanced scheme, the well-balanced ALE-WENO hybrid schemes indeed preserve the well-balanced property on moving meshes. Considering two different bottom topography on the computational domain $0\leq x\leq 10$, one is a smooth bottom
\begin{equation}\label{smo1D}
  b(x)=5e^{-0.4(x-5)^2},
\end{equation}
and the other is a discontinuous bottom
\begin{equation}\label{dis1D}
  b(x)=\left\{\begin{array}{lll}
    4,&\text{if}\ x\in[4,8], \\
    0,&\text{otherwise}.\\
    \end{array}\right.
\end{equation}
The initial condition is given by
$$h+b = 10,\quad hu=0.$$
We compute the solution up to $t = 0.5$ on moving grid (\ref{wangge1}) with $N = 200$ uniform cells at the beginning, and calculate the $L^1$ and $L^{\infty}$ errors at double precision for the surface level $h+b$ and the discharge $hu$. Computational results over two bottom functions are shown in Tables \ref{smoothh1D} and \ref{nonsmoothh1D} respectively. It can be seen that no matter what the bottom topography is, the numerical errors based on the well-balanced schemes are at the level of machine accuracy, which is in corresponds with the well-balanced property we expected.
\begin{table}[htb]
  \centering
  \caption{Example \ref{exactC1D}: $L^1$ and $L^{\infty}$ errors for the hydrostatic equilibrium state with smooth bottom topography  (\ref{smo1D}) by ALE-WENO hybrid schemes.}\label{smoothh1D}
  \begin{tabular}{c c c c c c c c c c }
   \toprule
    \multirow{2}{*} {Scheme} &\multicolumn{2}{c}{$L^1$ error}&\multicolumn{2}{c}{$L^{\infty}$ error}\\
    \cmidrule(lr){2-3} \cmidrule(lr){4-5}
    &{$h+b$}&{$hu$}&{$h+b$}&{$hu$}\\
    \midrule
   non well-balanced  scheme     & 7.97E{-06} &  4.09E{-05} &  1.95E{-06} & 1.09E{-05} \\
   hydrostatic reconstruction    & 8.72E{-13} &  8.49E{-13} &  1.10E{-13} & 2.39E{-13} \\
   special source term treatment & 8.65E{-13} &  9.18E{-13} &  1.10E{-13} & 2.77E{-13} \\
     \bottomrule
  \end{tabular}
\end{table}

\begin{table}[htb]
  \centering
  \caption{Example \ref{exactC1D}: $L^1$ and $L^{\infty}$ errors for the hydrostatic equilibrium state with discontinuous bottom topography (\ref{dis1D}) by ALE-WENO hybrid schemes.}\label{nonsmoothh1D}
  \begin{tabular}{c c c c c c c c c c }
   \toprule
    \multirow{2}{*} {Scheme} &\multicolumn{2}{c}{$L^1$ error}&\multicolumn{2}{c}{$L^{\infty}$ error}\\
    \cmidrule(lr){2-3} \cmidrule(lr){4-5}
    &{$h+b$}&{$hu$}&{$h+b$}&{$hu$}\\
    \midrule
   non well-balanced  scheme     & 1.07E{+00} & 3.08E{-01} &  5.30E{+00} & 2.51E{+00} \\
   hydrostatic reconstruction    & 9.24E{-13} & 9.17E{-13} &  1.31E{-13} & 2.77E{-13} \\
   special source term treatment & 9.24E{-13} & 9.33E{-13} &  1.23E{-13} & 2.43E{-13} \\
     \bottomrule
  \end{tabular}
\end{table}

\begin{example}{\bf A small perturbation of a steady-state water}\label{perturbation1D}
\end{example}
To demonstrate the performance of our ALE-WENO scheme and  well-balanced ALE-WENO hybrid schemes  on a rapidly varying flow and the small perturbation of a hydrostatic state, bottom topography, and the initial conditions are chosen from the same example proposed by LeVeque \cite{leveque1998balancing}:
  \begin{align*}
    & b(x)=\left\{\begin{array}{lll}
    0.25(\cos(10\pi(x-1.5))+1),&  \text{if} \ 1.4\leq x \leq 1.6, \\
    0,&\text{otherwise},\\
    \end{array}\right.  \\
    & (hu)(x,0)=0,\quad h(x,0) = \left\{\begin{array}{lll}
    1-b(x)+\epsilon,&  \text{if} \  1.1\leq x \leq 1.2, \\
    1-b(x),&\text{otherwise},\\
    \end{array}\right.
  \end{align*}
on the computational domain $[0,2]$. $\epsilon$ is a perturbation parameter. We test  $\epsilon=0.001$ (small pulse) at $t=0.2$  with transmissive boundary conditions. Figs.  \ref{WENO0.001h1D} displays the results with 200 moving cells and 3000 uniform cells for comparison.  Here we take $\varepsilon=10^{-9}$ in the ALE-WENO reconstruction  nonlinear weights formula (\ref{omegajjj}) to make sure that it's smaller than the square of the perturbation to avoid oscillations.\par
Theoretically, the exact solution will propagate to the left and right at the characteristic speeds $\pm \sqrt{gh}$.
Although the non well-balanced ALE-WENO  scheme behaves well in most experiments, it is difficult to calculate on a rapidly varying flow over a smooth bed for the small pulse perturbation.
By contrast, our well-balanced ALE-WENO hybrid schemes have their advantages in capturing the small perturbation of the hydrostatic equilibrium state, and the solutions are free of spurious numerical oscillations near the discontinuity.

\begin{figure}[htb]
  \centering
  \subfigure[surface level $h+b$]{
  \centering
     \includegraphics[width= 7cm,scale=1]{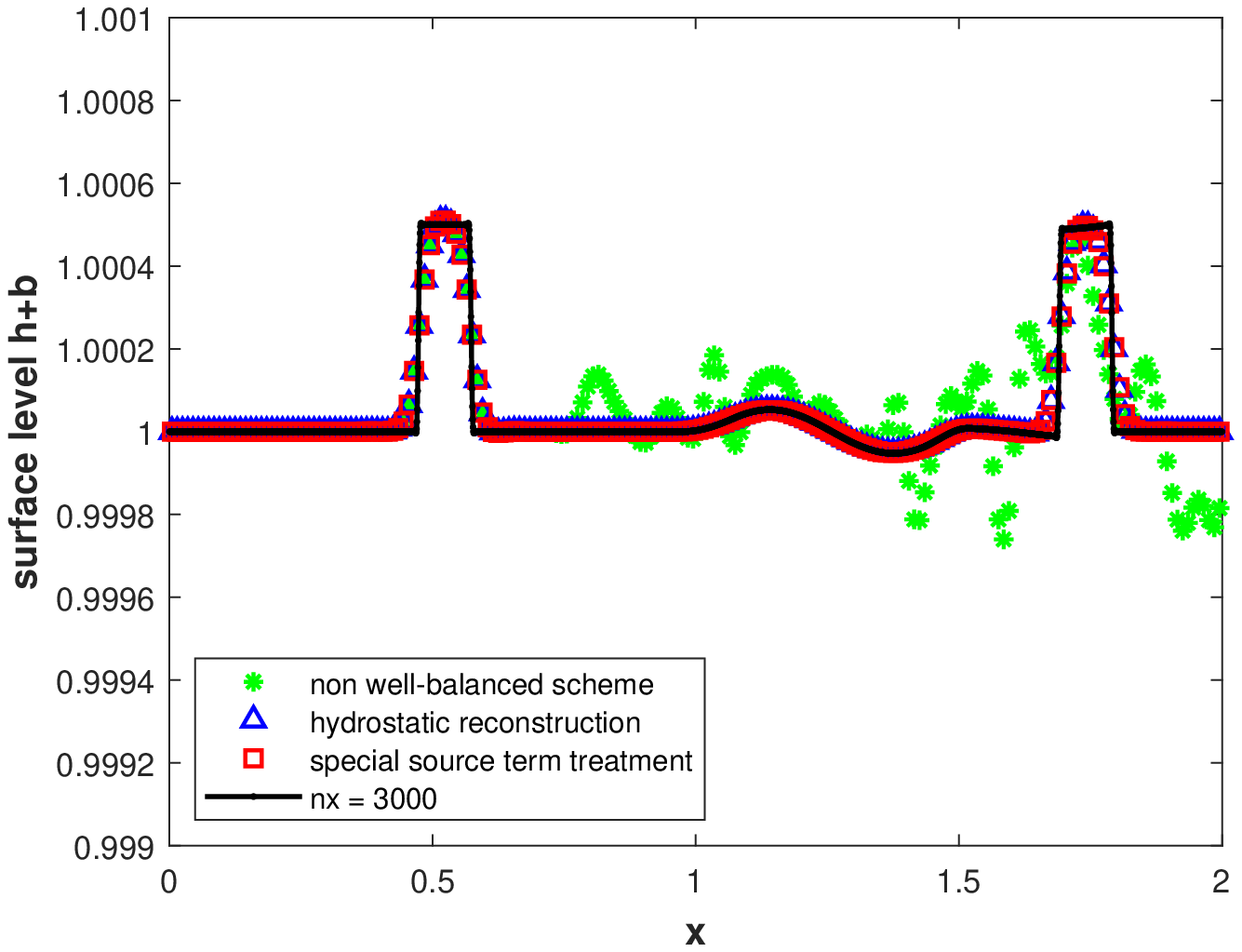}
  }
  \subfigure[the discharge $hu$]{
  \centering
     \includegraphics[width= 7cm,scale=1]{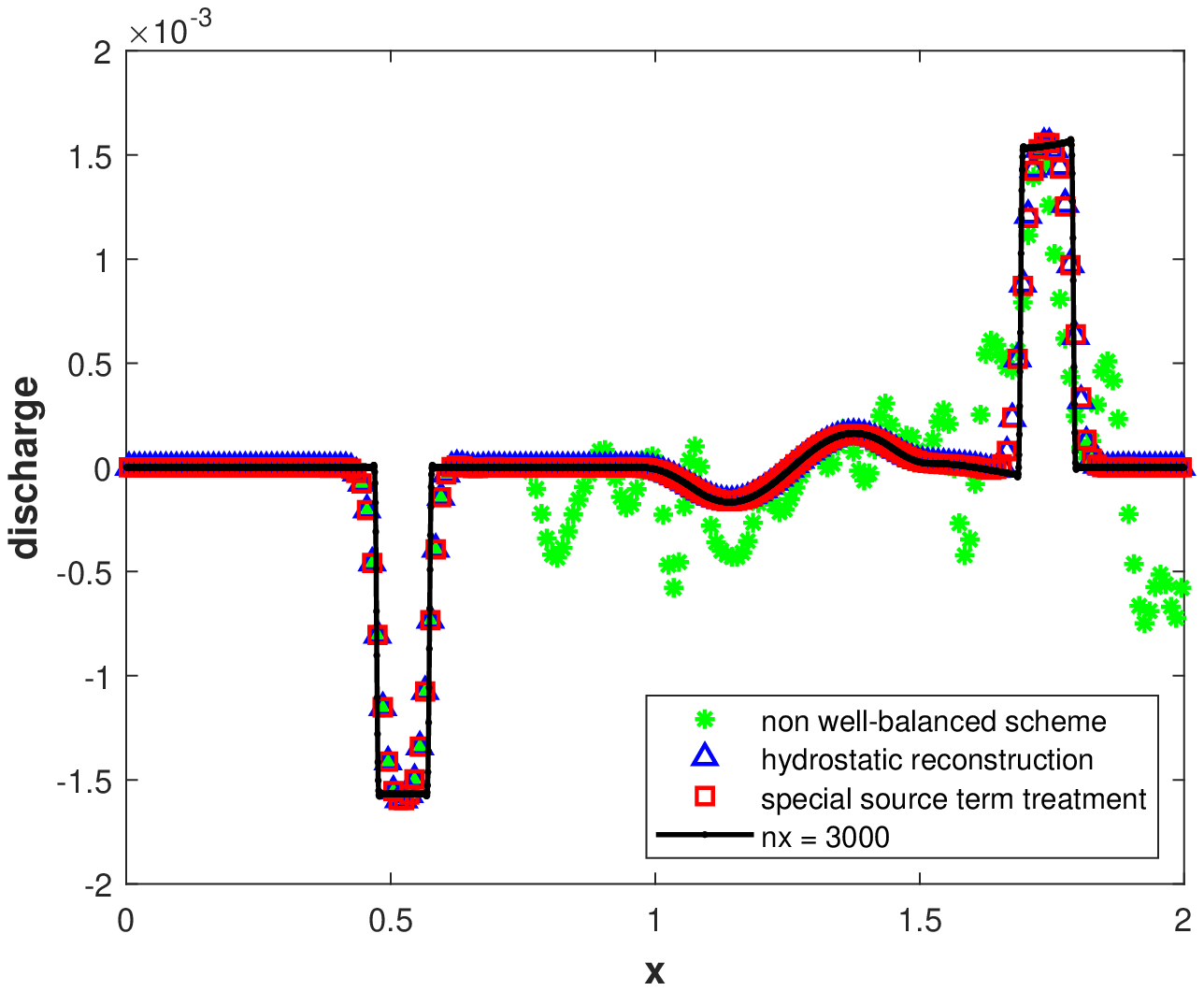}
  }
  \caption{Example \ref{perturbation1D}: ALE-WENO hybrid schemes with 200 moving cells for the small pulse $\epsilon=0.001$ at $t=0.2$. Left: surface level $h+b$; right: the discharge $hu$.}\label{WENO0.001h1D}
\end{figure}

For demonstrating the advantage of the ALE method, we choose a special moving mesh for the big pulse problem ($\epsilon=0.2$). Grid at the beginning is
\begin{equation}\label{0.2inis}
    x_{i+\frac{1}{2}}(0)=\left\{ \begin{array}{lcl}
                0.0372 i,&& 0\leqslant i \leqslant 25 , \\
                0.93+0.008(i-25), &&26 \leqslant i \leqslant 75, \\
                1.33+0.0268(i-75),&&76 \leqslant i \leqslant 100.
                 \end{array} \right.
\end{equation}
At the final time the grid is defined by
\begin{equation}\label{0.2ends}
    x_{i+\frac{1}{2}}(t_{\text{end}})=\left\{ \begin{array}{lcl}
                2 i/35,&& 0\leqslant i \leqslant 7,  \\
                0.4+{(i-7)}/140, &&8 \leqslant i \leqslant 35, \\
                0.6+53(i-35)/1650,&&36 \leqslant i \leqslant 68 ,\\
                1.66+(i-68)/150, &&69 \leqslant i \leqslant 98 ,\\
                1.86+0.07(i-98),&&99 \leqslant i \leqslant 100,
                 \end{array} \right.
\end{equation}
at this time we define the moving grid velocity
\begin{equation}\label{velocity}
  w_{i+\frac{1}{2}} = \dfrac{x_{i+\frac{1}{2}}(t_{\text{end}}) - x_{i+\frac{1}{2}}(0)}{t_{\text{end}}-0},
\end{equation}
and the moving mesh:
\begin{equation}\label{meshspecial}
  x_{i+\frac{1}{2}}(t) = x_{i+\frac{1}{2}}(0) + w_{i+\frac{1}{2}}t.
\end{equation}
We use our two well-balanced  ALE-WENO hybrid schemes consisting of 100 moving grid cells (\ref{meshspecial}) in contrast with the results resolved by the well-balanced WENO hybrid scheme involving 100 and 3000  static uniform grids.
Fig. \ref{WENO0.2h1Dcom} shows the numerical solutions with hydrostatic reconstruction, illustrating that under a relatively coarse mesh. There is a better agreement between the present result on an appropriate moving mesh and the reference computation than those solved by the WENO  scheme on the static mesh.

\begin{figure}[htb]
  \centering
  \subfigure[surface level $h+b$]{
  \centering
     \includegraphics[width= 7cm,scale=1]{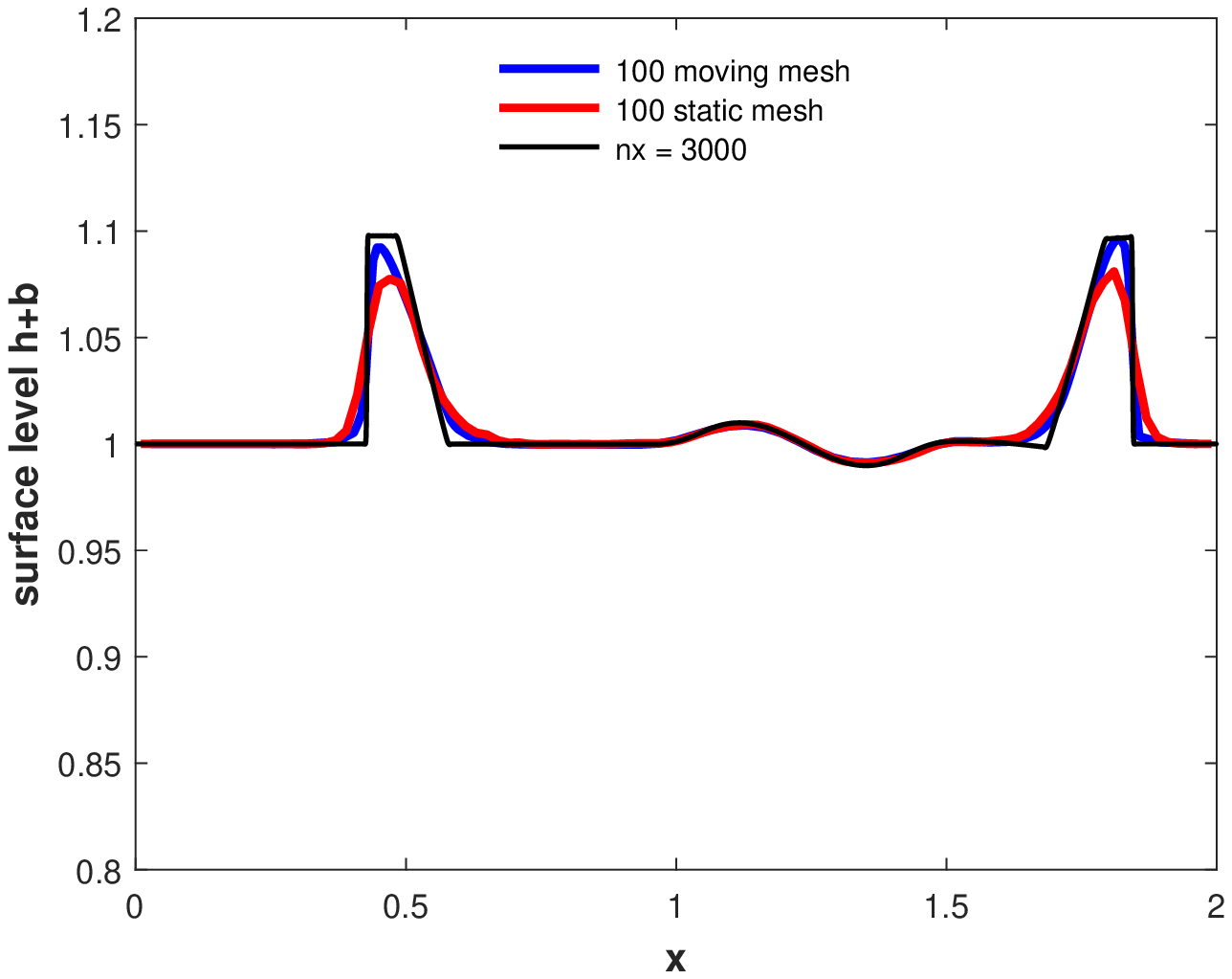}
  }
  \subfigure[the discharge $hu$]{
  \centering
     \includegraphics[width= 7cm,scale=1]{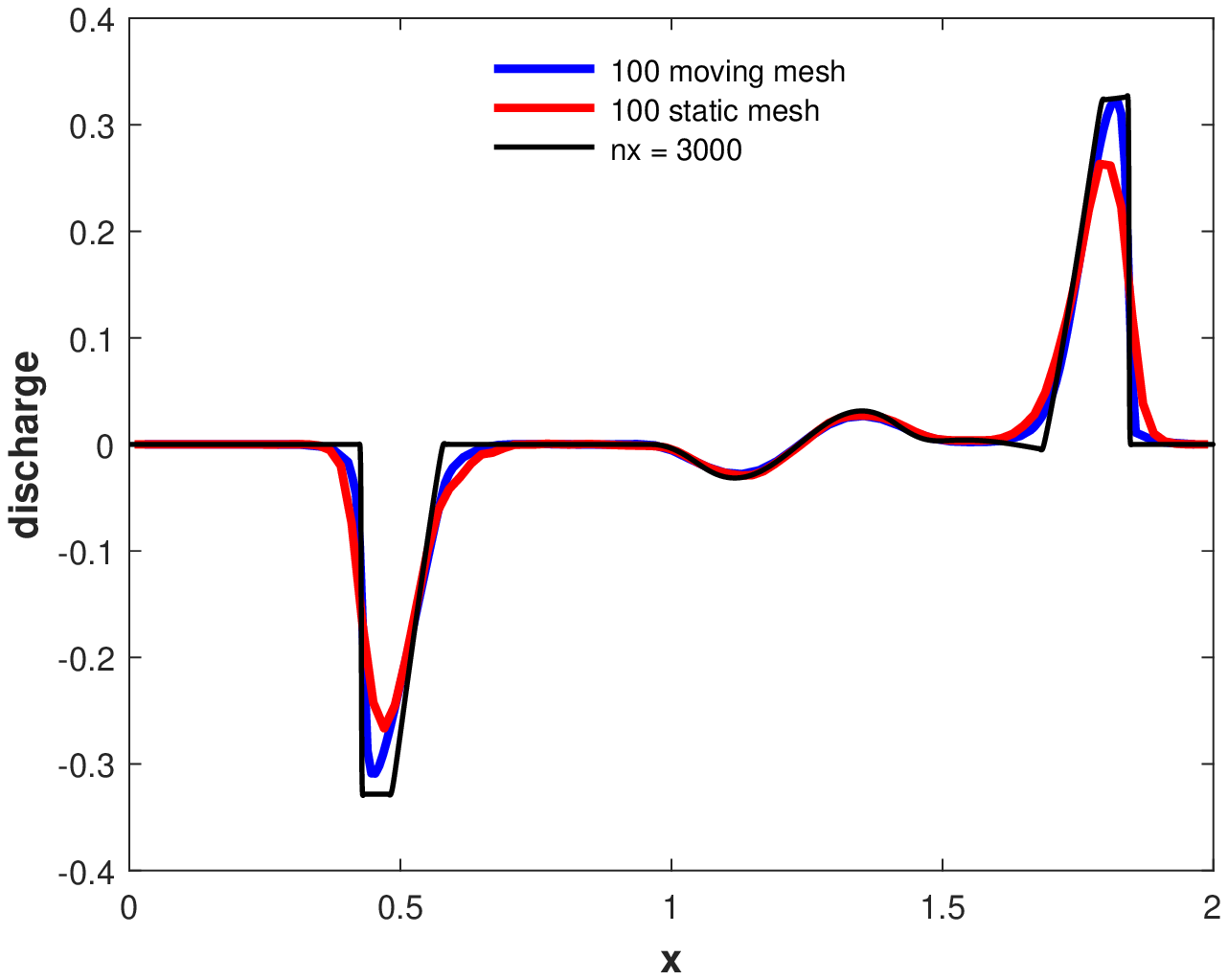}
  }
  \caption{Example \ref{perturbation1D}: The ALE-WENO hybrid scheme with hydrostatic reconstruction, by using 100 moving and static cells for the big pulse $\epsilon=0.2$ at  $t=0.2$. Left: the surface level $h+b$; right: the discharge $hu$.}\label{WENO0.2h1Dcom}
\end{figure}

\begin{example}{\bf The dam breaking problem over a discontinuous bottom topography}\label{bump1D}
\end{example}
We implement a dam break test case used in \cite{xing2006new} over a rectangular bump which involves  rapidly varying flow over a discontinuous bottom topography on a domain $[0,1500]$
\begin{equation}\label{bdam}
  b(x)=\left\{\begin{array}{lcl}
    8,&   \text{if} \ \   562.5\leq x \leq 937.5, \\
    0,&\text{otherwise}.\\
    \end{array}\right.
\end{equation}
and the initial conditions are given by
\begin{equation}\label{inidam}
 (hu)(x,0)=0,\quad h(x,0) = \left\{\begin{array}{lcl}
    20-b(x),& \text{if} \ \  x \leq 750, \\
    15-b(x),&\text{otherwise}.\\
    \end{array}\right.
\end{equation}
We can see from the values (\ref{bdam})-(\ref{inidam})  that the bottom topograph is in poor condition since the water height $h(x)$  and the bottom topography $b(x)$ is discontinuous at the points $x=562.5$ and $x=937.5$, while the surface level $h(x)+b(x)$ is smooth there, consequently we choose the following moving mesh
\begin{equation}\label{wangge3}
\begin{aligned}
 x_{j+\frac{1}{2}}(t) =x_{j+\frac{1}{2}}(0)& + \dfrac{1}{3(x_{\max}-x_{\min})^4}\sin\left(\dfrac{2\pi t}{t_{\text{end}}}\right)\left(x_{j+\frac{1}{2}}(0)-x_{\max}\right)\\
 &\left(x_{j+\frac{1}{2}}(0)-x_{\min}\right)\left(x_{j+\frac{1}{2}}(0)-562.5\right)\left(x_{j+\frac{1}{2}}(0)-937.5\right),
 \end{aligned}
 \end{equation}
in order to make sure that  $x=562.5$ and $x=937.5$ are located at the cell boundary. Furthermore,  the approximation of the bottom topography is chosen as the $L^2$ projection (\ref{baverage}) of the exact bottom topography instead of (\ref{bottomrk}). Figs. \ref{WENOh151D} and  \ref{WENOh601D} display the numerical results at two different times $t=15$ and $t=60$, calculated by our ALE-WENO hybrid schemes with 100 moving grid cells and 4000 uniform grid cells for comparison.
It shows a failed simulation of the non well-balanced ALE-WENO scheme, while two well-balanced WENO hybrid schemes on moving meshes give well resolved and oscillatory-free solutions.

\begin{figure}[htb]
  \centering
  \subfigure[initial condition and bottom topography]{
  \centering
     \includegraphics[width= 7cm,scale=1]{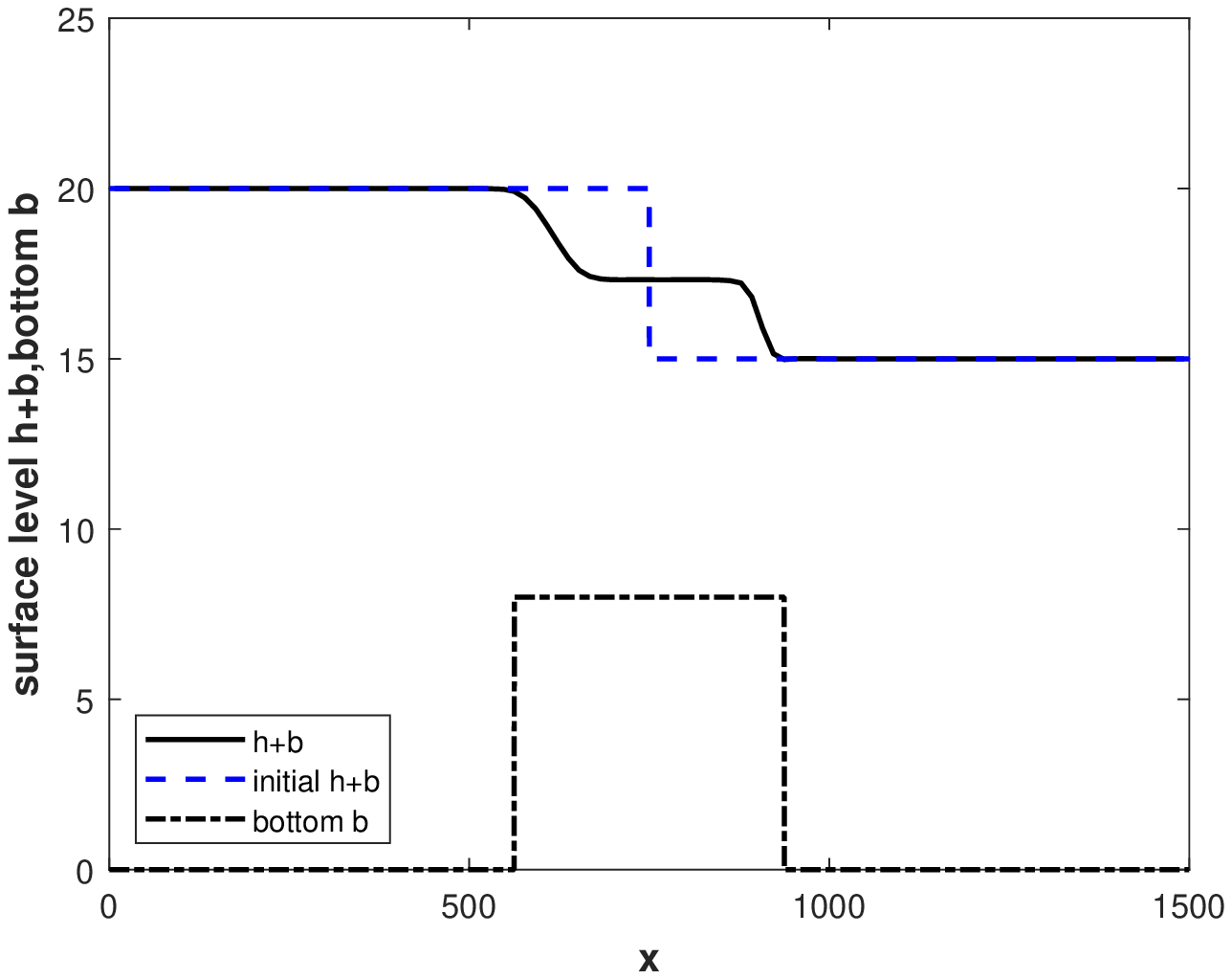}
  }
  \subfigure[surface level $h+b$]{
  \centering
     \includegraphics[width= 7cm,scale=1]{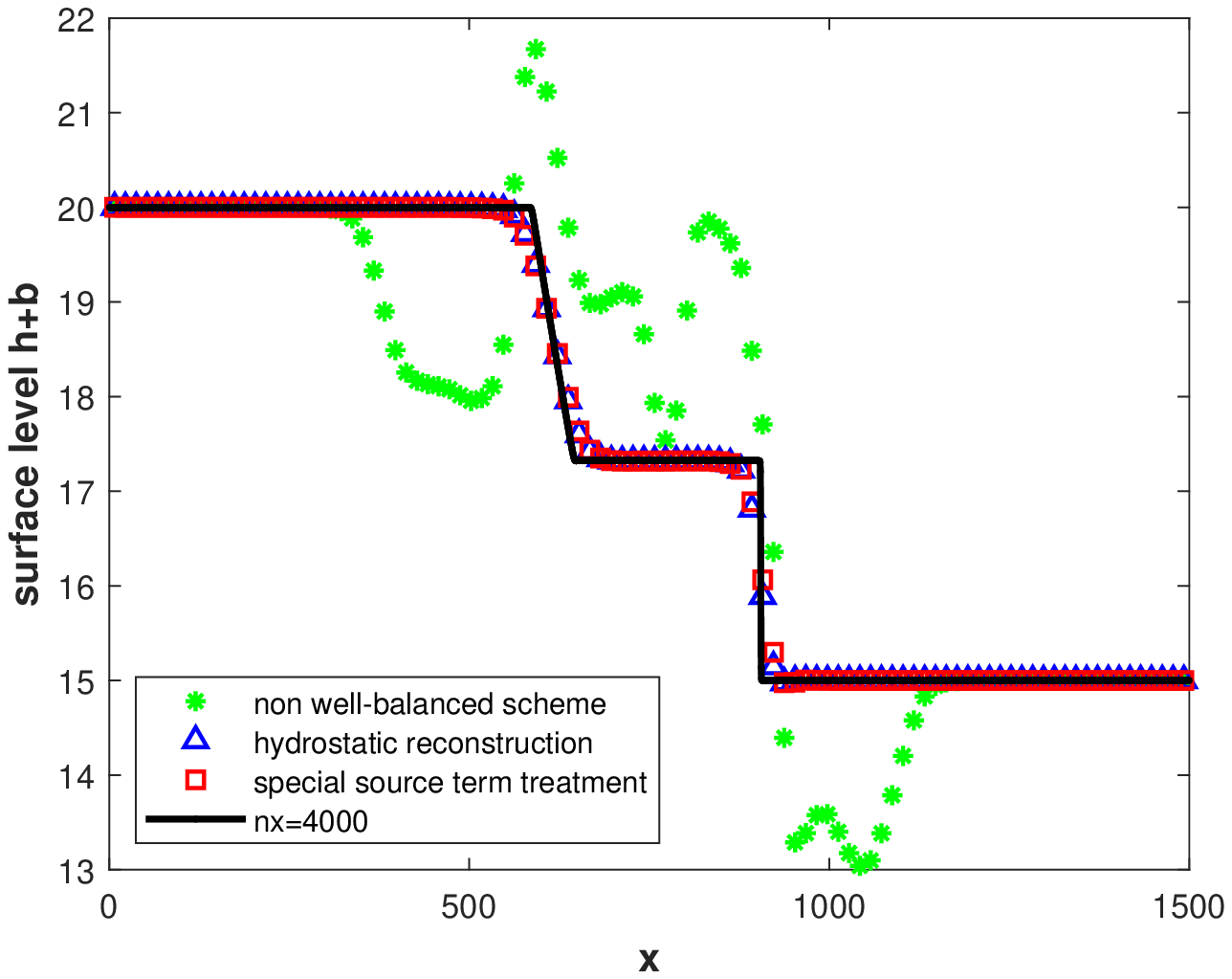}
  }
  \caption{Example \ref{bump1D}: ALE-WENO hybrid schemes  with 100 moving cells at time $t=15$. Left: the numerical solution of the surface level, plotted with the initial condition and the bottom topography; right: the surface level using 4000 uniform grid cells for comparison.}\label{WENOh151D}
\end{figure}

\begin{figure}[htb]
  \centering
  \subfigure[initial condition and bottom topography]{
  \centering
     \includegraphics[width= 7cm,scale=1]{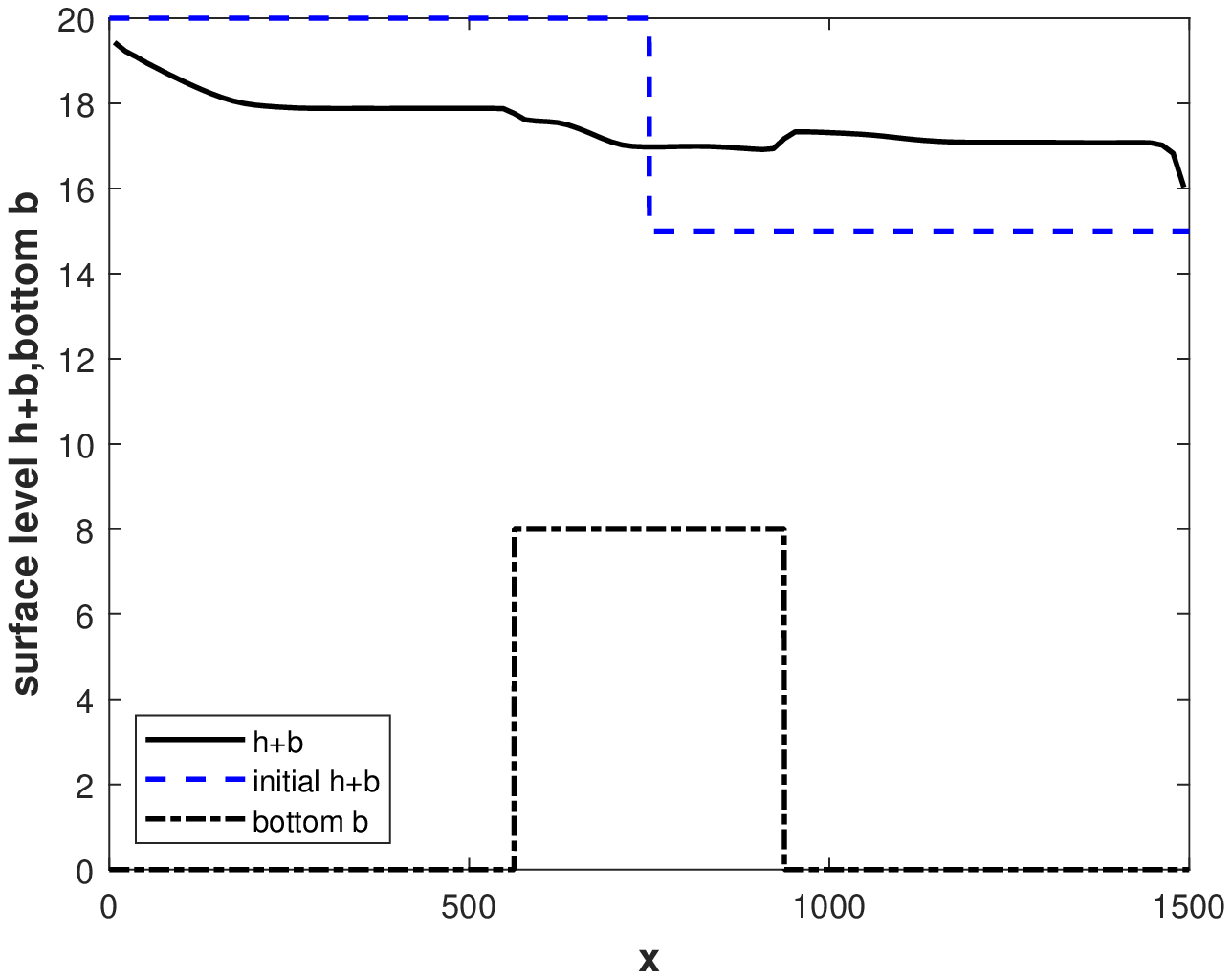}
  }
  \subfigure[surface level $h+b$]{
  \centering
     \includegraphics[width= 7cm,scale=1]{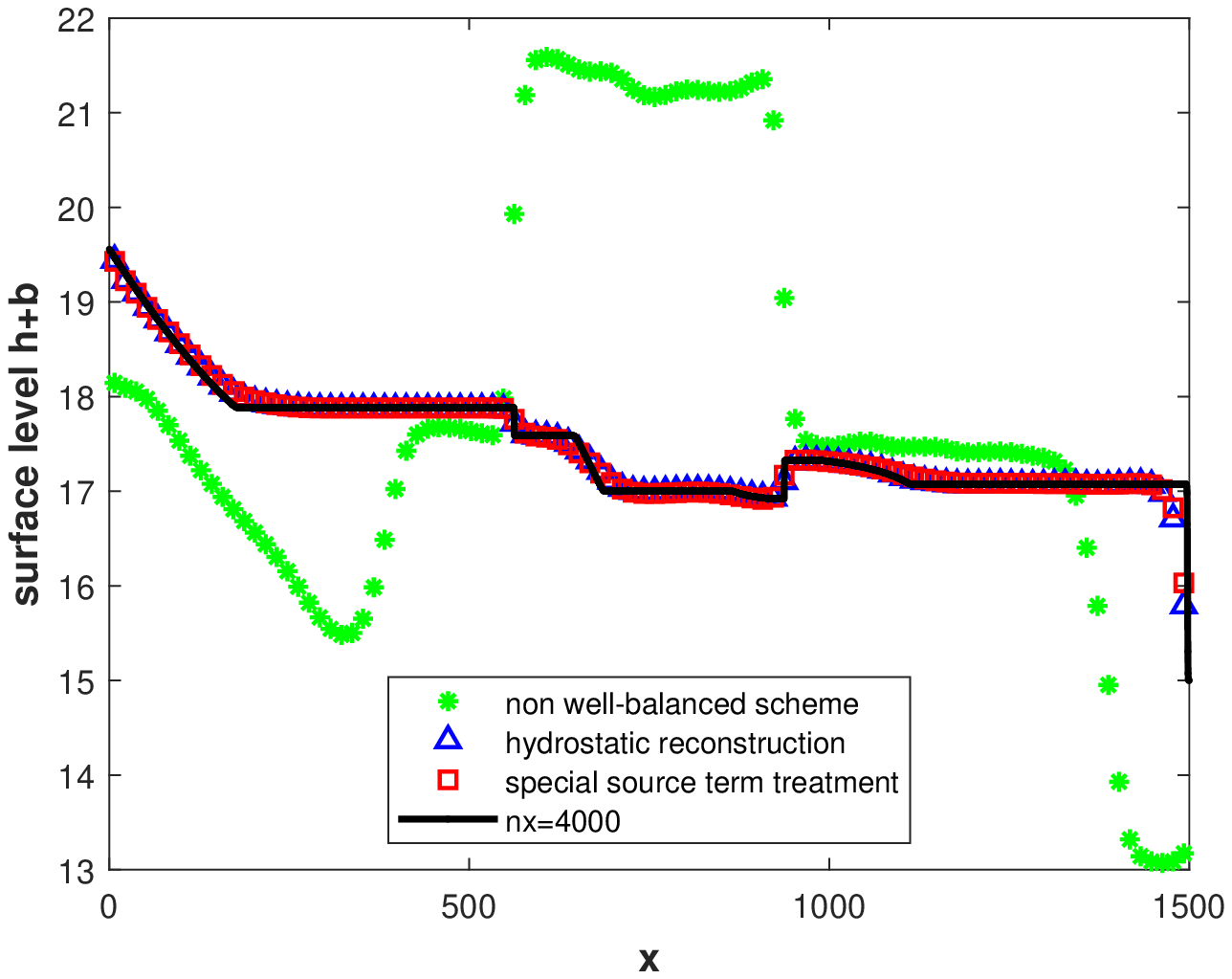}
  }
  \caption{Example \ref{bump1D}: ALE-WENO hybrid schemes  with 100 moving cells at time $t=60$. Left: the numerical solution of the surface level, plotted with the initial condition and the bottom topography; right: the surface level using  4000 uniform grid cells  for comparison.}\label{WENOh601D}
\end{figure}

\subsection{Two-dimensional tests}
\begin{example}{\bf Accuracy test}\label{accuracy2D}
\end{example}
First we test our schemes can achieve high order accuracy for a smooth solution in two dimensions. We have the following bottom function and initial conditions  on the computational domain $[0,1]\times[0,1]$
 \begin{equation}\label{inismooth}
   \begin{aligned}
 b(x,y) = \sin(2\pi x)+\cos(2\pi y),\quad & h(x,y,0)=10+e^{\sin(2\pi x)}\cos(2\pi y), \\
  hu(x,y,0) = \sin(\cos(2\pi x))\sin(2\pi y),\quad & hv(x,y,0) = \cos(2\pi x)\cos(\sin(2\pi y)),
  \end{aligned}
\end{equation}
with periodic boundary conditions. We compute the solution on the moving grid  (\ref{wangge2}) up to $t =0.05$ when the solution is still smooth. Similar to one-dimensional case, since the exact solution for the initial value problem cannot be expressed explicitly, we use the  well-balanced fifth-order finite volume WENO-JS scheme to compute a reference solution on uniform meshes with $ 1600\times 1600$ cells. Tables \ref{WENO5f2D}-\ref{WENO5s2D} show the $L^1$ errors and numerical orders of accuracy for the surface level $h+b$ and the discharge $hu$, $hv$ with three   ALE-WENO hybrid schemes. Herein the CFL number is chosen according to the values in the tables, which guarantee the spatial errors dominate. The expected fifth-order accuracy can be seen on moving meshes in two-dimensional.

\begin{table}[htb]
  \centering
  \caption{Example \ref{accuracy2D}: $L^1$ errors for cell averages and numerical orders of accuracy by the non well-balanced ALE-WENO hybrid scheme for initial condition (\ref{inismooth}), $t=0.05$.}\label{WENO5f2D}
  \begin{tabular}{ c c c c c c c c}
  \toprule
   \multirow{2}{*} {$N_x\times N_y$} &\multirow{2}{*} {CFL}& \multicolumn{2}{c}{$h+b$}&\multicolumn{2}{c}{$ hu $}&\multicolumn{2}{c}{$ hv $}\\
   \cmidrule(lr){3-4} \cmidrule(lr){5-6}\cmidrule(lr){7-8}
   ~ & ~&$L^1$ error &order&$ L^1 $error&order&$ L^1$ error&order\\
   \midrule
   25 $\times$ 25 &0.6& 6.55E{-03} &    -- & 2.13E{-02} &   -- & 5.87E{-02} &   --\\
   50 $\times$ 50 &0.6& 1.09E{-03} &  2.59 & 2.04E{-03} & 3.39 & 8.39E{-03} & 2.81\\
  100 $\times$ 100&0.5& 9.22E{-05} &  3.56 & 1.75E{-04} & 3.54 & 7.24E{-04} & 3.54\\
  200 $\times$ 200&0.4& 4.11E{-06} &  4.49 & 6.76E{-06} & 4.70 & 3.26E{-05} & 4.47\\
  400 $\times$ 400&0.3& 1.57E{-07} &  4.71 & 2.93E{-07} & 4.53 & 1.24E{-06} & 4.72\\
  800 $\times$ 800&0.2& 4.81E{-09} &  5.03 & 1.04E{-08} & 4.81 & 3.96E{-08} & 4.97\\
  \bottomrule
  \end{tabular}
\end{table}

\begin{table}[htb]
  \centering
  \caption{Example \ref{accuracy2D}: $L^1$ errors for cell averages and numerical orders of accuracy by the ALE-WENO hybrid scheme with hydrostatic reconstruction for initial condition (\ref{inismooth}), $t=0.05$.}\label{WENO5h2D}
  \begin{tabular}{ c c c c c c c c}
  \toprule
   \multirow{2}{*} {$N_x\times N_y$} &\multirow{2}{*} {CFL}&  \multicolumn{2}{c}{$h+b$}&\multicolumn{2}{c}{$ hu $}&\multicolumn{2}{c}{$ hv $}\\
   \cmidrule(lr){3-4} \cmidrule(lr){5-6}\cmidrule(lr){7-8}
   ~ & ~&$L^1$ error &order&$ L^1 $error&order&$ L^1$ error&order\\
   \midrule
    25 $\times$ 25 &0.6& 6.47E{-03} &    -- &  2.24E{-02} &   -- & 5.64E{-02} &   -- \\
    50 $\times$ 50 &0.6& 1.08E{-03} &  2.58 &  2.07E{-03} & 3.44 & 8.35E{-03} & 2.76 \\
   100 $\times$ 100&0.5& 9.21E{-05} &  3.56 &  1.76E{-04} & 3.56 & 7.24E{-04} & 3.53 \\
   200 $\times$ 200&0.4& 4.11E{-06} &  4.49 &  6.76E{-06} & 4.70 & 3.27E{-05} & 4.47 \\
   400 $\times$ 400&0.3& 1.56E{-07} &  4.71 &  2.93E{-07} & 4.53 & 1.24E{-06} & 4.72 \\
   800 $\times$ 800&0.2& 4.75E{-09} &  5.04 &  1.04E{-08} & 4.82 & 3.95E{-08} & 4.97 \\
  \bottomrule
  \end{tabular}
\end{table}

\begin{table}[htb]
  \centering
  \caption{Example \ref{accuracy2D}: $L^1$ errors for cell averages and numerical orders of accuracy by the ALE-WENO hybrid scheme with special source term treatment for initial condition (\ref{inismooth}), $t=0.05$.}\label{WENO5s2D}
  \begin{tabular}{ c c c c c c c c}
  \toprule
   \multirow{2}{*} {$N_x\times N_y$} &\multirow{2}{*} {CFL}& \multicolumn{2}{c}{$h+b$}&\multicolumn{2}{c}{$ hu $}&\multicolumn{2}{c}{$ hv $}\\
   \cmidrule(lr){3-4} \cmidrule(lr){5-6}\cmidrule(lr){7-8}
   ~ & ~ &$L^1$ error &order&$ L^1 $error&order&$ L^1$ error&order\\
   \midrule
    25 $\times$ 25 &0.6& 6.52E{-03} &    -- & 2.23E{-02} &   -- &  5.65E{-02} &   --\\
    50 $\times$ 50 &0.6& 1.09E{-03} &  2.58 & 2.06E{-03} & 3.44 &  8.36E{-03} & 2.76\\
   100 $\times$ 100&0.5& 8.50E{-05} &  3.68 & 1.33E{-04} & 3.95 &  6.75E{-04} & 3.63\\
   200 $\times$ 200&0.4& 4.11E{-06} &  4.37 & 6.75E{-06} & 4.30 &  3.27E{-05} & 4.37\\
   400 $\times$ 400&0.3& 1.57E{-07} &  4.71 & 2.93E{-07} & 4.53 &  1.25E{-06} & 4.71\\
   800 $\times$ 800&0.2& 4.79E{-09} &  5.04 & 1.04E{-08} & 4.81 &  3.98E{-08} & 4.97\\
  \bottomrule
  \end{tabular}
\end{table}

\begin{example}{\bf A quasi-stationary vortex}\label{vortex2D}
\end{example}
We test the same example with a flat bottom topography $b(x,y,0)=0$ in \cite{beisiegel2021metrics}. Let the computational domain be $[0,4]\times[0,2]$, a vortex with its radius $r_m = 0.45$ around $c=(1,1)^T$ has a tangential velocity defined by
$$ v_r(r) =  \left\{\begin{array}{lll}
    v_{max}\dfrac{s\cdot r}{r_m^2 - r^2}\sqrt{2\exp\left(\dfrac{1}{r^2 - r_m^2} \right)},&\text{for} \ 0\leq r<r_m, \\
    0,& \text{otherwise},
    \end{array}\right.$$
where $r = \sqrt{(x-1)^2+(y-1)^2}$ is the radius distance from  $c$ and $v_{max}=0.5$ is the maximum tangential velocity,  the scaling factor is defined as
\begin{align*}
 s = \dfrac{|r_{vm}^2 - r_m^2|}{r_{vm}\sqrt{2\exp(1/(r_{vm}^2 - r_m^2))}}, \ \text{where}\ \  r_{vm} = \dfrac{1}{2}\sqrt{-2+2\sqrt{1+4r_m^4}}.
\end{align*}
We choose the following initial conditions
 \begin{equation}\label{vo2}
\begin{aligned}
 &h(x,y,0) = \left\{\begin{array}{lll}
    1 - \dfrac{v_{max}^2 s^2}{g} \exp\left(\dfrac{1}{r^2 - r_m^2} \right),& \text{for} \ 0\leq r<r_m, \\
    1 ,& \text{otherwise},
    \end{array}\right.\\
    & u(x,y,0) = 1- v_r(r)\dfrac{y-2}{r}, v(x,y,0) =  - v_r(r)\dfrac{x-1}{r},
\end{aligned}
\end{equation}
simulating  up to $t=2$  with periodic boundary condition. The analytical solution of this test is that the vortex moves from left to right with the velocity $1m/s$. Therefore we resolve the quasi-stationary vortex problem  on  a special moving mesh involving $110\times70$ cell grids at the beginning
\begin{equation}\label{gridbinx}
\begin{aligned}
  &x_{i+\frac{1}{2}}(0)=\left\{ \begin{array}{lcl}
                0.05 i,&& 0\leqslant i \leqslant 10 , \\
                0.5+0.02(i-10), &&11 \leqslant i \leqslant 60, \\
                1.5+0.05(i-1.5),&&61 \leqslant i \leqslant 110,
                 \end{array} \right.
   \end{aligned}
\end{equation}
\begin{equation}\label{gridbiny}
\begin{aligned}
  &y_{j+\frac{1}{2}}(0) =\left\{ \begin{array}{lcl}
                0.05 j,&& 0\leqslant j \leqslant 10 , \\
                0.5+0.02(j-10), &&11 \leqslant j \leqslant 60, \\
                1.5+0.05(j-1.5),&&61 \leqslant j \leqslant 70,
                 \end{array} \right.
\end{aligned}
\end{equation}
and the mesh moves with the velocity $\boldsymbol w(x,y,t)=(1,0)^T$, which means
\begin{equation}\label{spemovingmesh}
  x_{i+\frac{1}{2}}(t) = x_{i+\frac{1}{2}}(0) + t, \ y_{j+\frac{1}{2}}(t) = y_{j+\frac{1}{2}}(0).
\end{equation}
Moreover, we compare our non well-balanced ALE-WENO scheme with the WENO  scheme on the static uniform mesh with the same cell number. The $L^1$ and $L^{\infty}$ errors of water height are computed for comparison and listed in Table \ref{2Dvor}. It shows the minor numerical errors of the ALE-WENO scheme under the same number of mesh elements. Additionally, the simulation with $110\times 70$ moving grids has almost the same numerical error as that on $220\times 140$ uniform grids, while the number of elements is $75\%$ less than the uniform ones.

\begin{table}[htb]
  \centering
  \caption{Example \ref{vortex2D}: $L^1$ and $L^{\infty}$ errors of the non well-balanced ALE-WENO scheme on static uniform meshes and  moving meshes for the quasi-stationary vortex problem.}\label{2Dvor}
  \begin{tabular}{c c c c c c c c c c c}
   \toprule
     \multirow{2}{*} {$N_x\times N_y$}&\multicolumn{2}{c}{static uniform mesh}&\multicolumn{2}{c}{special moving mesh}\\
    \cmidrule(lr){2-3} \cmidrule(lr){4-5}
     &{$L^1$ error}&{$L^{\infty}$ error}&{$L^1$ error}&{$L^{\infty}$ error}\\
    \midrule
     110 $\times$ 70 & 1.20E{-03} & 2.85E{-03} & 2.18E{-04} &  5.29E{-04} \\
     220 $\times$ 140& 2.21E{-04} & 4.12E{-04} & 2.96E{-05} &  7.56E{-05} \\
     \bottomrule
  \end{tabular}
\end{table}

\begin{example}{\bf Water drop problem}\label{waterdrop2D}
\end{example}
We apply the non well-balanced ALE-WENO hybrid scheme to illustrate the behavior of the water drop problem, which involves an initially localised impulse that develops into distinct fronts travelling throughout the domain. Gaussian shaped peak initial conditions
\begin{equation}\label{waterd2D}
  \begin{aligned}
    &h(x,y,0) = 1 + 0.1 e^{-1000((x-1)^2+(y-1)^2)},\\
    & hu(x,y,0) = hv(x,y,0) = 0,
  \end{aligned}
 \end{equation}
are considered on domain $[0,2]^2$ with reflective boundary condition over a flat bottom topography $b=0$. The initial and the evolution of water height at different times using the  ALE-WENO  scheme are provided in  Fig. \ref{water32D}, which yields well simulation of our schemes on moving meshes. The well-balanced schemes perform very similar, thus we omit it here.

\begin{figure}[htbp]
  \centering
  \subfigure[water height, $t = 0$]{
  \centering
     \includegraphics[width= 6cm,scale=1]{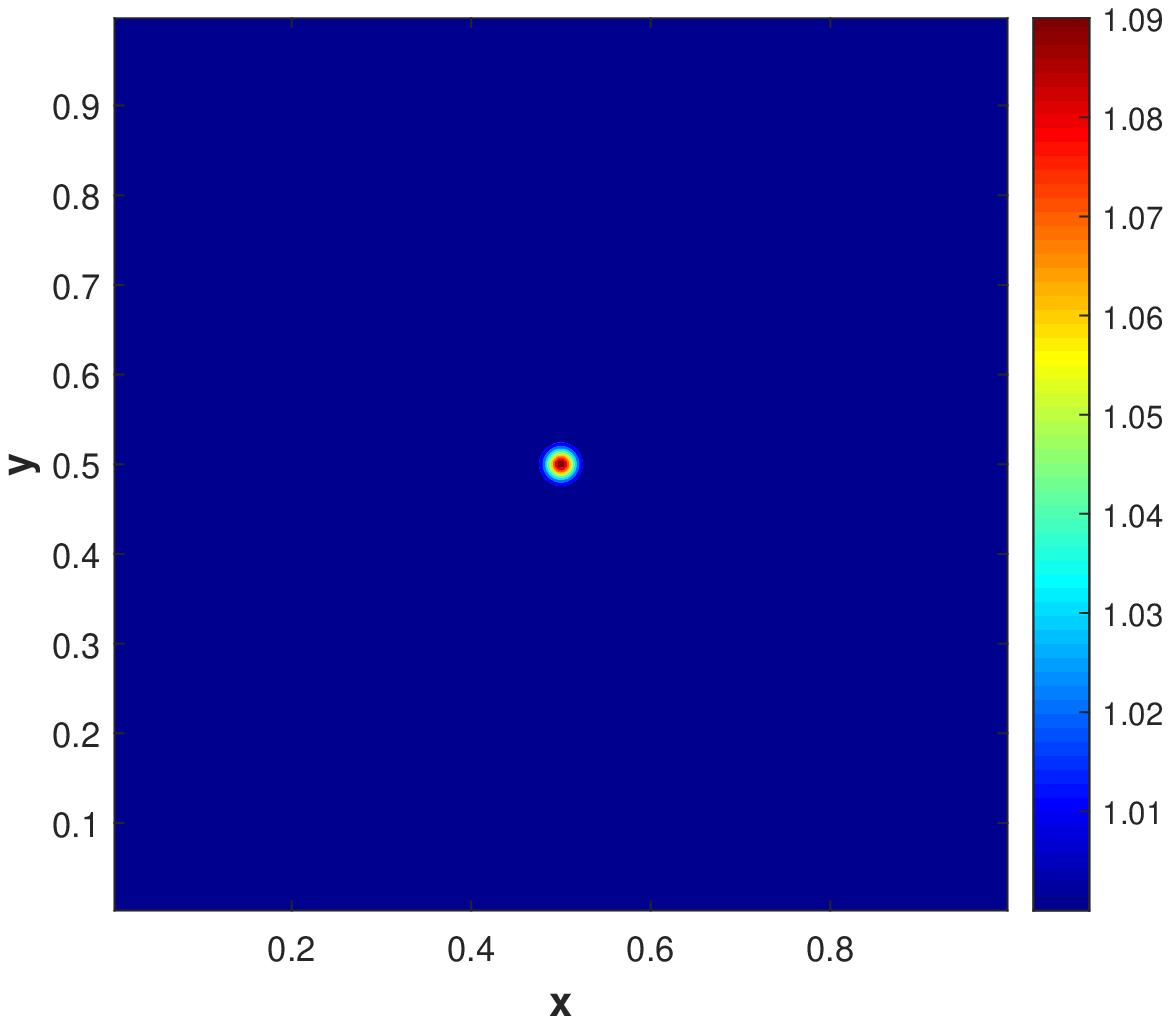}
  }
  \subfigure[water height, $t = 0.35$]{
  \centering
     \includegraphics[width= 6cm,scale=1]{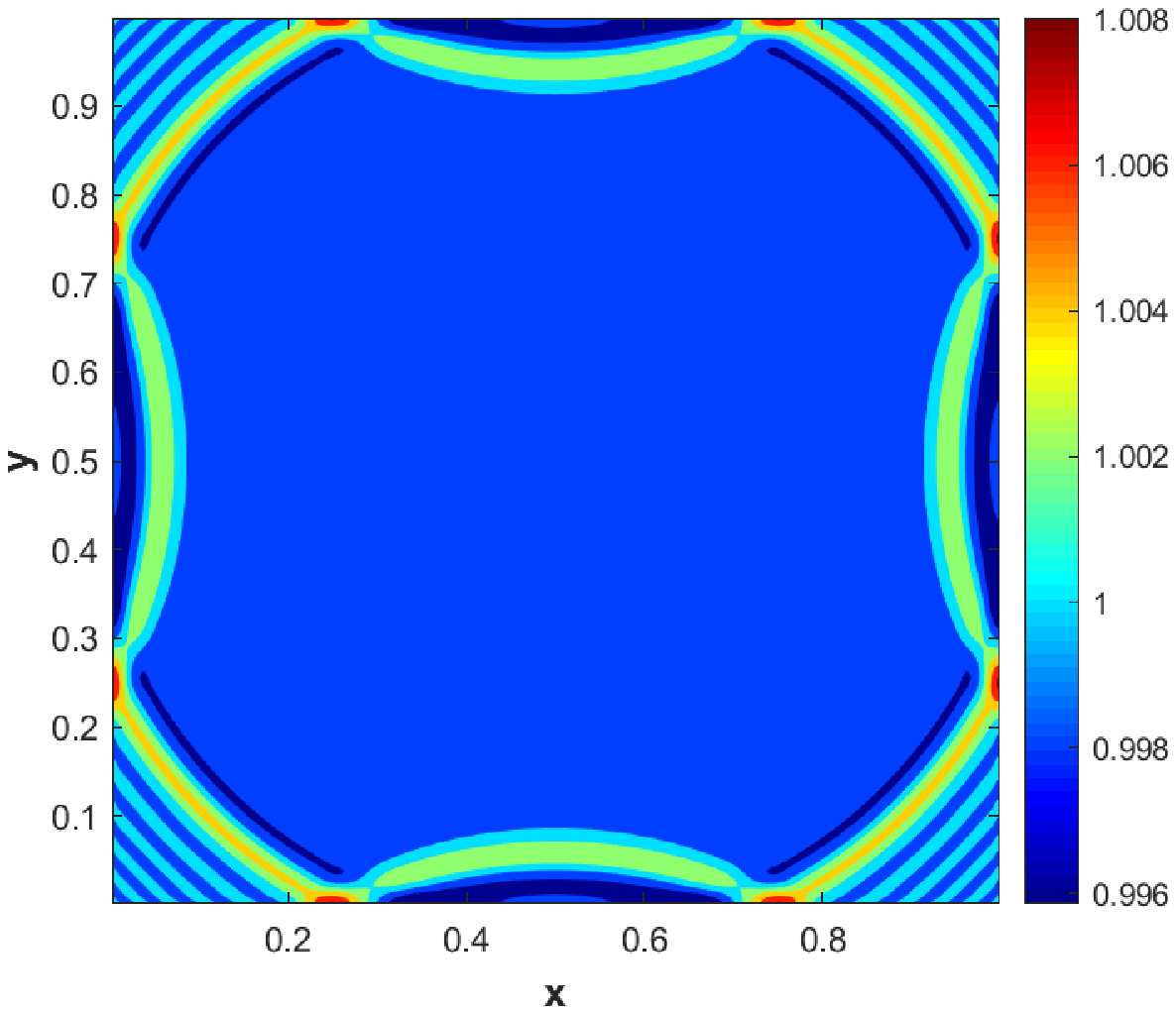}
  }

  \subfigure[water height, $t = 0.6$]{
  \centering
     \includegraphics[width= 6cm,scale=1]{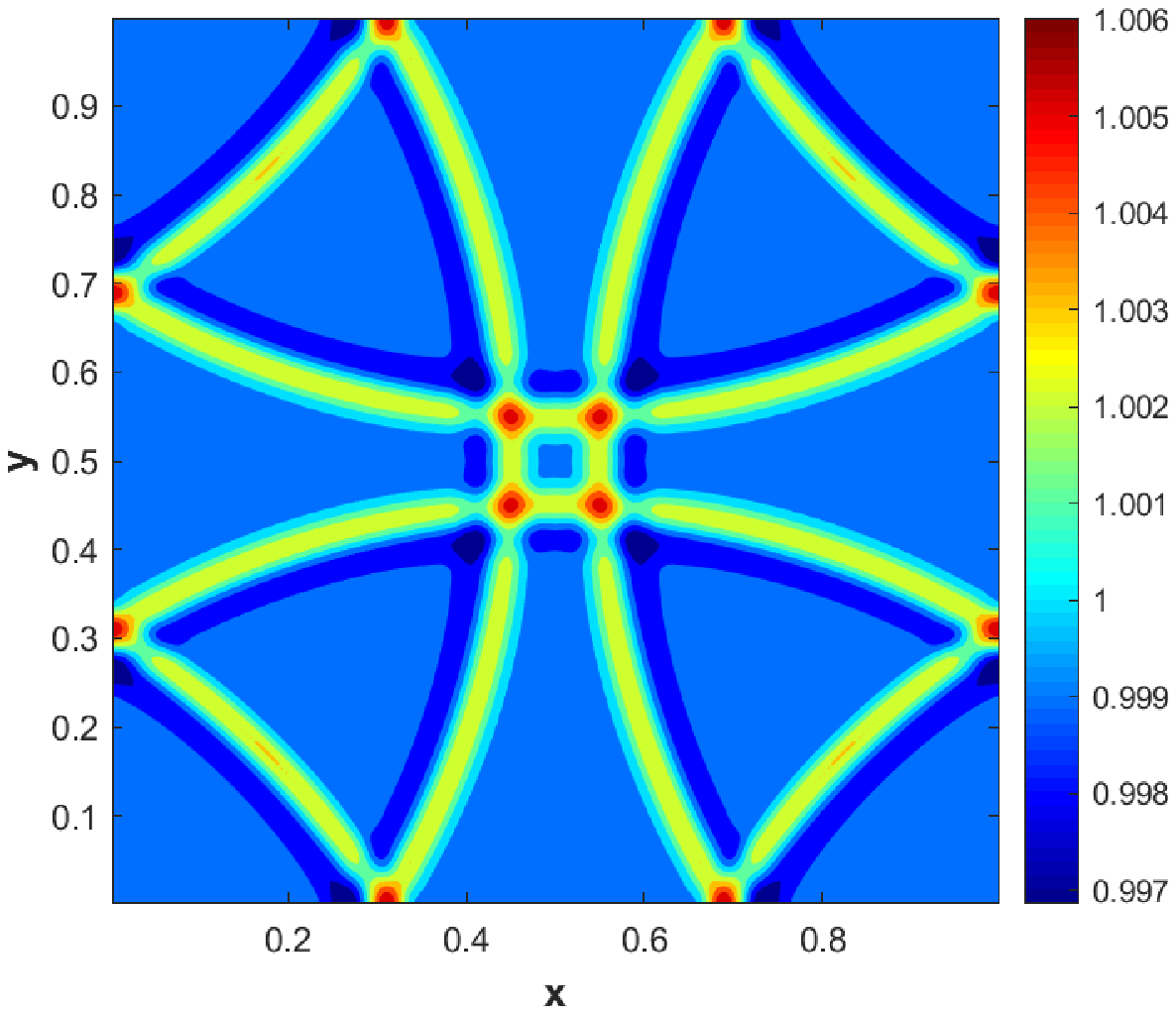}
  }
  \subfigure[water height, $t = 0.75$]{
  \centering
     \includegraphics[width= 6cm,scale=1]{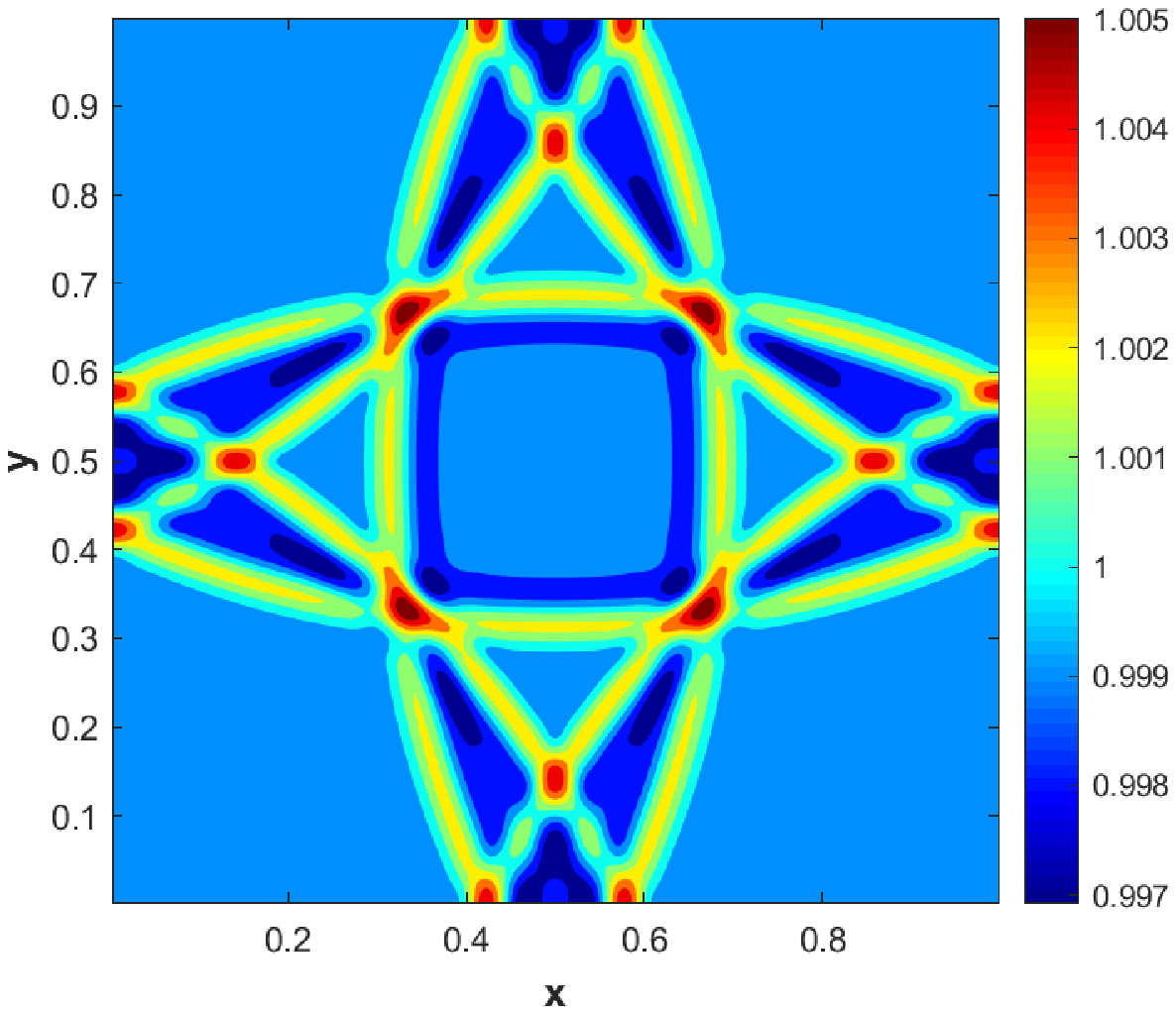}
  }

  \subfigure[water height, $t = 0.8$]{
  \centering
     \includegraphics[width= 6cm,scale=1]{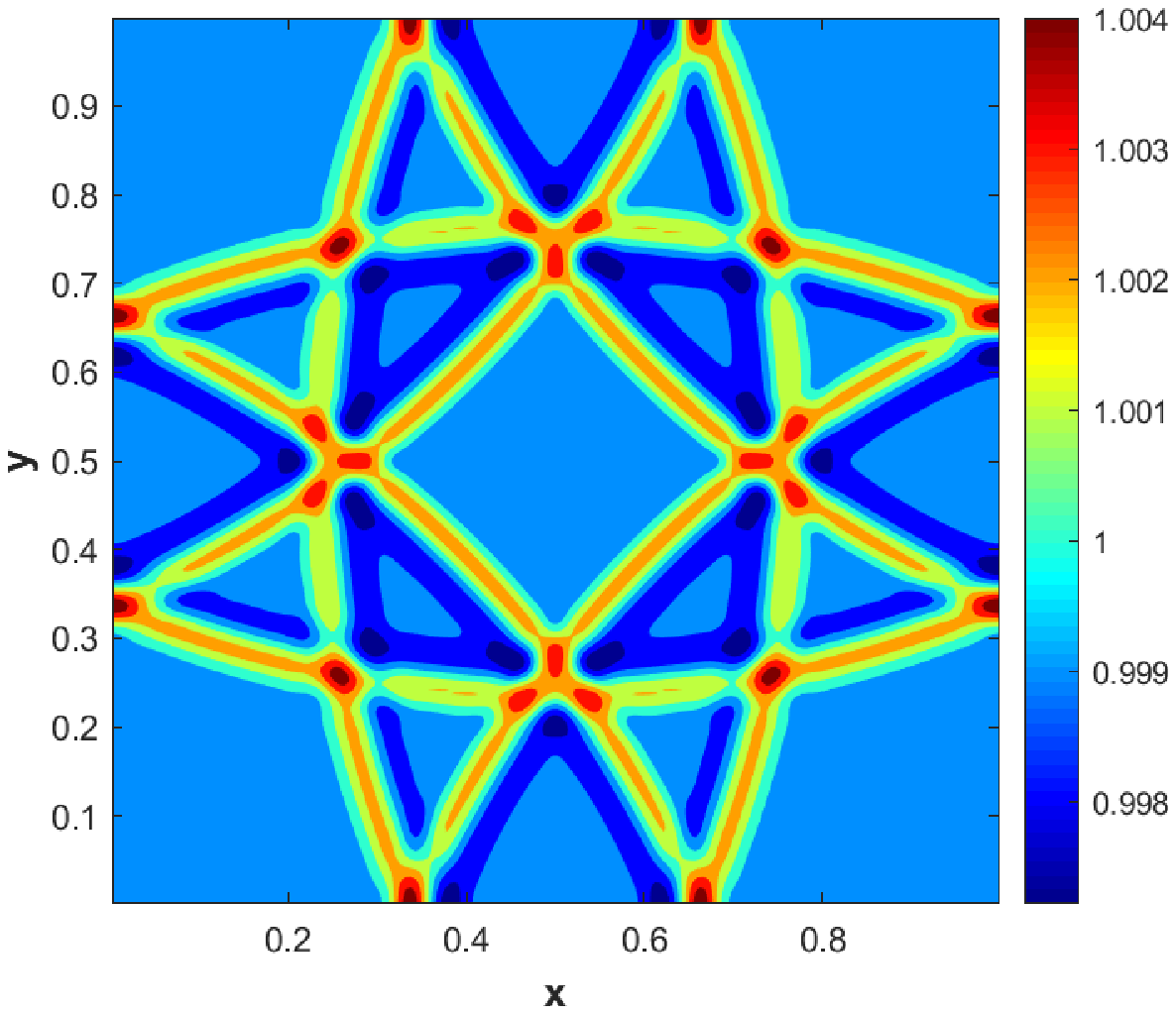}
  }
  \subfigure[water height, $t = 0.9$]{
  \centering
     \includegraphics[width= 6cm,scale=1]{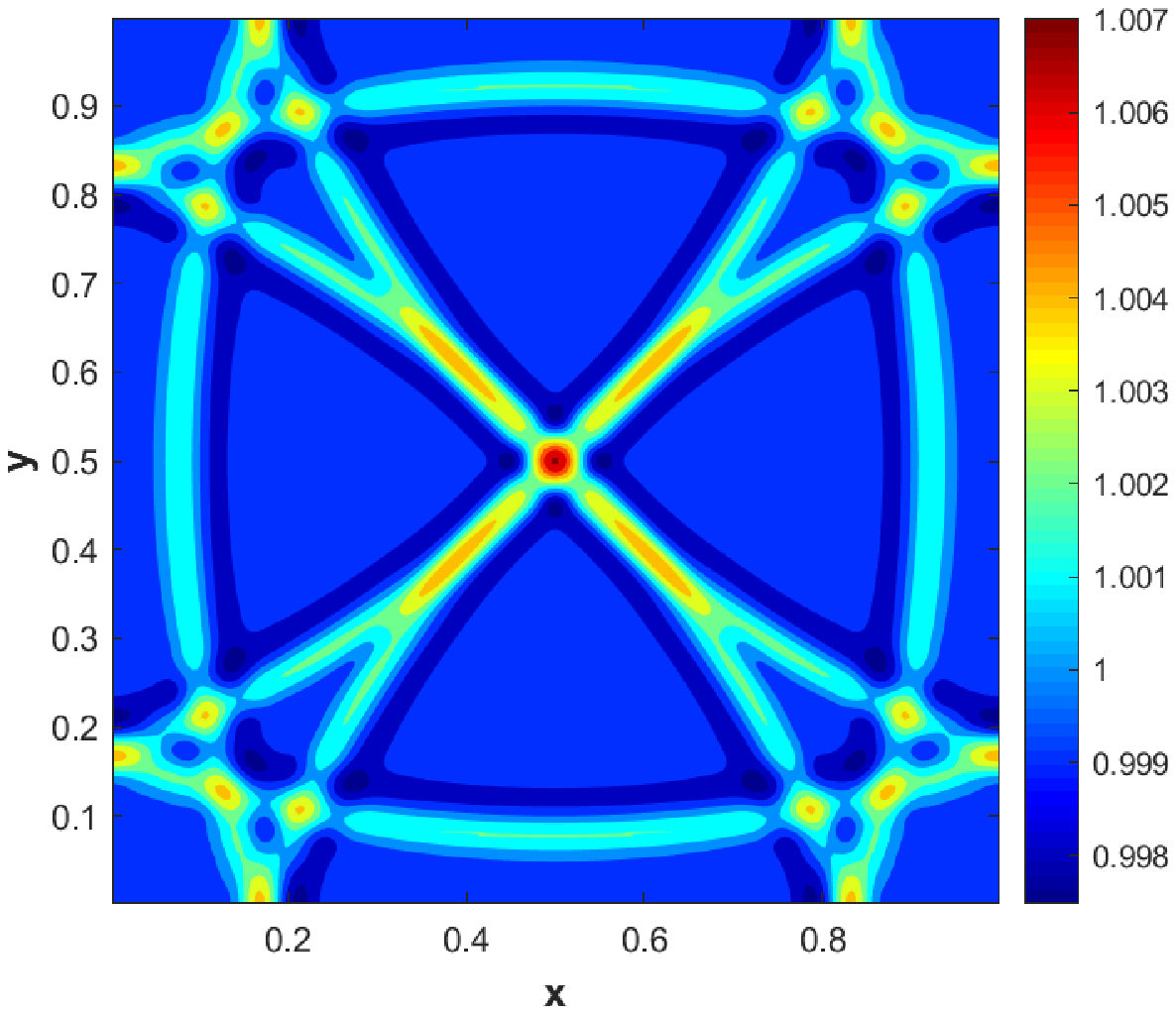}
  }
  \caption{Example \ref{waterdrop2D}: The contours of the  numerical  solutions of water height at times $t=0,0.6,0.75,0.8,0.9$ by the non well-balanced ALE-WENO scheme.}\label{water32D}
\end{figure}

\begin{example}{\bf Circular dam break problem}\label{dambr2D}
\end{example}
To illustrate the efficiency of the  positivity-preserving limiter on moving meshes in two dimensions, we test the complete break of a circular dam over a flat bottom topography $b(x,y)=0$ which is introduced in \cite{xing2013positivity}. The dam is located at $r = \sqrt{x^2+y^2}=60$ on a square domain $[-100,100]^2$.  Initial conditions are given by
 \begin{equation}\label{damini2D}
  \begin{aligned}
    &h(x,y,0) = \left\{\begin{array}{lll}
    10,&\text{if}\ r\leq 60, \\
    0,&\text{otherwise},\\
    \end{array}\right.\ u(x,y,0) = v(x,y,0) = 0,
  \end{aligned}
 \end{equation}
and the CFL number is taken as 0.08. We compute until $t = 1.75$ with the  $200\times 200$ moving mesh, and only show the numerical solutions using the non well-balanced ALE-WENO scheme in Fig. \ref{dambrae2D} due to the very similar numerical performance of other schemes. Table \ref{minwater22} shows the minimum of the water height. We can observe the excellent resolution and non-oscillatory near wet and dry fronts.

\begin{table}[htbp]
  \centering
  \caption{Example \ref{dambr2D}: The minimum of the water height at $t=1.75$ with  initial condition (\ref{damini2D}).}\label{minwater22}
  \begin{tabular}{c c c c }
   \toprule
    \multirow{1}{*} {Scheme} &{non well-balanced scheme}\\
    \midrule
   $\min h$    &   1.00E{-11}\\
     \bottomrule
  \end{tabular}
\end{table}

\begin{figure}[htbp]
  \centering
  \subfigure[3D view of the surface level]{
  \centering
     \includegraphics[width= 6.5cm,scale=1]{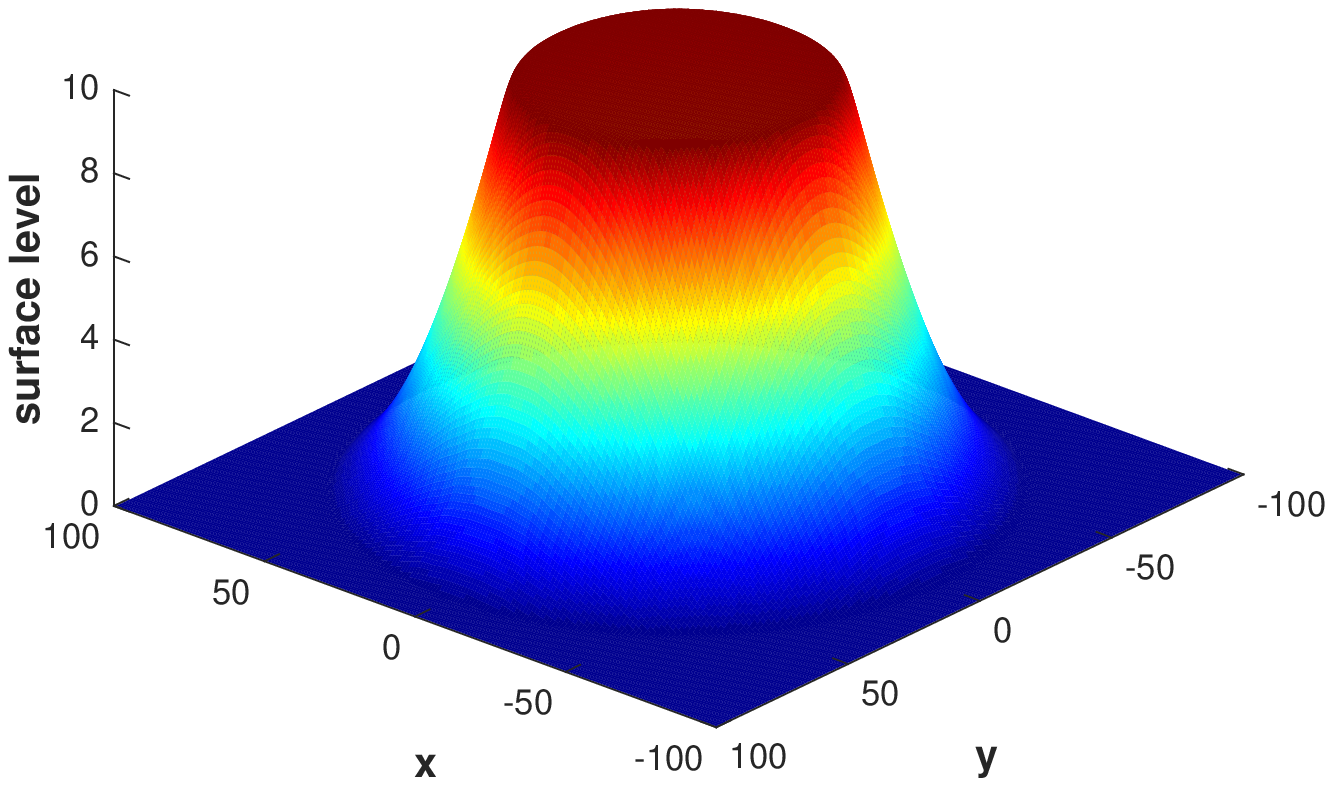}
  }
  \subfigure[the contours of the surface level]{
  \centering
     \includegraphics[width= 6cm,scale=1]{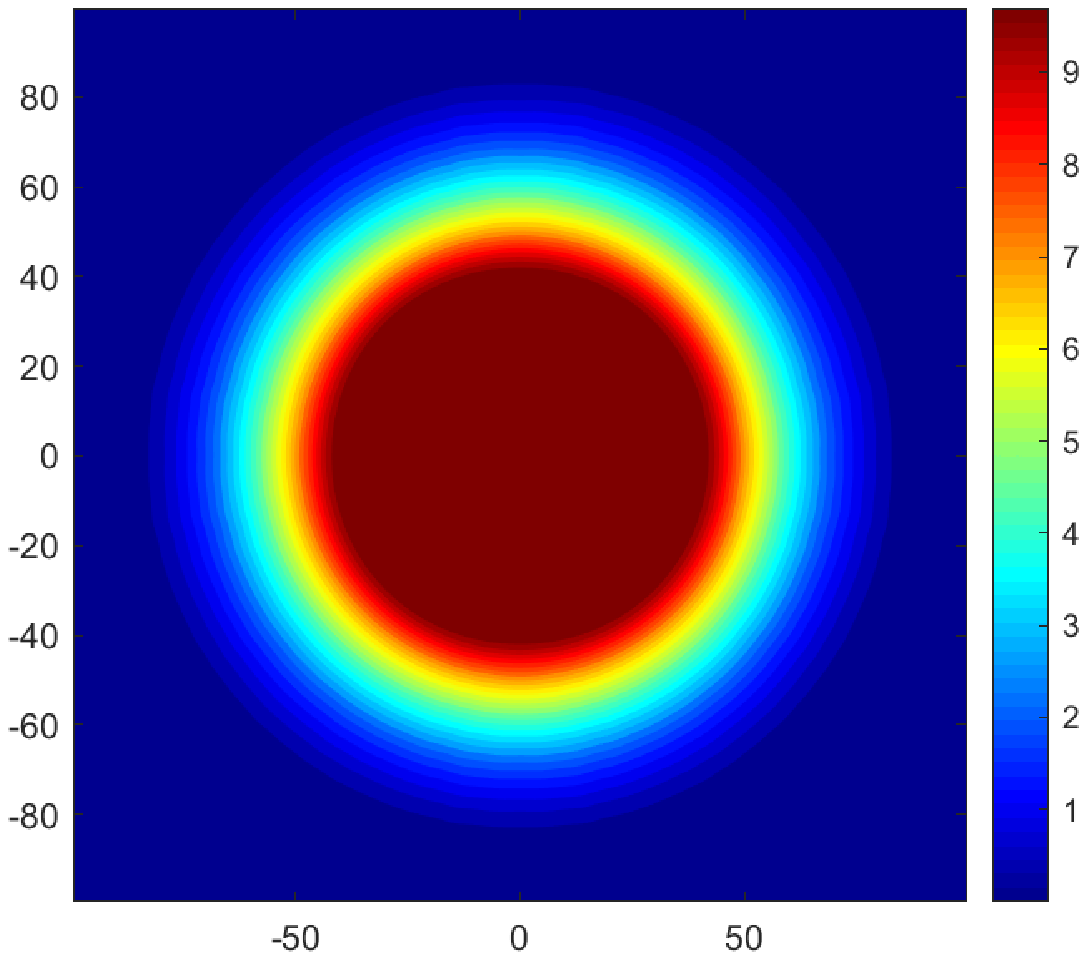}
  }
  \caption{Example \ref{dambr2D}:  Left: 3D view of the surface level; Right: the contours of the  surface level, by the non well-balanced ALE-WENO scheme on the $200\times 200$ moving mesh at time $t=1.75$.}\label{dambrae2D}

\end{figure}

\begin{example}{\bf Long wave resonance in a parabolic basin}\label{parabolic2D}
\end{example}
We take the same example in \cite{vater2019limiter} to address the correct representation of a moving shoreline and test the positivity-preserving property of our schemes with a non-flat bottom topography.
Let the computation domain be $[-4000,4000]^2$, we take a parabolic bottom topography $$b(x,y,0) =\dfrac{h_0}{a^2}(x^2+y^2),$$
and the initial conditions
 \begin{equation}\label{paini2D}
  \begin{aligned}
    &h(x,y,0) = \max\left\{ 0,h_0\left( \dfrac{\sqrt{1-A^2}}{1-A} - \dfrac{x^2+y^2}{a^2}\dfrac{1-A^2}{(1-A)^2} \right) \right\},\
     u(x,y,0) = v(x,y,0) = 0,
  \end{aligned}
 \end{equation}
where $A = \dfrac{a^4-r_0^4}{a^4+r_0^4}$, $h_0 = 1,r_0 = 2000,a=2500.$ The analytic solution is given by
 \begin{equation}\label{paexact2D}
  \begin{aligned}
    &h(x,y,t) = \max\left\{ 0,h_0\left( \dfrac{\sqrt{1-A^2}}{1-A\cos(\kappa t)} - \dfrac{x^2+y^2}{a^2}\dfrac{1-A^2}{(1-A\cos(\kappa t))^2} \right) \right\},\\
    & u(x,y,t) = \dfrac{\kappa A\sin(\kappa t)}{2(1-A\cos(\kappa t))}x,\quad  v(x,y,t) = \dfrac{\kappa A\sin(\kappa t)}{2(1-A\cos(\kappa t))}y,
  \end{aligned}
 \end{equation}
which are periodic with the period $T = {2\pi}/{\kappa}$ and $\kappa = \sqrt{8gh_0}/a$. Here, the CFL number is taken as 0.08 throughout our calculation. We  compute the numerical solutions on moving meshes until time $T$ with $100\times100$ uniform grid cells at the beginning,  seeing the surface level along the line $y=0$ in Fig. \ref{surfac} and the minimum of the water height at different times in Table \ref{minwaterp2D}.  Our numerical solutions using either the non well-balanced ALE-WENO scheme or the well-balanced ALE-WENO hybrid schemes can reach a nice agreement with the exact solution, which indicate the effectiveness of our schemes under the ALE framework.

\begin{figure}[htb]
  \centering
  \subfigure[the surface level at time $t = T/6$]{
  \centering
     \includegraphics[width= 6.5cm,scale=1]{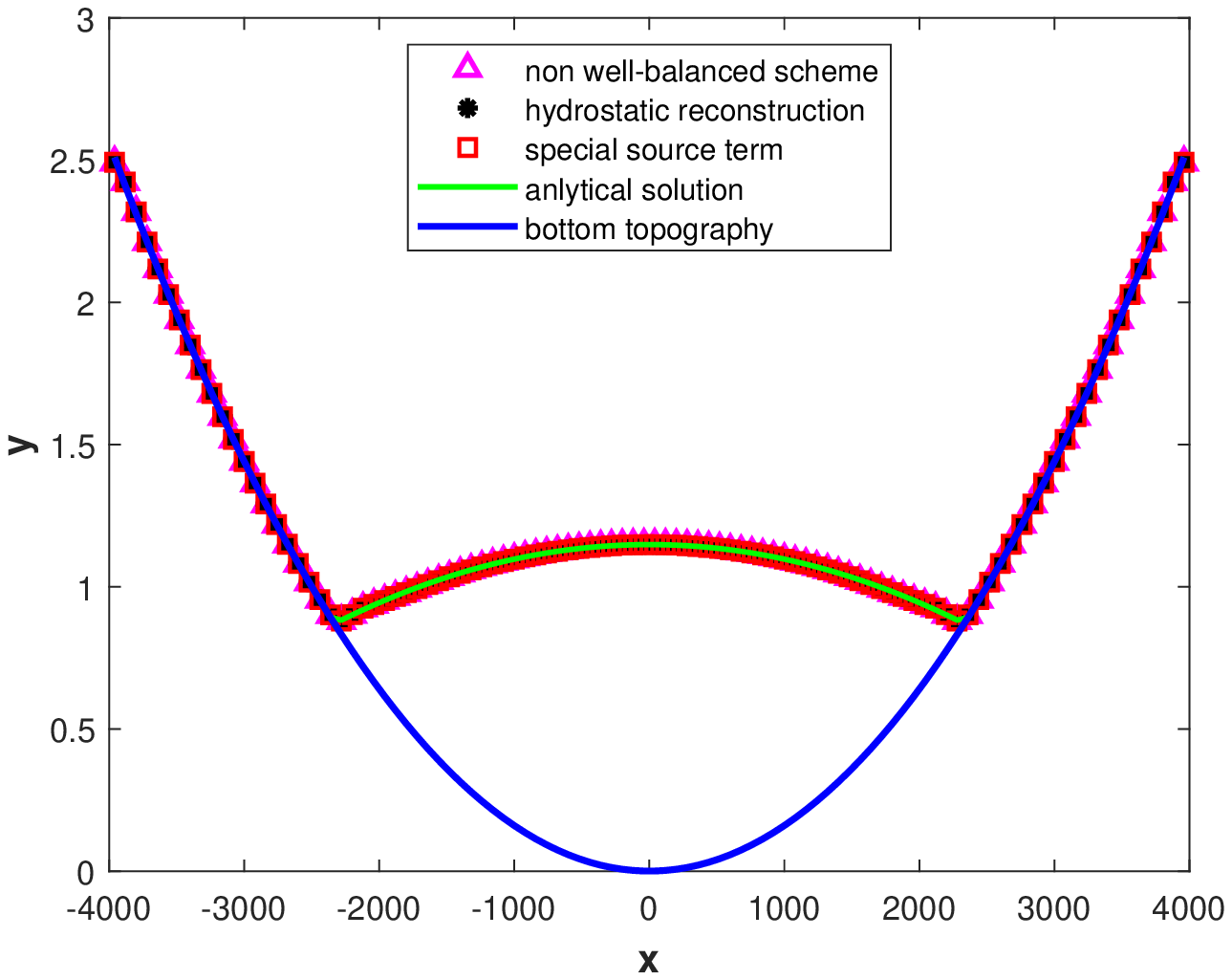}
  }
  \subfigure[the surface level at time $t = T/2$]{
  \centering
     \includegraphics[width= 6.5cm,scale=1]{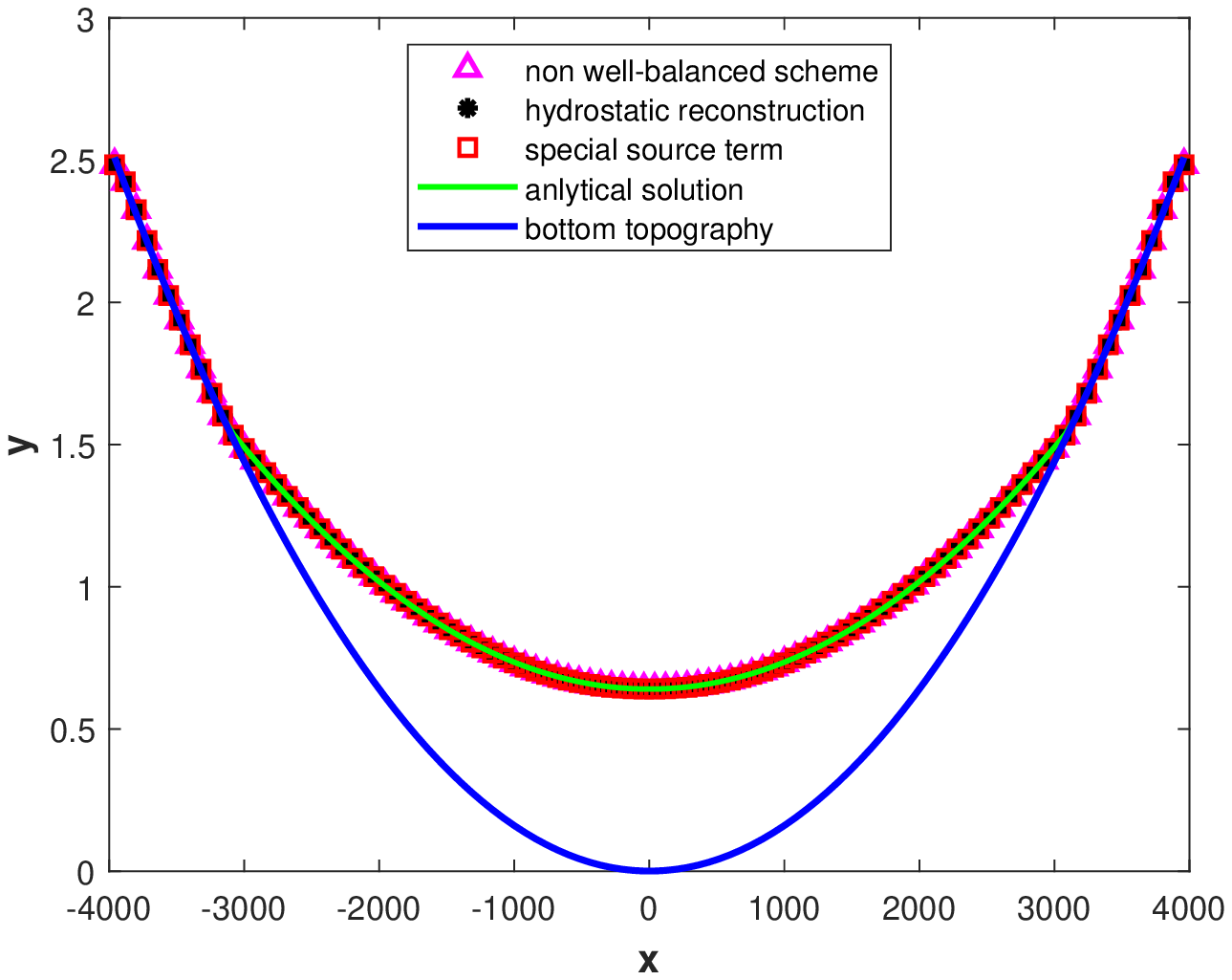}
  }
  \subfigure[the surface level at time $t = 3T/4$]{
  \centering
     \includegraphics[width= 6.5cm,scale=1]{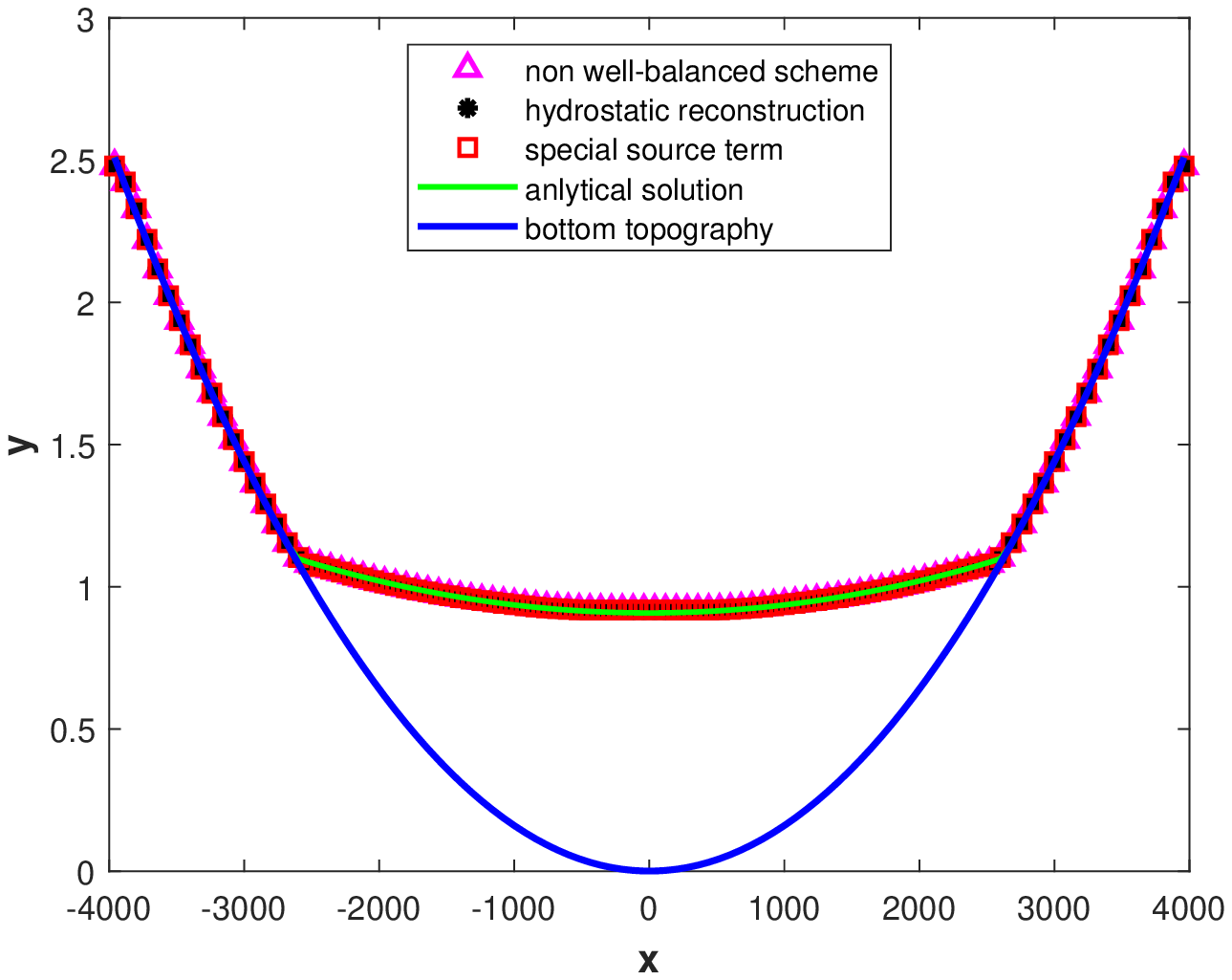}
  }
  \subfigure[the surface level at time $t = T$]{
  \centering
     \includegraphics[width= 6.5cm,scale=1]{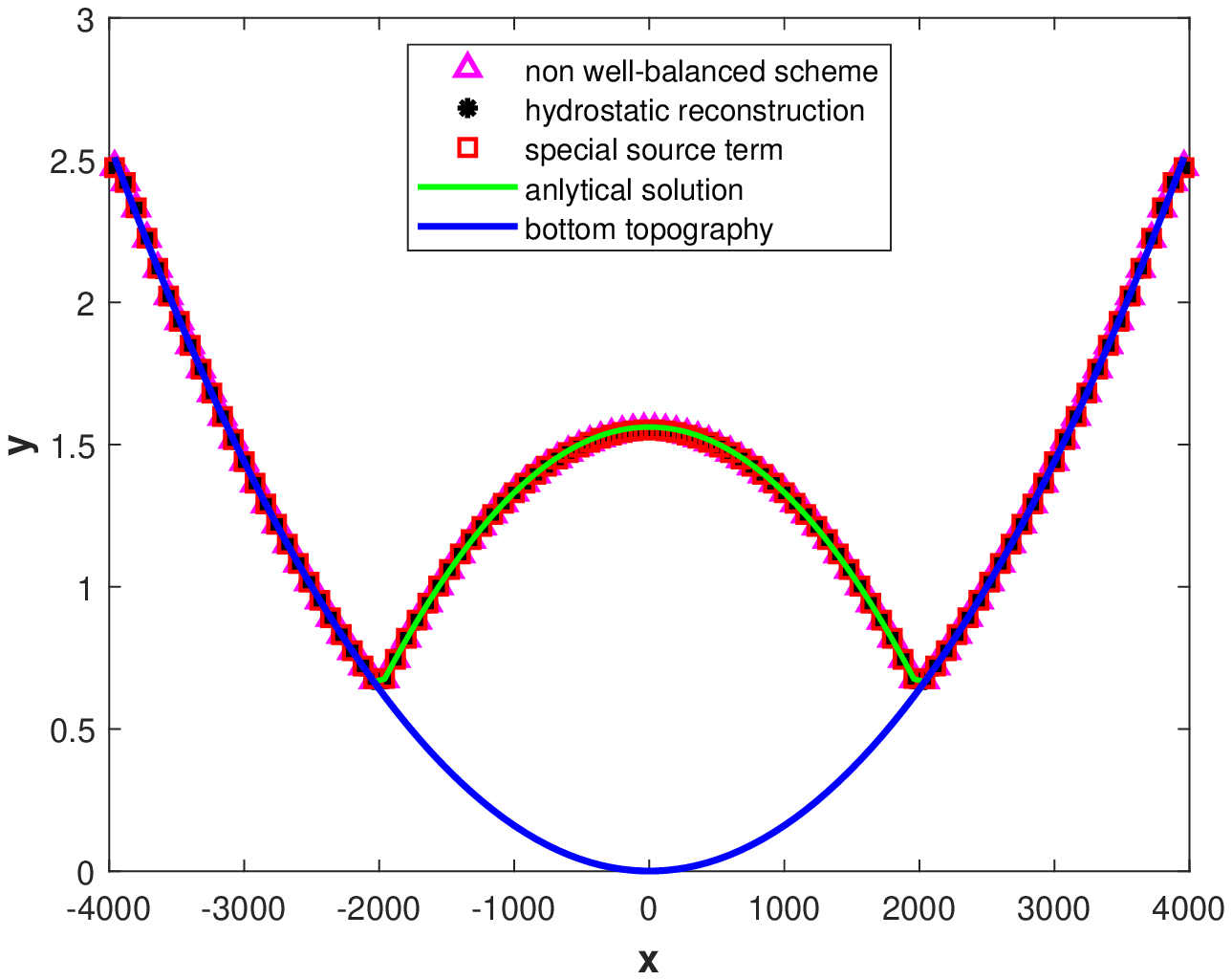}
  }
  \caption{Example \ref{parabolic2D}: Numerical solutions of the surface level along the line $y=0$  at times $T/6,T/2,3T/4,T$ (from top left to bottom right), by ALE-WENO hybrid schemes on the $100\times 100$ moving mesh.}\label{surfac}
\end{figure}

\begin{table}[htb]
  \centering
  \caption{Example \ref{parabolic2D}: The minimum  water height at different times by ALE-WENO hybrid schemes.}\label{minwaterp2D}
  \begin{tabular}{c c c c c c c c c c }
   \toprule
    \multirow{1}{*} {Scheme} &{$t = T/6$}&{$t = T/2$}&{$t = 3T/4$}&{$t = T$}\\
    \midrule
   non well-balanced  scheme   &  1.18E{-11} & 1.04E{-11}&  1.34E{-11}&  1.34E{-11} \\
   hydrostatic reconstruction   &  1.19E{-11} & 1.03E{-11}&  1.16E{-11}&  1.23E{-11} \\
   special source term treatment&  1.63E{-11} & 1.69E{-11}&  1.17E{-11}&  1.34E{-11} \\
     \bottomrule
  \end{tabular}
\end{table}

\begin{example}{\bf Test for the  well-balanced  property}\label{exactC2D}
\end{example}
To verify our developed  well-balanced schemes on moving meshes can  maintain the well-balanced property over a non-flat bottom on an unit square $[0,1]\times[0,1]$, we choose the bottom topography:
\begin{equation}\label{smo2D}
  b(x,y)=0.8e^{-50\left((x-0.5)^2+(y-0.5)^2\right)},
\end{equation}
and the initial conditions
$$h(x,y,0)+b(x,y) = 1,\quad hu(x,y,0)=0, \quad hv(x,y,0)=0.$$
We compute the solution up  to  $t = 0.1$ on moving grids (\ref{wangge2}), using  $100\times 100 $ uniform rectangular mesh at the beginning.
The $L^1$ and $L^{\infty}$ errors at double precision for the surface level $h+b$ and the discharge $hu$, $hv$ are listed in Table \ref{2Dsmooth}. It can also  be clearly seen that the hydrostatic equilibrium state is  exactly preserved in two-dimensional.
\begin{table}[htb]
  \centering
  \caption{Example \ref{exactC2D}: $L^1$ and $L^{\infty}$ errors by the ALE-WENO hybrid schemes for the hydrostatic equilibrium state with a smooth bottom (\ref{smo2D}).}\label{2Dsmooth}
  \begin{tabular}{c c c c c c c c c c c c}
   \toprule
    \multirow{2}{*} {Scheme}&\multicolumn{3}{c}{$L^1$ error}&\multicolumn{3}{c}{$L^{\infty}$ error}\\
    \cmidrule(lr){2-4} \cmidrule(lr){5-7}
    &{$h+b$}&{$hu$}&{$hv$}&{$h+b$}&{$hu$}&{$hv$}\\
    \midrule
     non well-balanced scheme   & 1.27E{-06} &  3.63E{-06} & 3.52E{-06} &  2.01E{-05} & 2.40E{-05} & 2.30E{-05}\\
    hydrostatic reconstruction    & 4.87E{-16} &  1.14E{-15} & 8.71E{-16} &  2.44E{-15} & 5.49E{-15} & 4.27E{-15}\\
    special source term treatment & 4.87E{-16} &  1.13E{-15} & 8.64E{-16} &  2.44E{-15} & 4.73E{-15} & 3.96E{-15}\\
     \bottomrule
  \end{tabular}
\end{table}

\begin{example}{\bf A small perturbation of a two dimensional steady-state water}\label{perturbation2D}
\end{example}
The last classical example we test is given by LeVeque \cite{leveque1998balancing} which shows the capability of  our two  well-balanced ALE-WENO hybrid schemes for the perturbation of the hydrostatic state in two dimensions. Initial conditions are given by
  \begin{align*}
    &h(x,y,0) = \left\{\begin{array}{lll}
    1-b(x,y)+\epsilon, & \text{if} \ \ 0.05\leq x \leq 0.15, \\
    1-b(x,y),&\text{otherwise},\\
    \end{array}\right. \\
    &(hu)(x,y,0)=(hv)(x,y,0)=0,
  \end{align*}
with an elliptical hump bottom topography
$$b(x,y)=0.8e^{-5(x-0.9)^2-50(y-0.5)^2},$$
on the computational domain $[0,2]\times [0,1]$. We compute solutions on 200$\times$100  moving mesh  at different times for $\epsilon= 0.01$ and a smaller perturbation $\epsilon= 0.0001$. The contours of the surface level $h+b$ using the ALE-WENO hybrid scheme with special source term treatment and the non well-balanced ALE-WENO scheme for two different $\epsilon$ are plotted in Figs. \ref{surfacehydro2D} and  \ref{surfacespecial2D} respectively. Results by the ALE-WENO hybrid scheme with hydrostatic reconstruction are very similar, hence we omit it here.
It suggests that our well-balanced schemes perform better than the non well-balanced ALE-WENO scheme. They are capable of achieving high resolution and capturing the complex small features of the flow very well especially for the smaller perturbation.

\begin{figure}[htbp]
  \centering
  {\centering
     \includegraphics[width= 7cm,scale=1]{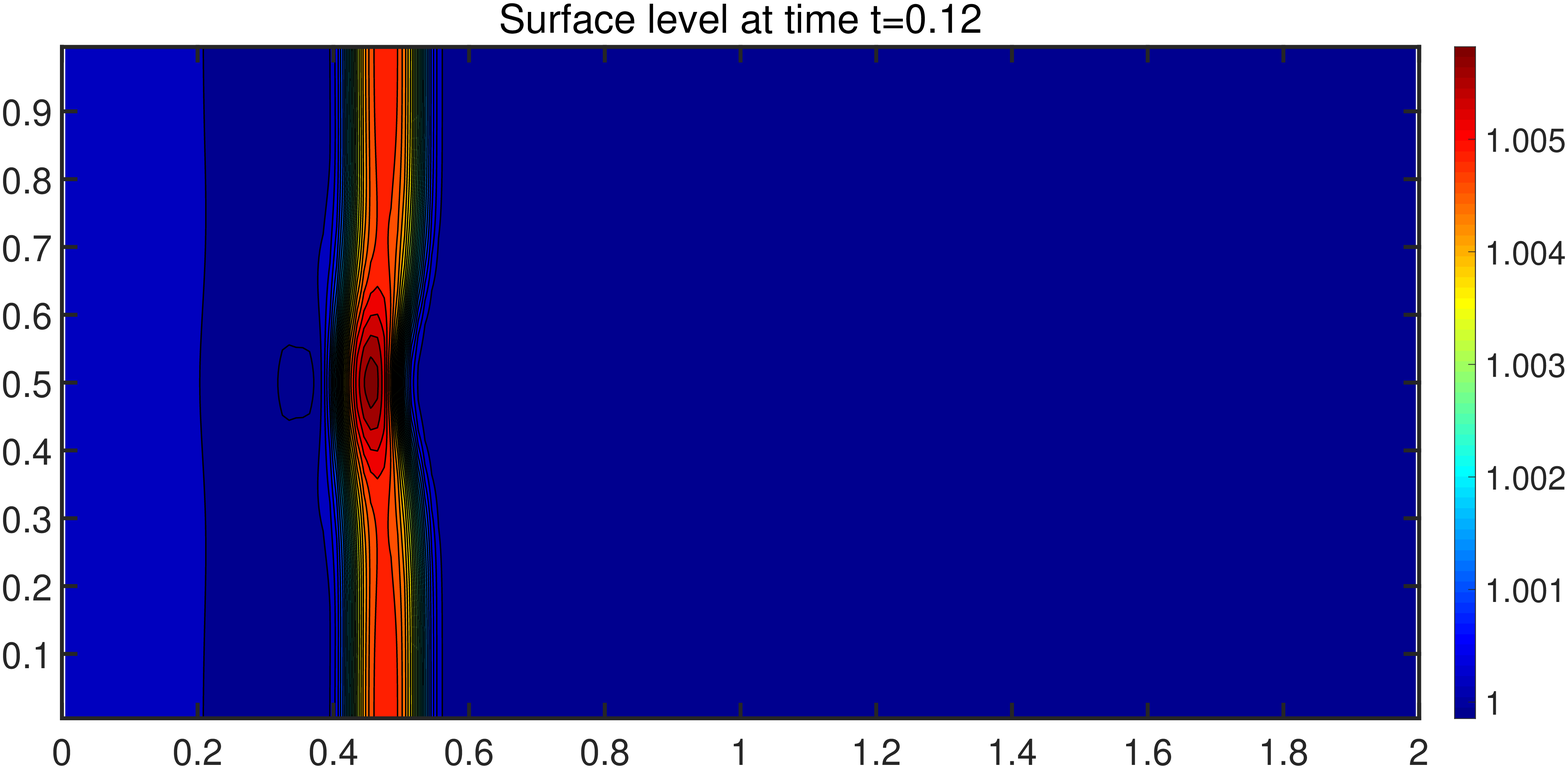}
  }
  {\centering
     \includegraphics[width= 7cm,scale=1]{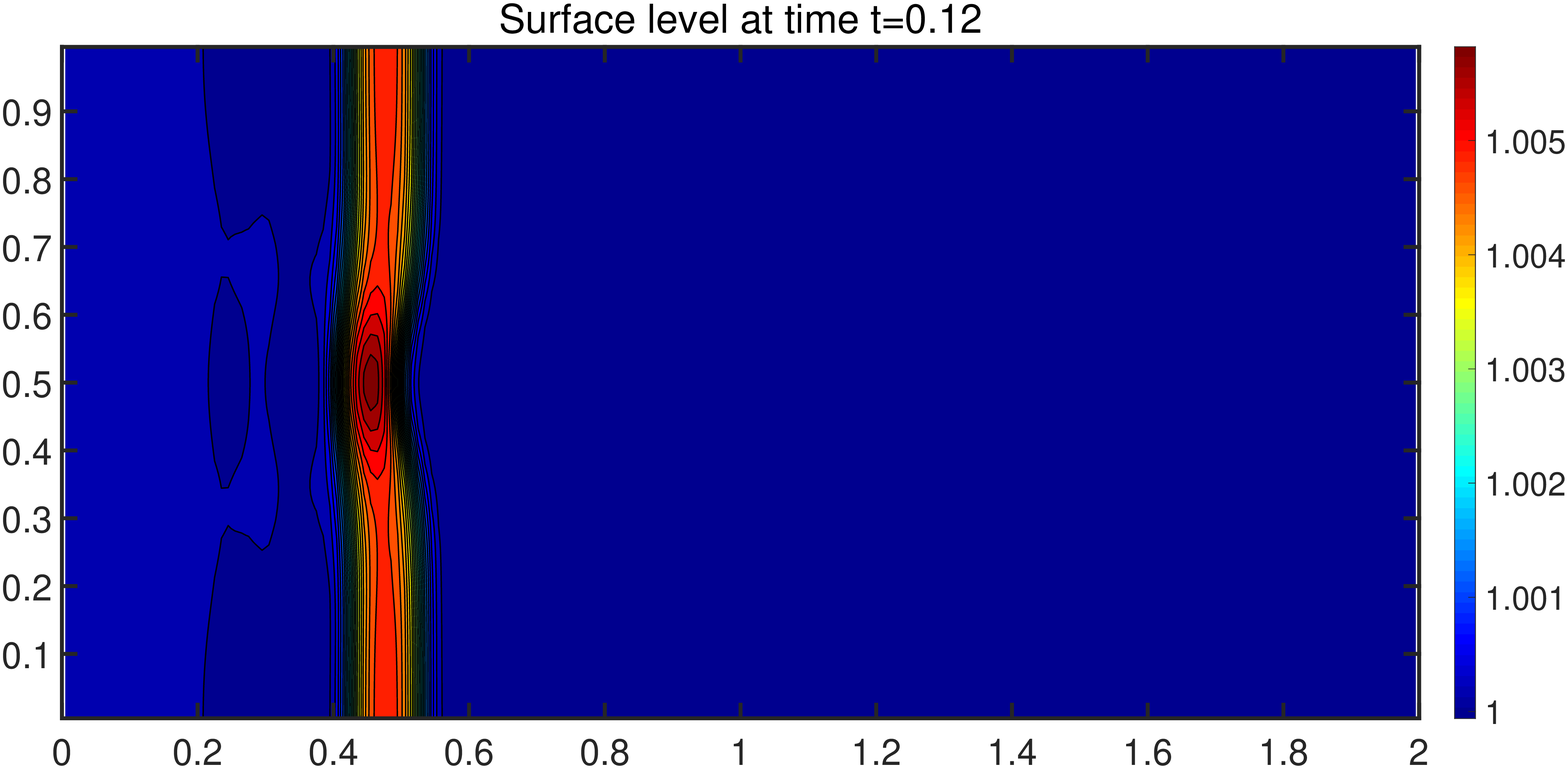}
  }
  {\centering
     \includegraphics[width= 7cm,scale=1]{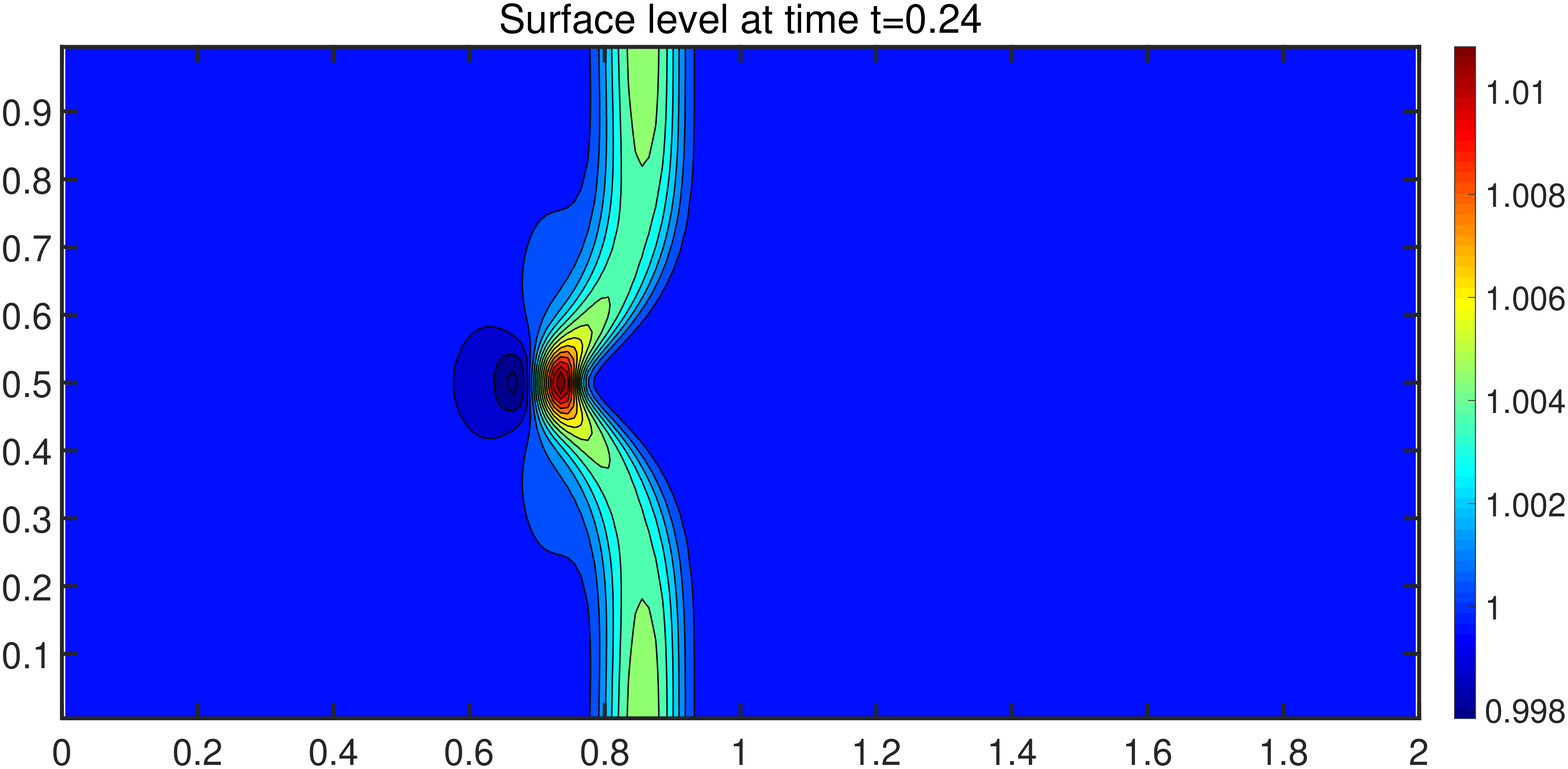}
  }
  {\centering
     \includegraphics[width= 7cm,scale=1]{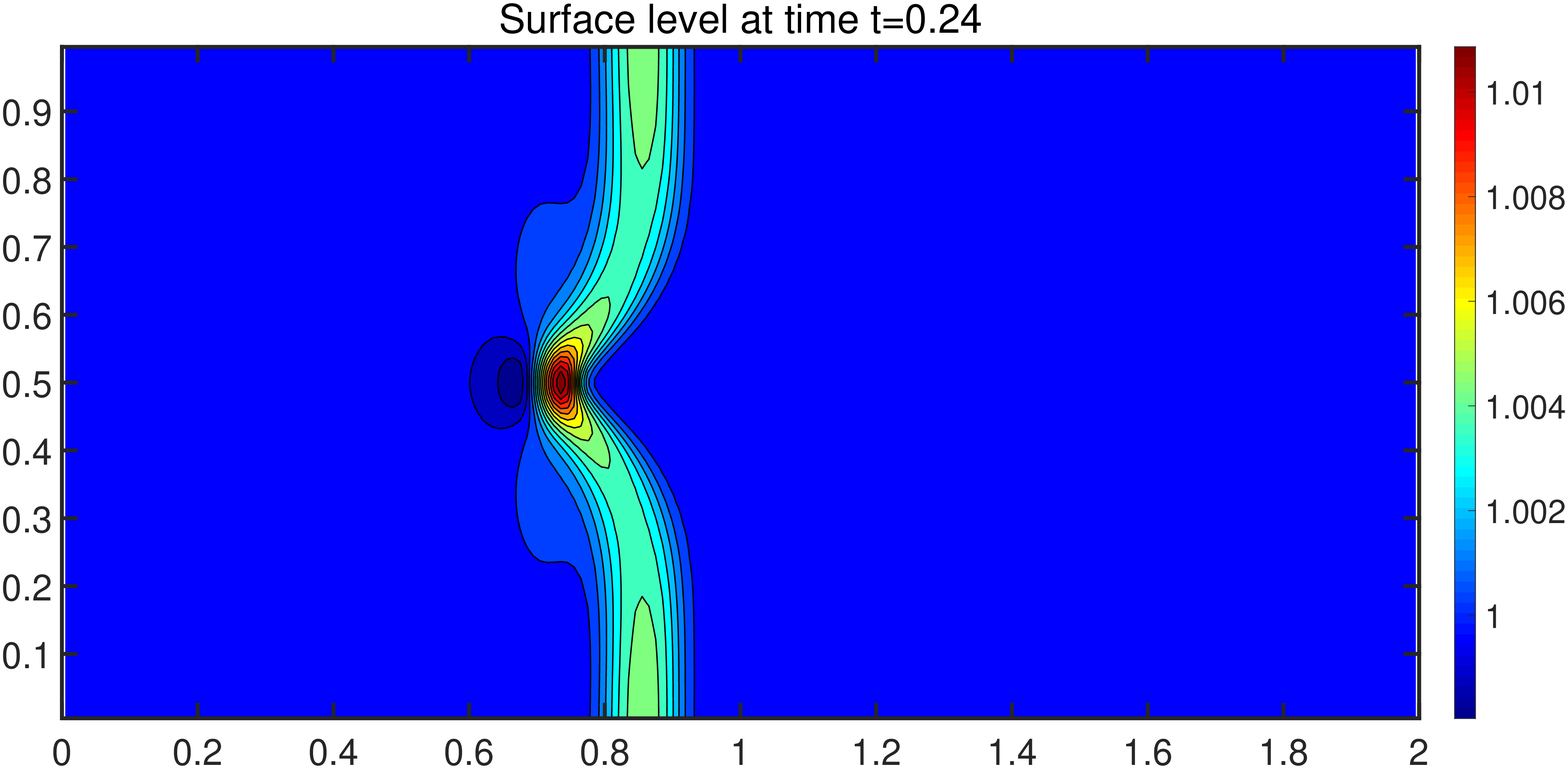}
  }
  {\centering
     \includegraphics[width= 7cm,scale=1]{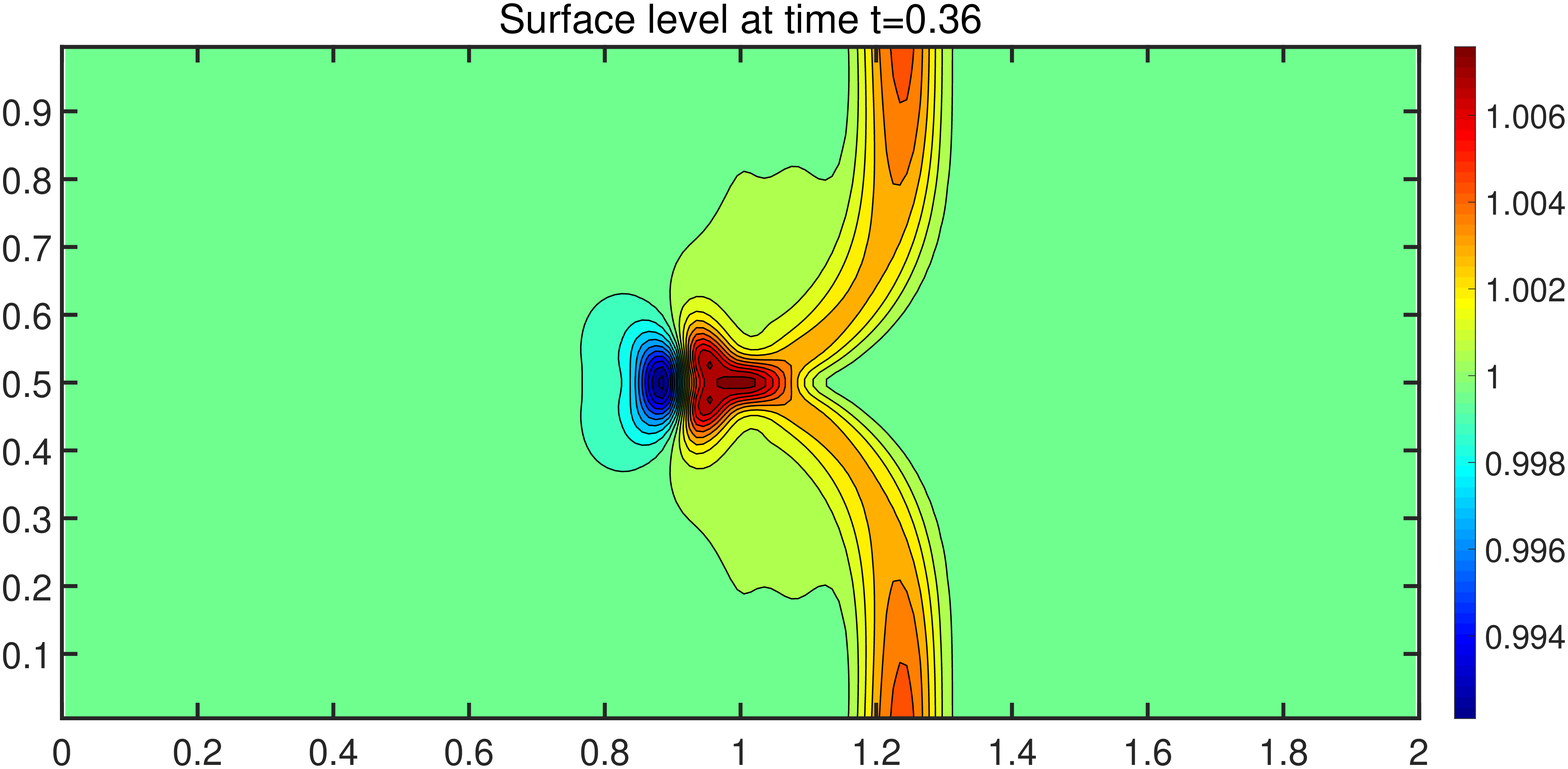}
  }
  {\centering
     \includegraphics[width= 7cm,scale=1]{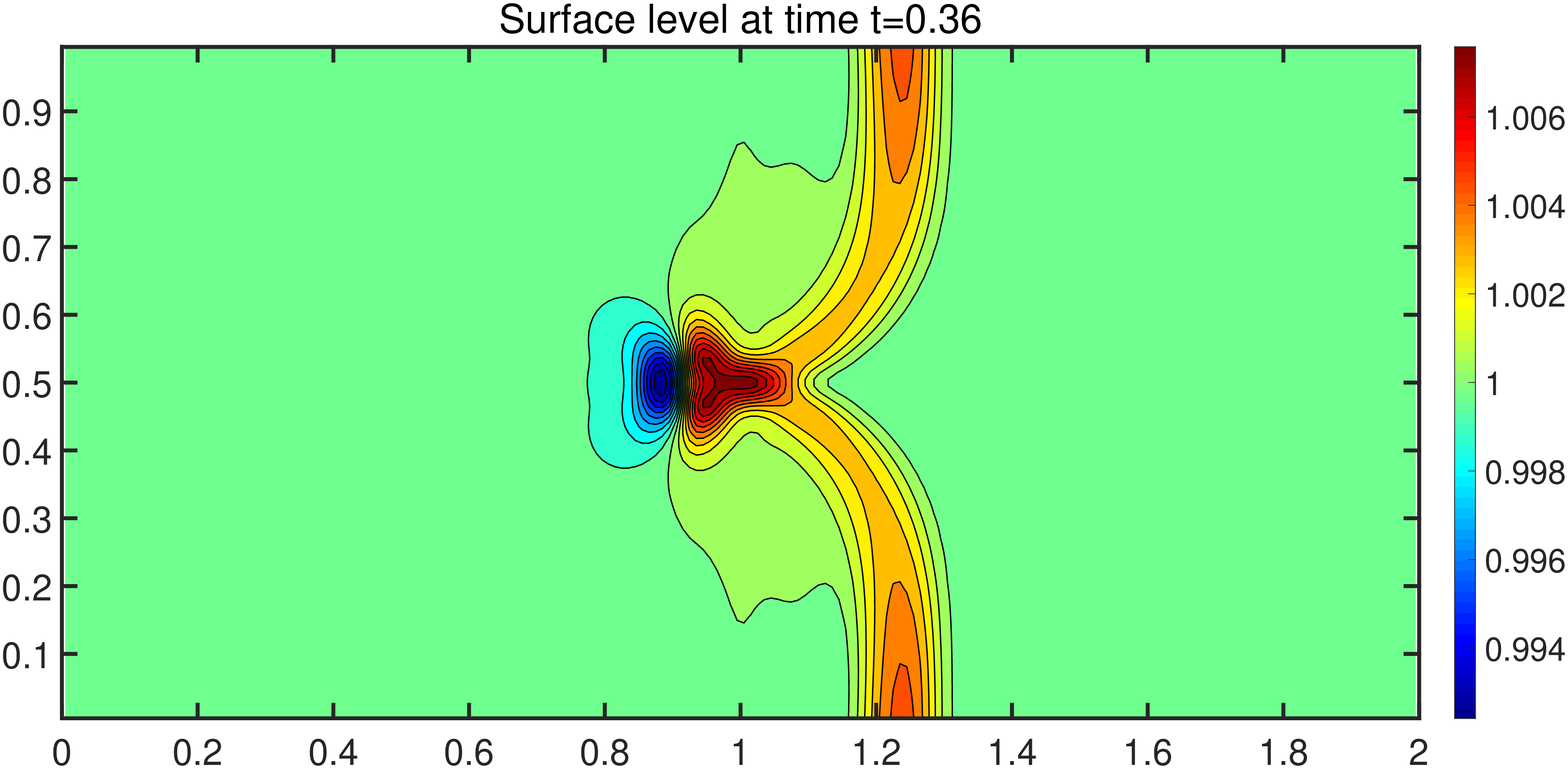}
  }
  {\centering
     \includegraphics[width= 7cm,scale=1]{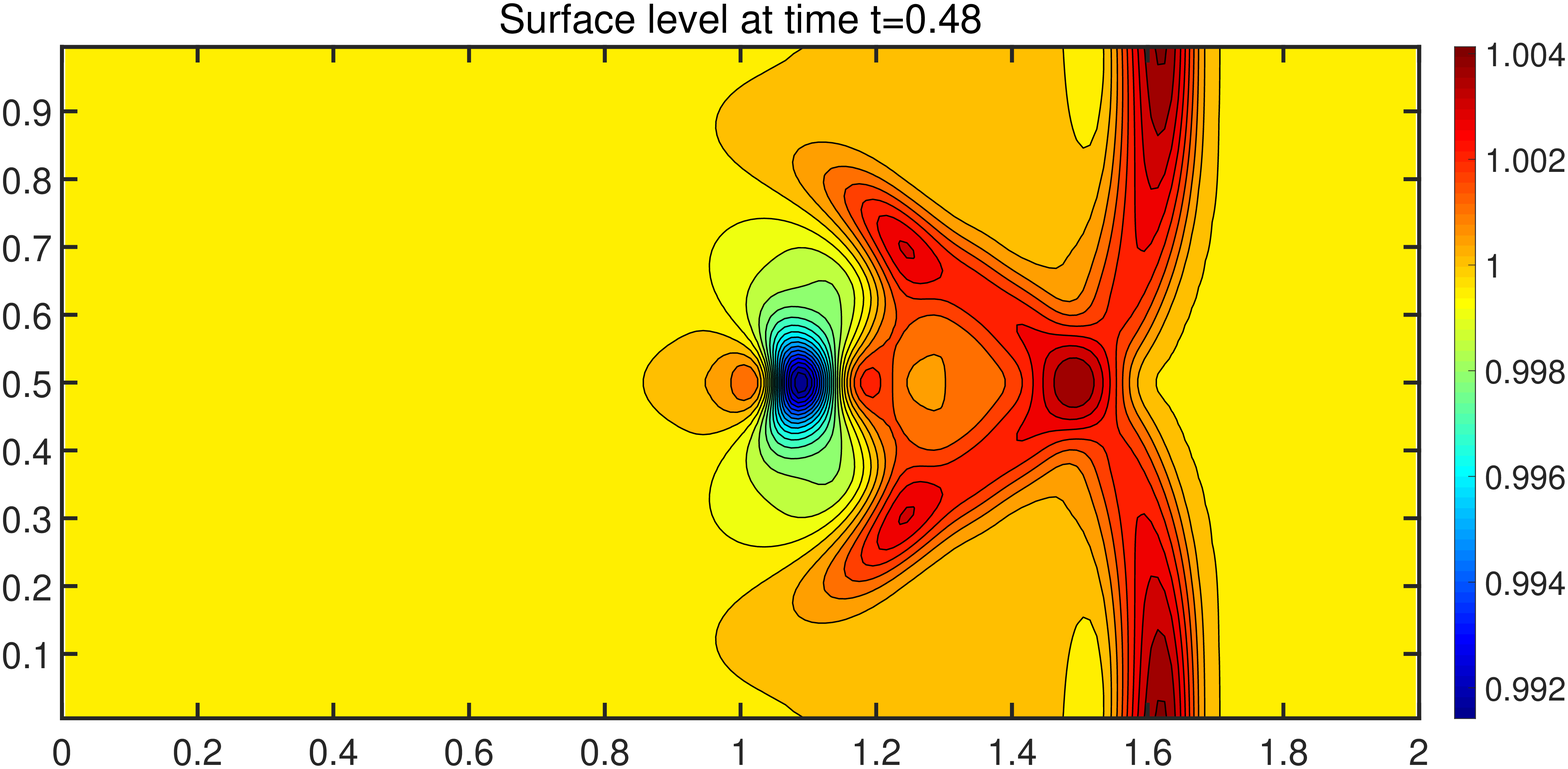}
  }
  {\centering
     \includegraphics[width= 7cm,scale=1]{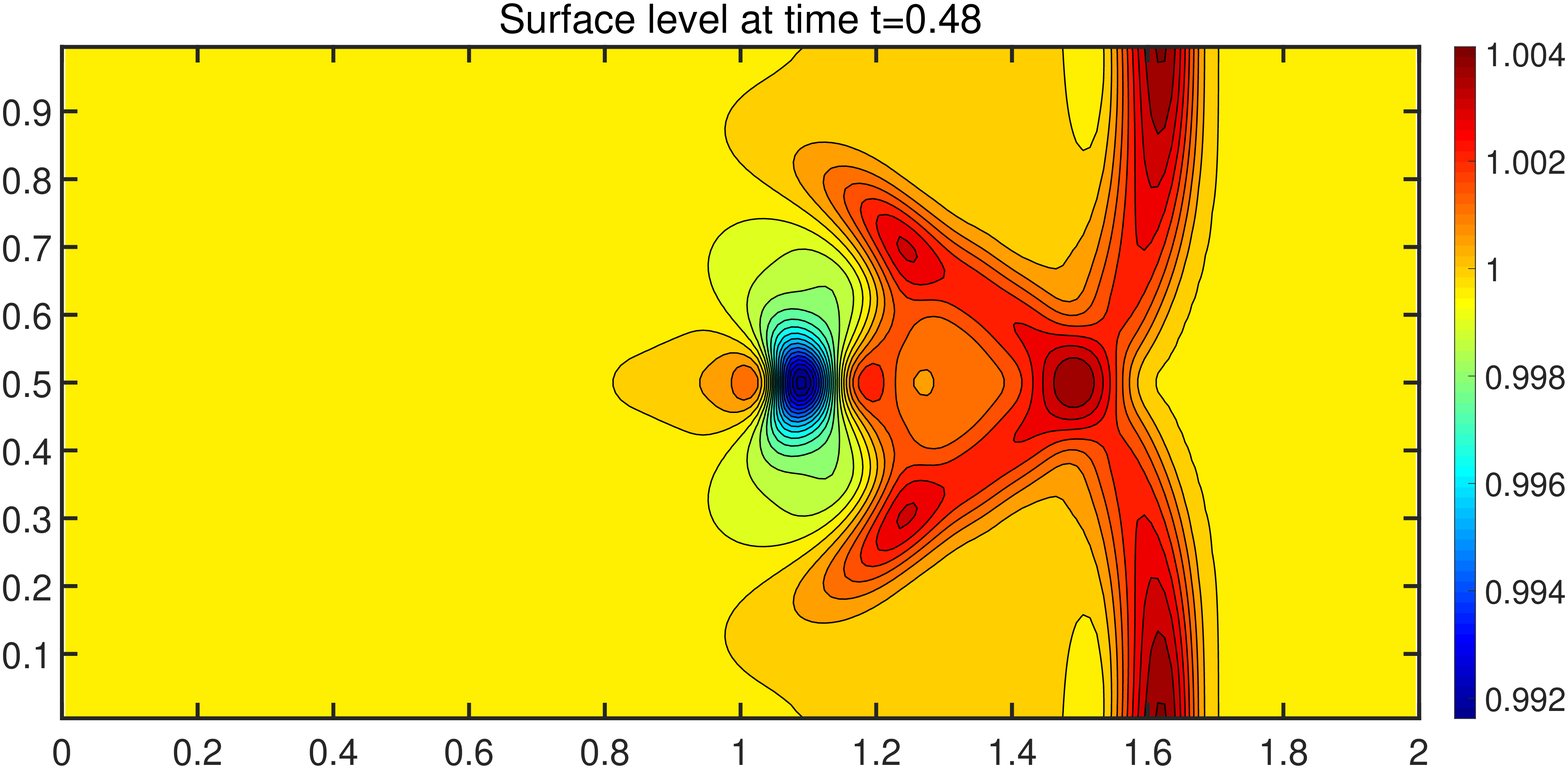}
  }
  {\centering
     \includegraphics[width= 7cm,scale=1]{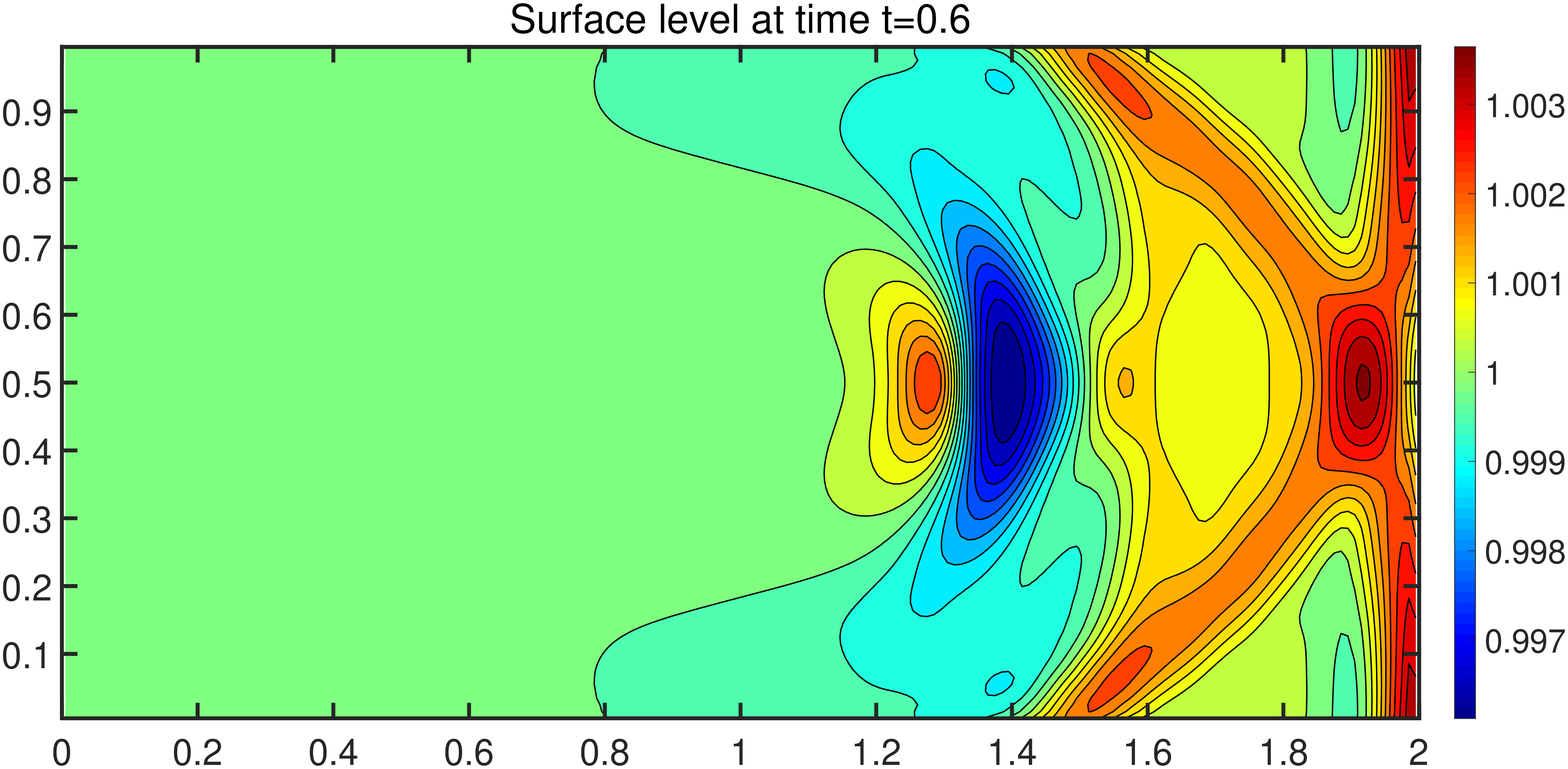}
  }
  {\centering
     \includegraphics[width= 7cm,scale=1]{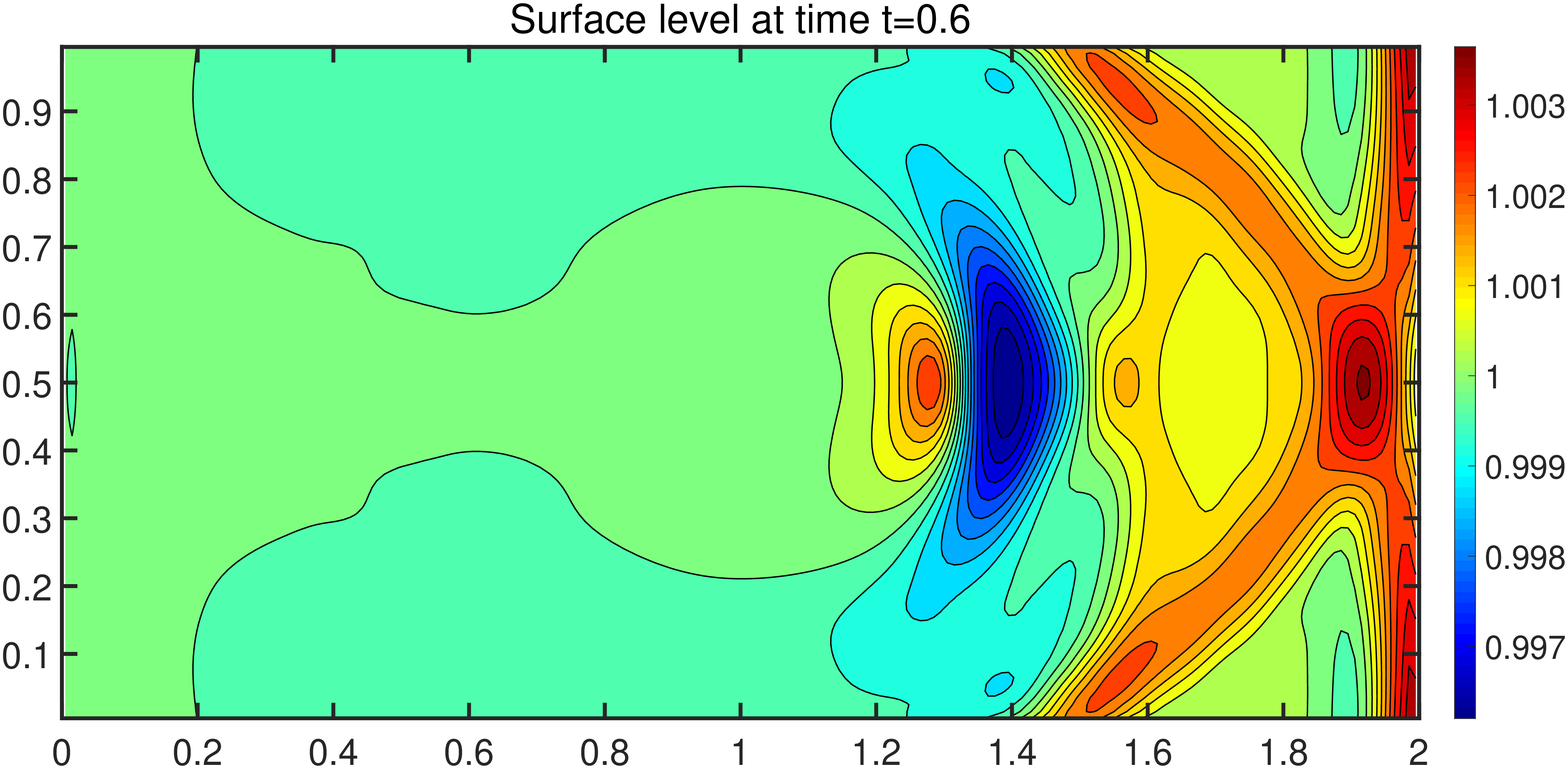}
  }
 \caption{Example \ref{perturbation2D}: The contours of the surface level with 30 uniformly spaced contour lines, by using 200$\times$100  moving mesh for $\epsilon= 0.01$. From top to bottom: at time $t=0.12$ from 0.99942 to 1.00656;  at time $t=0.24$ from 0.99318 to 1.01649;  at time $t=0.36$ from 0.98824 to 1.01161;  at time $t=0.48$ from 0.99023 to 1.00516;   at time $t=0.6$ from 0.99534 to 1.00629.  Left: results by the  ALE-WENO hybrid scheme with special source term treatment. Right: results by the non well-balanced ALE-WENO scheme.}\label{surfacehydro2D}
\end{figure}

\begin{figure}[htbp]
  \centering
   {\centering
     \includegraphics[width= 7cm,scale=1]{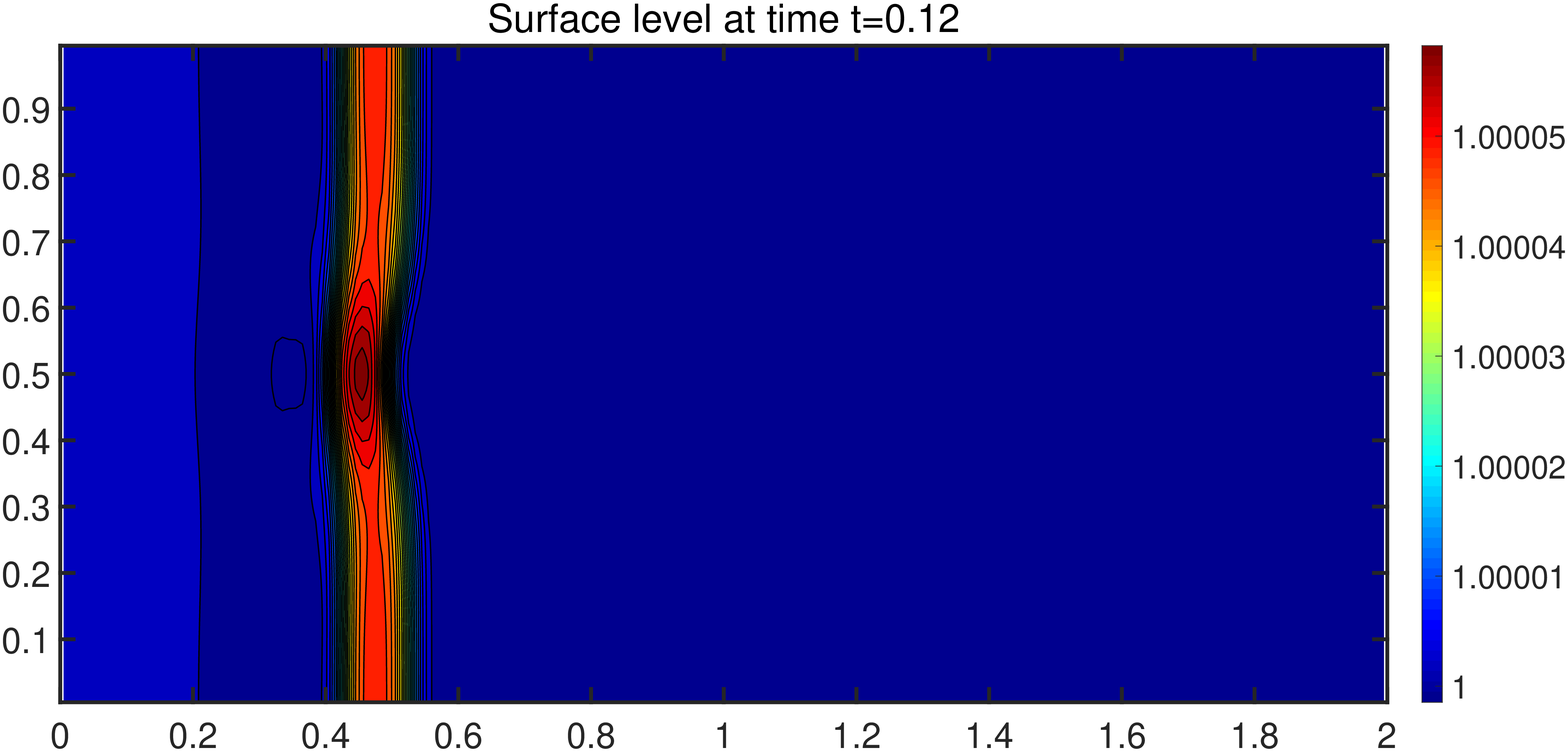}
  }
  {\centering
     \includegraphics[width= 7cm,scale=1]{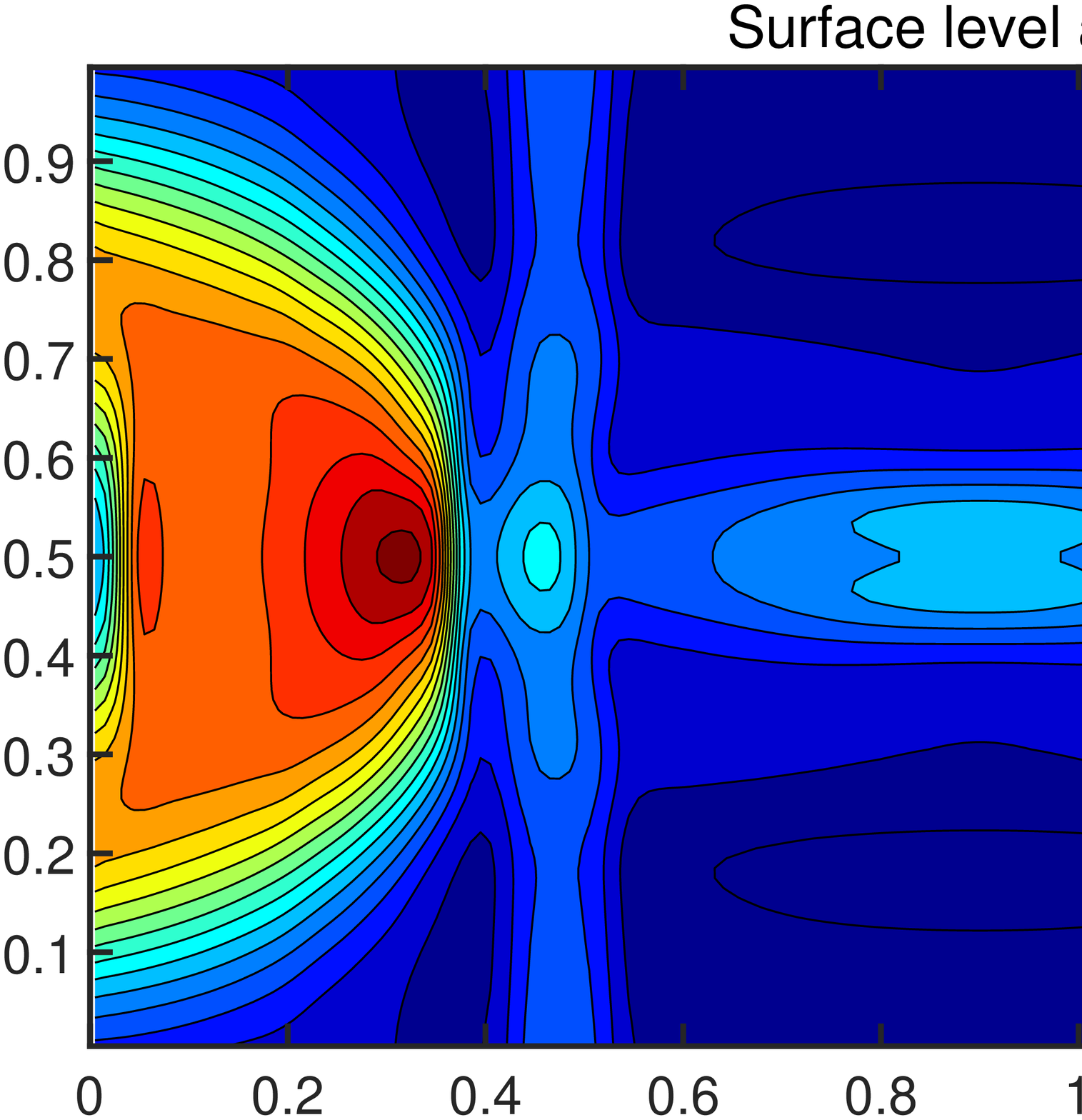}
  }
  {\centering
     \includegraphics[width= 7cm,scale=1]{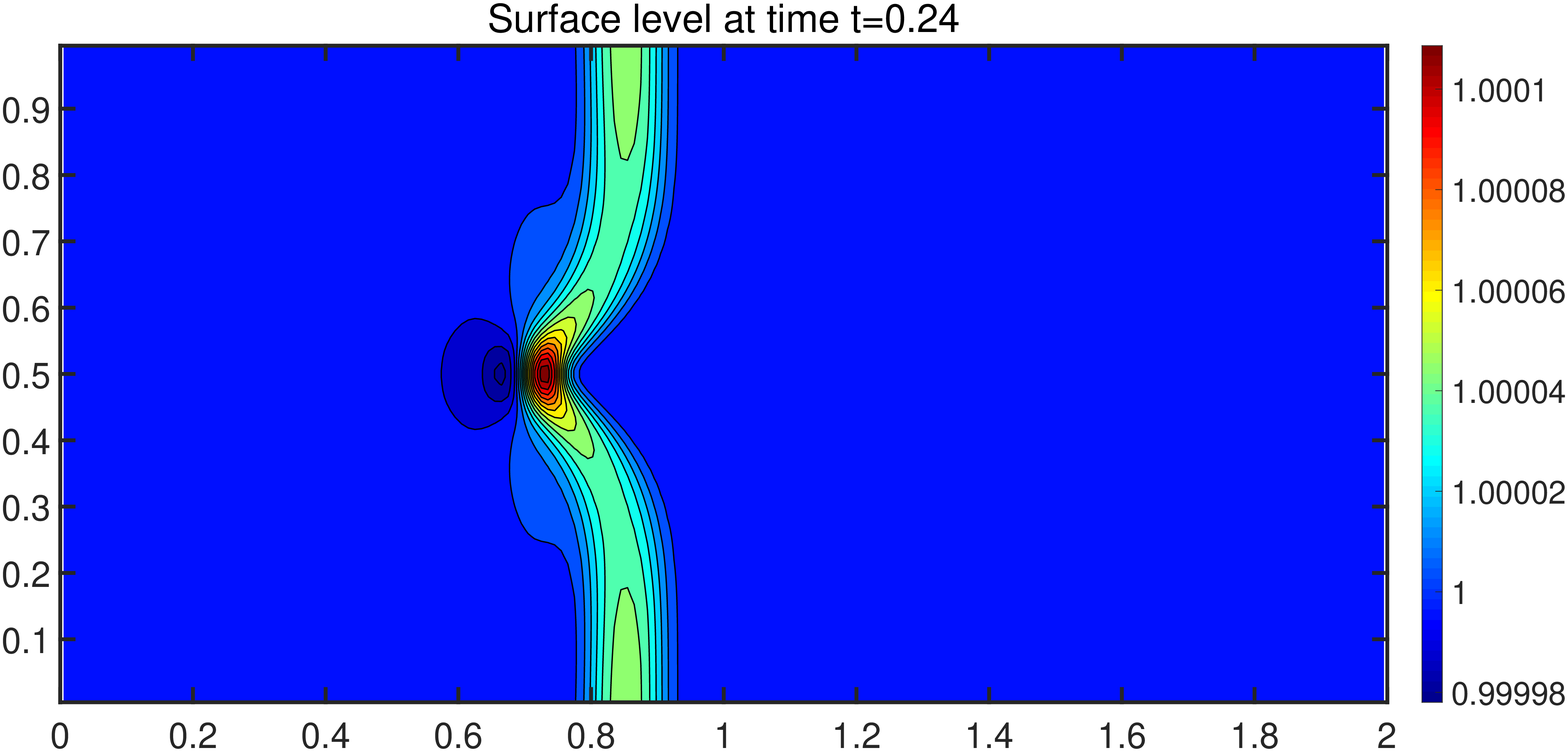}
  }
  {\centering
     \includegraphics[width= 7cm,scale=1]{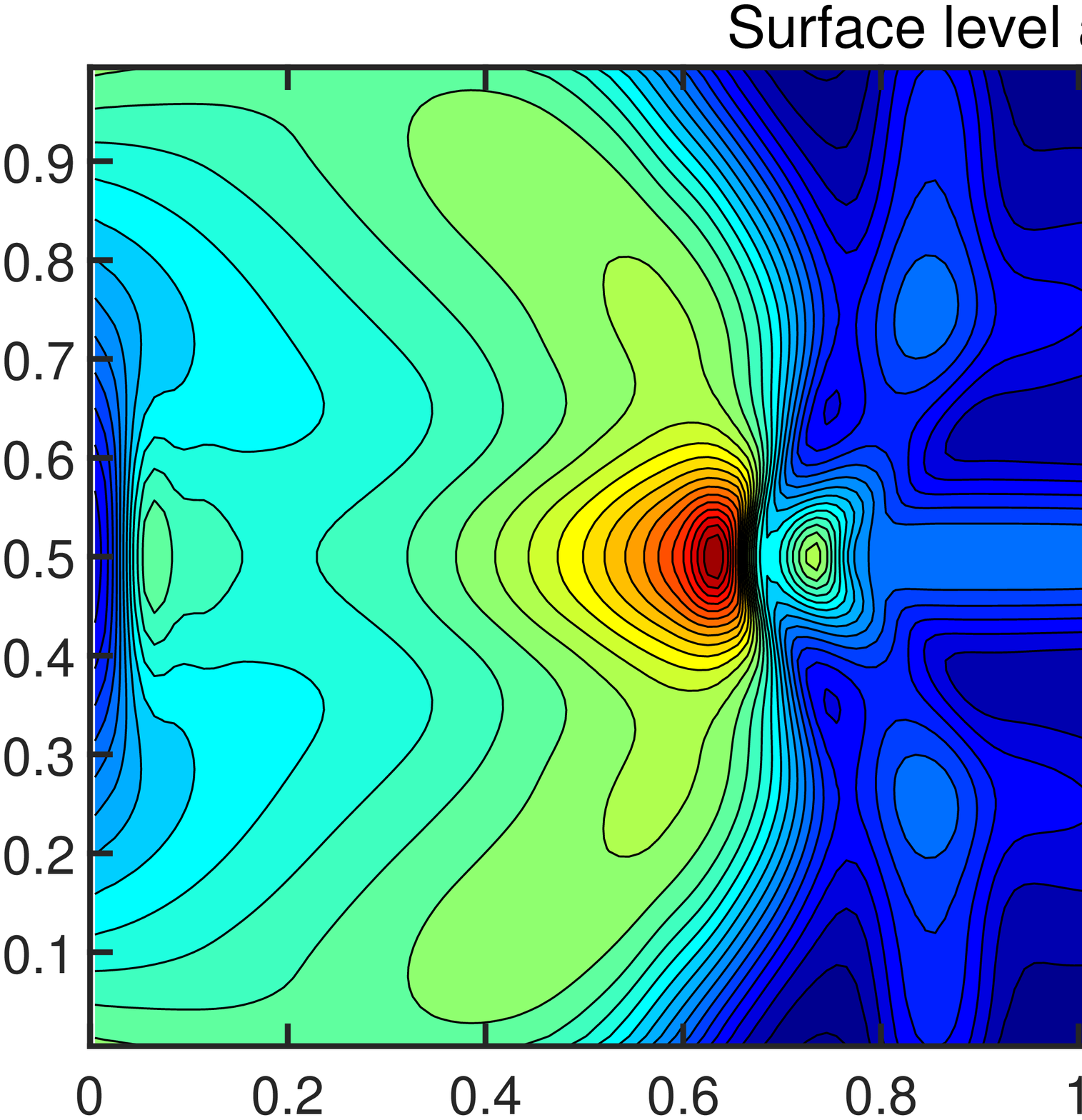}
  }
  {\centering
     \includegraphics[width= 7cm,scale=1]{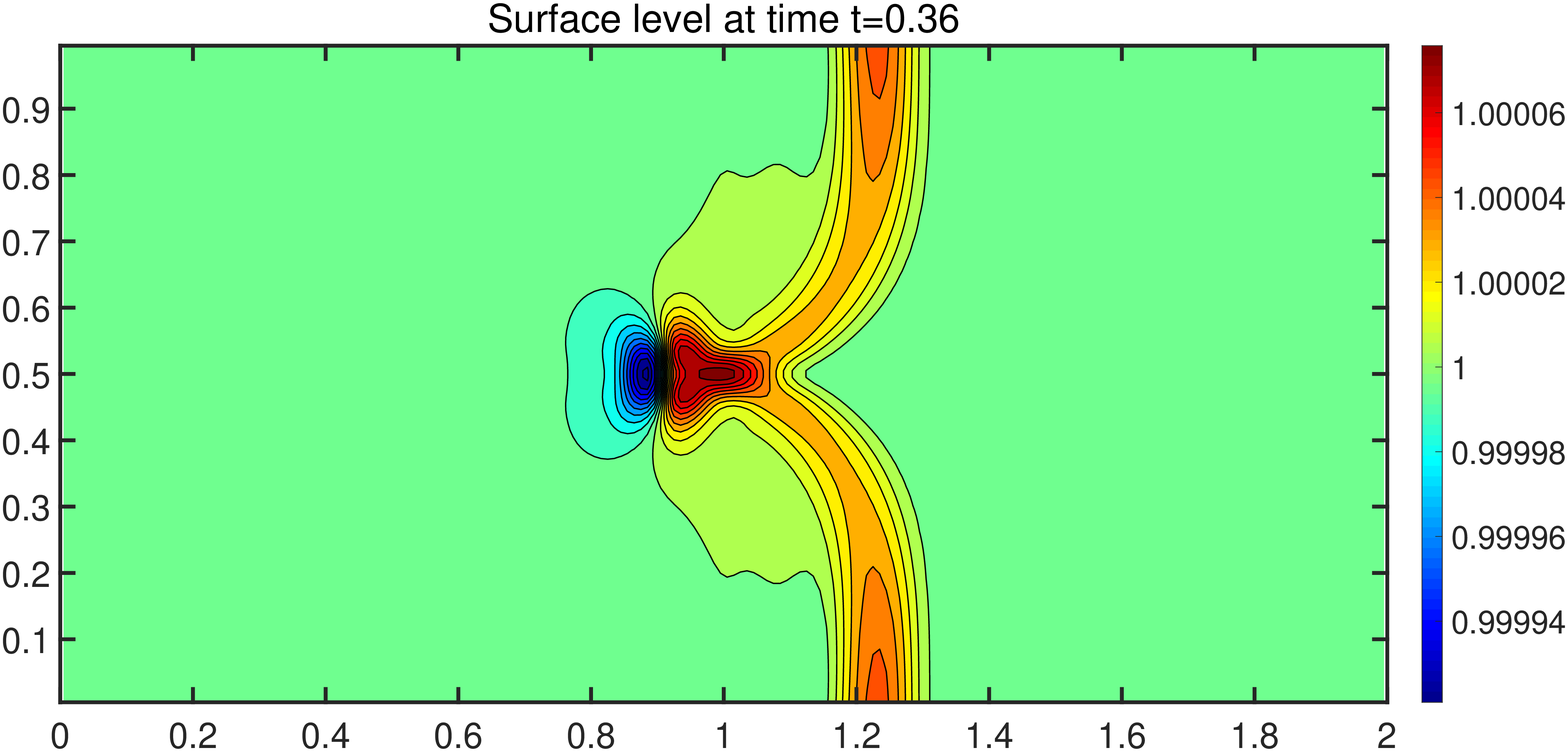}
  }
  {\centering
     \includegraphics[width= 7cm,scale=1]{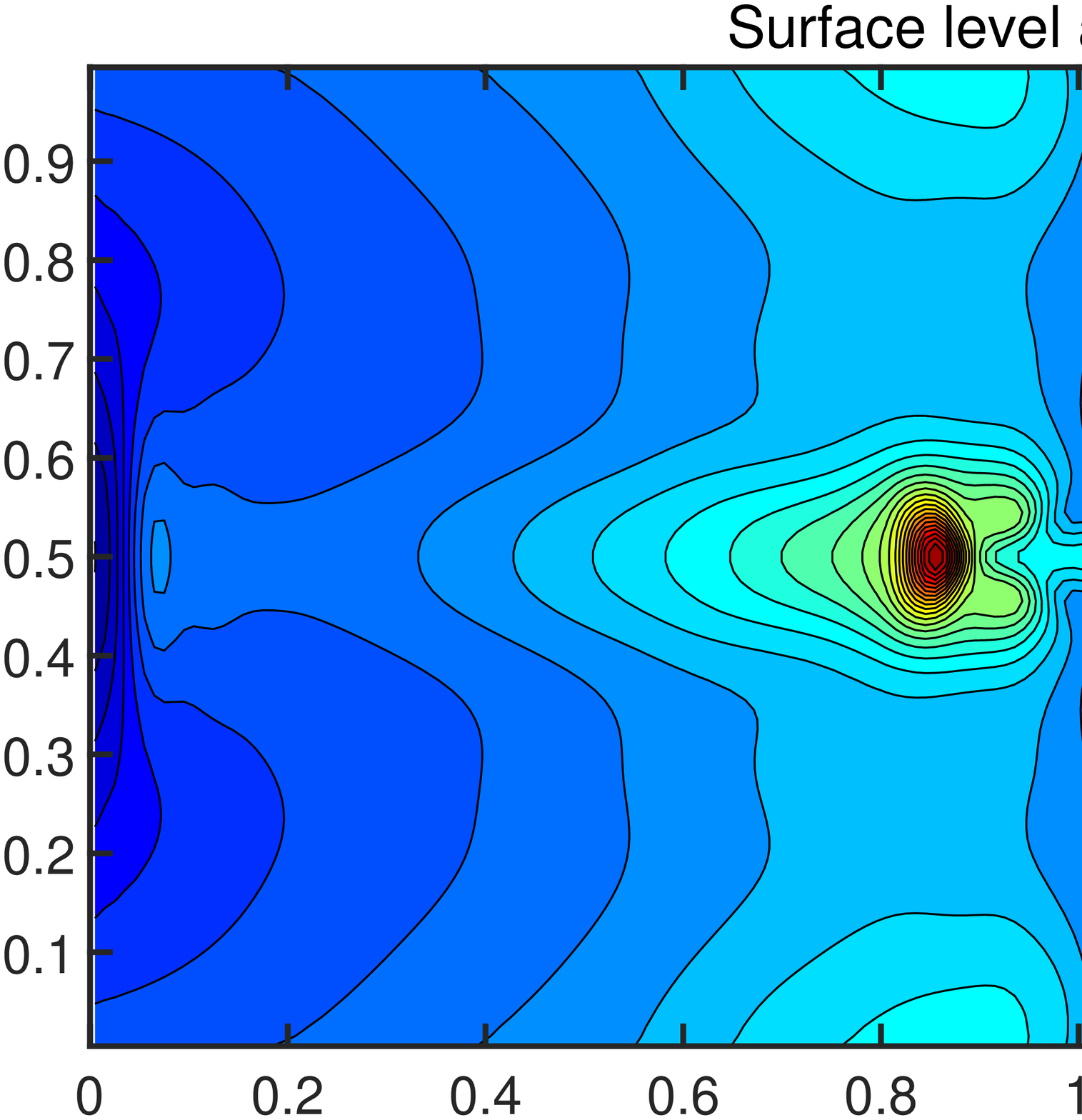}
  }
  {\centering
     \includegraphics[width= 7cm,scale=1]{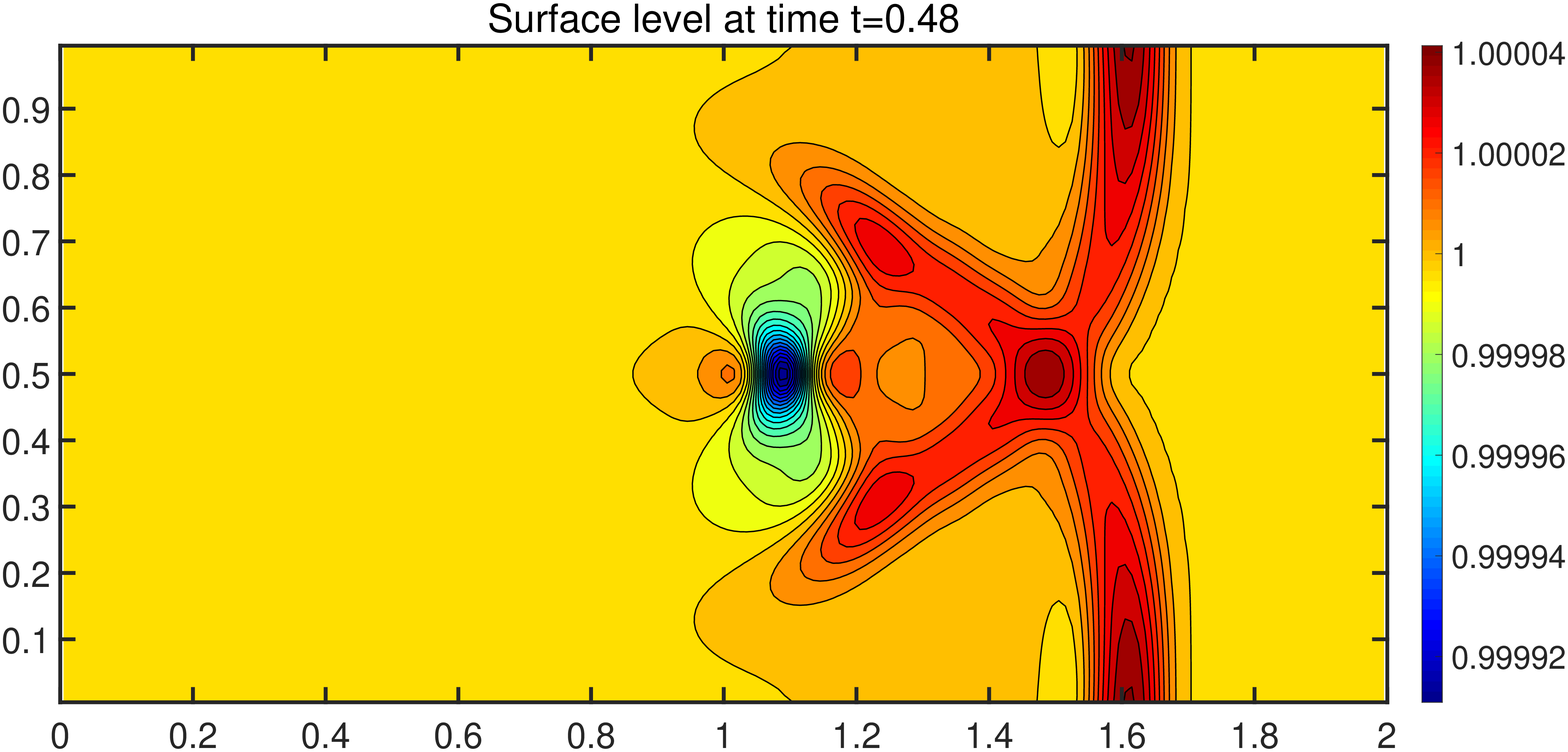}
  }
  {\centering
     \includegraphics[width= 7cm,scale=1]{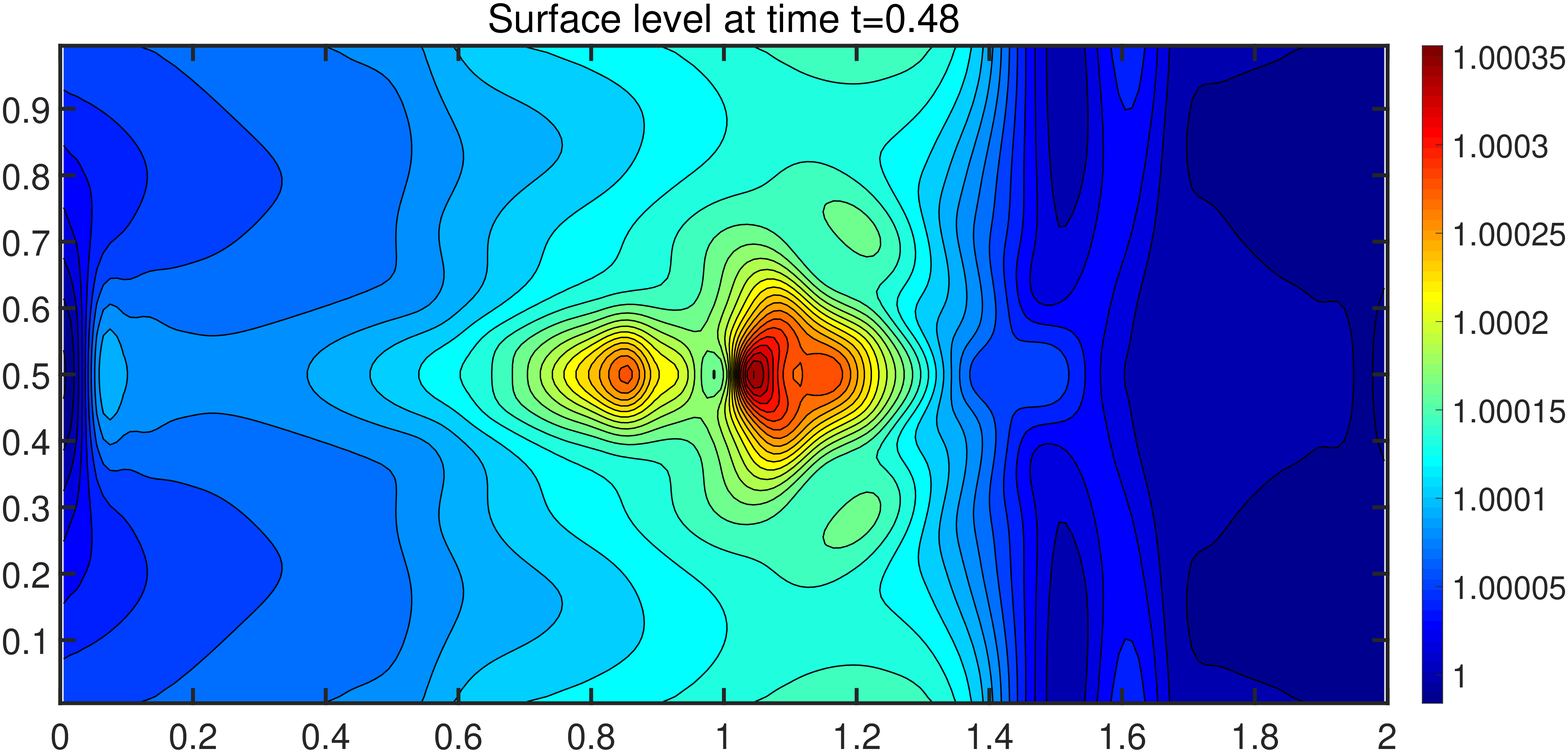}
  }
  {\centering
     \includegraphics[width= 7cm,scale=1]{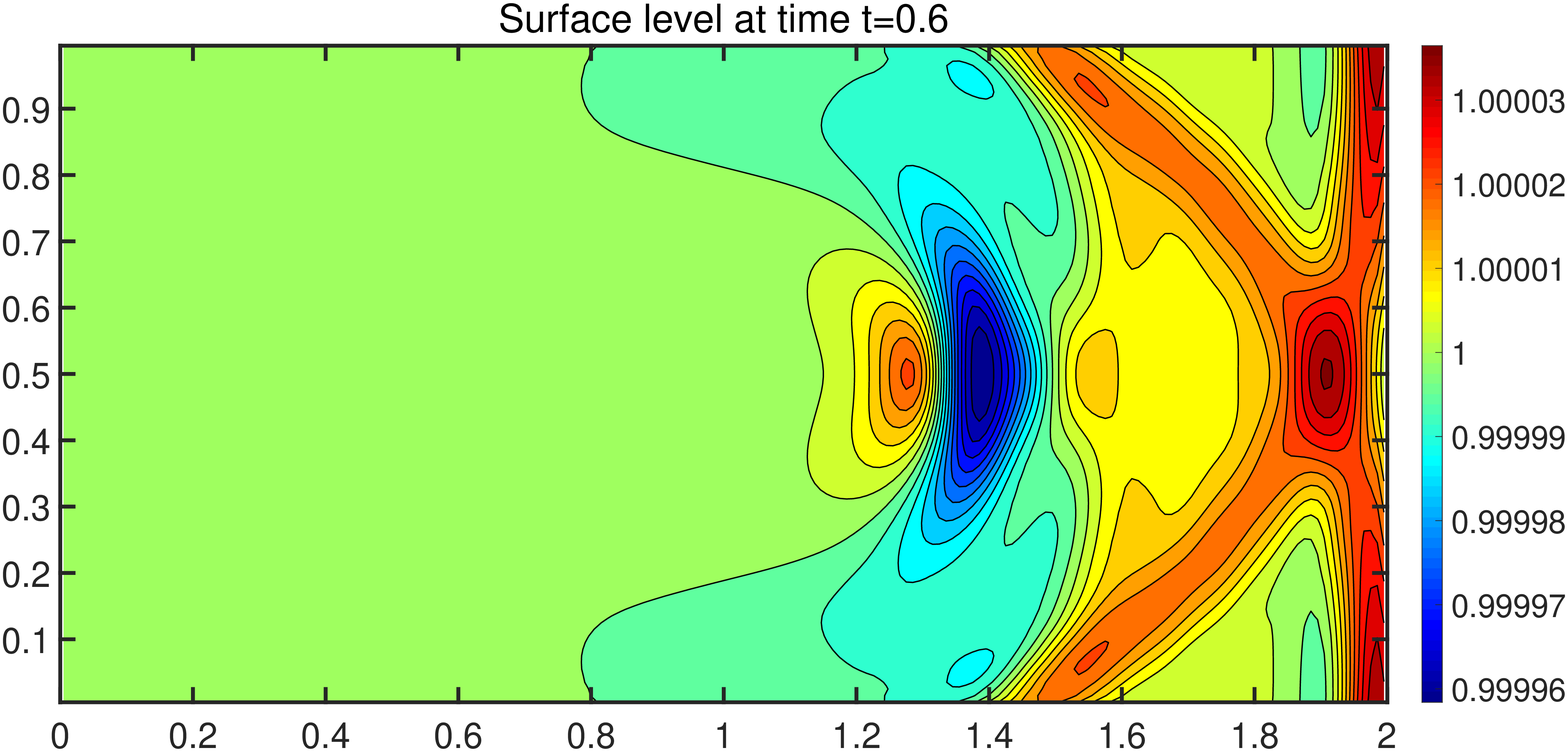}
  }
  {\centering
     \includegraphics[width= 7cm,scale=1]{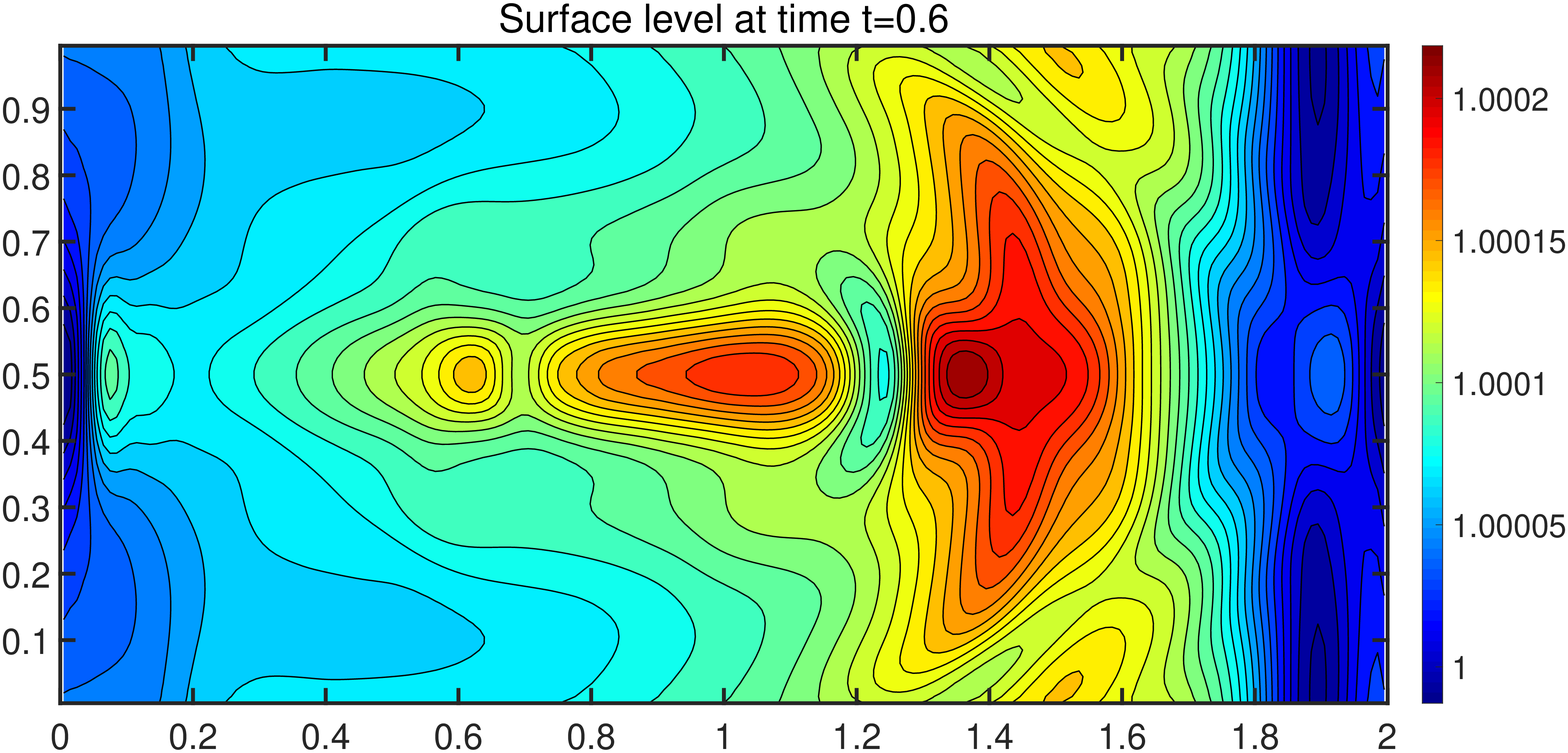}
  }
 \caption{Example \ref{perturbation2D}:  The contours of the surface level with 30 uniformly spaced contour lines, by using 200$\times$100  moving mesh for $\epsilon= 0.0001$. From top to bottom: results at times $t=0.12, 0.24,0.36,0.48,0.6$.  Left: results by the  ALE-WENO hybrid scheme with special source term treatment. Right: results by the non well-balanced ALE-WENO scheme. }\label{surfacespecial2D}
\end{figure}

\section{Conclusion}\label{se:co}
We have developed the structure-preserving finite volume ALE-WENO hybrid schemes for the shallow water equations in one and two dimensions. Rigorous theoretical analysis has shown that the ALE-WENO scheme maintains the positivity-preserving property, which relies on the geometric conservation law on moving meshes. We also developed the well-balanced finite volume ALE-WENO hybrid schemes on moving meshes based on the hydrostatic reconstruction and special source term treatment techniques, in which the positivity-preserving property and the well-balanced property can be satisfied simultaneously.
Numerical experiments in different circumstances were provided to illustrate these schemes' high order accuracy, positivity-preserving, and essentially oscillatory-free properties.
A better agreement can be observed with suitable moving meshes for the specific problem than those calculated on static grids. Meanwhile, the well-balanced schemes on moving meshes can maintain the exact hydrostatic state solutions up to machine error and perform better when capturing the complex small features close or away from the equilibrium state. In our future work, finite volume ALE-WENO hybrid schemes for moving equilibrium state of the shallow water equations and other certain hydrostatic states of conservation balance laws will be considered.

\setcounter{equation}{0}
\renewcommand\theequation{A.\arabic{equation}}
\renewcommand{\appendixname}{Appendix~\Alph{section}}
\begin{appendix}
\section{Proof of Proposition \ref{propp1} }\label{a1}
The scheme with the modified Euler Forward time discretization satisfied by the cell average of the bottom topography $b$ and the surface level $\eta$ can be written as
\begin{align*}
J_{\mathcal{T}_{ij}^{n+1}} \overline {b}_{ij}^{n+1} &= J_{\mathcal{T}_{ij}^{n}}\left(\overline {b}_{ij}^{n} +  \dfrac{1}{\Delta_{ij}^{n}}\Delta t\sum\limits_{l=1}^4 \displaystyle{\int_{e_{ij}^{l,n}} \widehat {\boldsymbol {\mathcal B}}(\boldsymbol w,b^{-},b^{+}, \boldsymbol n_l) \ ds }\right),\\
J_{\mathcal{T}_{ij}^{n+1}} \overline {\eta}_{ij}^{n+1} &= J_{\mathcal{T}_{ij}^{n}}\left(\overline {\eta}_{ij}^{n} -  \dfrac{1}{\Delta_{ij}^{n}}\Delta t\sum\limits_{l=1}^4 \displaystyle{\int_{e_{ij}^{l,n}} \widehat {\boldsymbol {\mathcal H}}^{(1)}(\boldsymbol w,\boldsymbol q^{-},\boldsymbol q^{+}, \boldsymbol n_l) \ ds }\right).
\end{align*}
Let $\bar h_{ij}^n  = \bar \eta_{ij}^{n}-\bar b_{ij}^{n}$, subtracting the above two equations and rewriting the scheme as
\begin{equation}\label{pphbar}
\overline {h}_{ij}^{n+1} = \dfrac{ J_{\mathcal{T}_{ij}^{n}}}{J_{\mathcal{T}_{ij}^{n+1}} } \left(\overline {h}_{ij}^{n} - \dfrac{\Delta t}{\Delta_{ij}^{n} }\sum_{l=1}^4 \displaystyle{\int_{e_{ij}^{l,n}} \widehat {\boldsymbol {\mathcal H}}^{(1)}(\boldsymbol w,\boldsymbol q^{-},\boldsymbol q^{+}, \boldsymbol n_l) + \widehat {\boldsymbol {\mathcal B}}(\boldsymbol w,b^{-},b^{+}, \boldsymbol n_l) \ ds }\right),
\end{equation}
where
\begin{equation}\label{pphhat}
\begin{aligned}
  &\widehat {\boldsymbol {\mathcal H}}^{(1)}(\boldsymbol w,\boldsymbol q^{-},\boldsymbol q^{+}, \boldsymbol n_l) + \widehat {\boldsymbol {\mathcal B}}(\boldsymbol w,b^{-},b^{+}, \boldsymbol n_l)\\
   & = \dfrac{1}{2}\left[ h^{-}\left( (u^{-}-w^x,v^{-}-w^y) \cdot\boldsymbol n_l +\alpha\right) + h^{+}\left( (u^{+}-w^x,v^{+}-w^y) \cdot\boldsymbol n_l-\alpha\right)  \right]\leq \alpha h^{-}.
   \end{aligned}
\end{equation}
The last inequality sees from the definition of  $\alpha$.
Using the  bilinear mapping  (\ref{mappingbilinear})  and the definition  (\ref{deltajn}) we can get  
\begin{equation}\label{barhjn}
\begin{aligned}
\bar h_{ij}^n & = \dfrac{1}{\Delta_{ij}^n} \displaystyle{\iint_{\mathcal T_{ij}^n}   h_{ij}^n(x,y)  \ dxdy} =  \dfrac{1}{\Delta_{ij}} \displaystyle{\iint_{\widetilde{\mathcal T} } \widetilde h_{ij}^n(\xi,\zeta)\mathcal{J}_{\mathcal T_{ij}^n}(\xi,\zeta) \ d\xi d\zeta}\\
& = (1-\sigma_1-\sigma_k ) \delta_{ij}^1+\dfrac{1}{\Delta_{ij}^n}\sigma_1 \sum\limits_{\beta=1}^k  \sigma_\beta \widetilde h^-_{i_\beta,j_1} \mathcal{J}(\xi_{i_\beta},\zeta_{j_1})  + \frac{1}{\Delta_{ij}^n}\sigma_k \sum\limits_{\beta=1}^k  \sigma_\beta \widetilde h^-_{i_\beta,j_k} \mathcal{J}(\xi_{i_\beta},\zeta_{j_k})\\
& = (1-\sigma_1-\sigma_k ) \delta_{ij}^2+\dfrac{1}{\Delta_{ij}^n} \sigma_1 \sum\limits_{\nu=1}^k  \sigma_\nu \widetilde h^-_{i_1,j_\nu} \mathcal{J}(\xi_{i_1},\zeta_{j_\nu})  + \dfrac{1}{\Delta_{ij}^n} \sigma_k \sum\limits_{\nu=1}^k  \sigma_\nu \widetilde h^-_{i_k,j_\nu} \mathcal{J}(\xi_{i_k},\zeta_{j_\nu}).
\end{aligned}
\end{equation}
Let $ \left|e_{ij}^{l,n}\right|$ be the length of an edge at time $t_n$ and $$\mu_l = \dfrac{\Delta t}{\Delta _{ij}^n} |e_{ij}^{l,n}|, \quad \mu = \sum\limits_{l=1}^4 \mu_l = \dfrac{\Delta t}{\Delta _{ij}^n}\sum\limits_{l=1}^4 |e_{ij}^{l,n}|.$$
Denoting $\widetilde e^1$ to be the edge of the reference domain which lies on  $\xi$-axis, and $\widetilde e^2,\widetilde e^3,\widetilde e^4$ are sorted  anticlockwise. Using the metric transformation and  plugging  (\ref{pphhat}), (\ref{barhjn}) back into (\ref{pphbar}) we have
\begin{align*}
 \overline {h}_{ij}^{n+1} &\geq \dfrac{ J_{\mathcal{T}_{ij}^{n}}}{J_{\mathcal{T}_{ij}^{n+1}} } \left(\dfrac{\mu_1}{\mu}\overline {h}_{ij}^{n} + \dfrac{\mu_2}{\mu}\overline {h}_{ij}^{n}+\dfrac{\mu_3}{\mu}\overline {h}_{ij}^{n}+\dfrac{\mu_4}{\mu}\overline {h}_{ij}^{n}- \dfrac{\Delta t}{\Delta_{ij}^{n} }\sum_{l=1}^4 |e_{ij}^{l,n}| \displaystyle{\int_{\widetilde e^l} \alpha \widetilde h^{-} \ d\tilde s}\right)\\
 &\geq \dfrac{ J_{\mathcal{T}_{ij}^{n}}}{J_{\mathcal{T}_{ij}^{n+1}} }
  \left[\dfrac{\mu_1}{\mu}\left((1-\sigma_1-\sigma_k ) \delta_{ij}^1+c_0 \sigma_1 \sum\limits_{\beta=1}^k  \sigma_\beta \widetilde h^-_{i_\beta,j_1} + c_0 \sigma_k \sum\limits_{\beta=1}^k  \sigma_\beta \widetilde h^-_{i_\beta,j_k}  \right) - \mu_1 \sum\limits_{\beta=1}^k  \sigma_\beta \alpha \widetilde h^-_{i_\beta,j_1}\right.\\
 &\qquad  \qquad + \dfrac{\mu_2}{\mu}\left( (1-\sigma_1-\sigma_k ) \delta_{ij}^2+c_0 \sigma_1 \sum\limits_{\nu=1}^k  \sigma_\nu \widetilde h^-_{i_1,j_\nu} + c_0 \sigma_k \sum\limits_{\nu=1}^k  \sigma_\nu \widetilde h^-_{i_k,j_\nu} \right) - \mu_2 \sum\limits_{\nu=1}^k  \sigma_\nu \alpha\widetilde h^-_{i_k,j_\nu}\\
 &\qquad  \qquad+\dfrac{\mu_3}{\mu}\left((1-\sigma_1-\sigma_k ) \delta_{ij}^1+c_0 \sigma_1 \sum\limits_{\beta=1}^k  \sigma_\beta \widetilde h^-_{i_\beta,j_1} + c_0 \sigma_k \sum\limits_{\beta=1}^k  \sigma_\beta \widetilde h^-_{i_\beta,j_k} \right) - \mu_3 \sum\limits_{\beta=1}^k  \sigma_\beta \alpha \widetilde h^-_{i_\beta,j_k}\\
 &\left. \qquad  \qquad+\dfrac{\mu_4}{\mu}\left( (1-\sigma_1-\sigma_k ) \delta_{ij}^2+c_0 \sigma_1 \sum\limits_{\nu=1}^k  \sigma_\nu \widetilde h^-_{i_1,j_\nu} + c_0 \sigma_k \sum\limits_{\nu=1}^k  \sigma_\nu \widetilde h^-_{i_k,j_\nu} \right) - \mu_4 \sum\limits_{\nu=1}^k  \sigma_\nu \alpha\widetilde h^-_{i_1,j_\nu}\right].
   \end{align*}
The second  inequality is due to  (\ref{c0}).
Eventually, we can get:
\begin{align*}
 \overline {h}_{ij}^{n+1} &\geq  \dfrac{ J_{\mathcal{T}_{ij}^{n}}}{J_{\mathcal{T}_{ij}^{n+1}} } \left[\dfrac{\mu_1+\mu_3}{\mu}(1-\sigma_1-\sigma_k ) \delta_{ij}^1 + \dfrac{\mu_1}{\mu} c_0 \sigma_k \sum\limits_{\beta=1}^k  \sigma_\beta \widetilde h^-_{i_\beta,j_k}+ \dfrac{\mu_3}{\mu}c_0 \sigma_1 \sum\limits_{\beta=1}^k  \sigma_\beta \widetilde h^-_{i_\beta,j_1}\right.\\
 & \qquad  \qquad  \qquad  \qquad  \qquad  \quad \ \ \  + \mu_1 \sum\limits_{\beta=1}^k  \sigma_\beta \widetilde h^-_{i_\beta,j_1} \left(\dfrac{1}{\mu}c_0  \sigma_1-\alpha \right) + \mu_3 \sum\limits_{\beta=1}^k  \sigma_\beta \widetilde h^-_{i_\beta,j_k} \left(\dfrac{1}{\mu}c_0  \sigma_k-\alpha \right) \\
 &\qquad  \qquad +\dfrac{\mu_2+\mu_4}{\mu} (1-\sigma_1-\sigma_k ) \delta_{ij}^2 + \dfrac{\mu_2}{\mu}c_0 \sigma_1 \sum\limits_{\nu=1}^k  \sigma_\nu \widetilde h^-_{i_1,j_\nu}  + \dfrac{\mu_4}{\mu}c_0\sigma_k \sum\limits_{\nu=1}^k  \sigma_\nu \widetilde h^-_{i_k,j_\nu}\\
 & \left. \qquad  \qquad  \qquad  \qquad  \qquad  \quad \ \ \  + \mu_2 \sum\limits_{\nu=1}^k  \sigma_\nu \widetilde h^-_{i_k,j_\nu} \left(\dfrac{1}{\mu}c_0 \sigma_k-\alpha \right) + \mu_4 \sum\limits_{\nu=1}^k  \sigma_\nu \widetilde h^-_{i_1,j_\nu} \left(\dfrac{1}{\mu}c_0  \sigma_1-\alpha \right)\right].
   \end{align*}
Using the equations (\ref{Jacobi2D}) and (\ref{c0}), combined with the CFL condition  (\ref{CFLcondition}) we obtain
\begin{align*}
\dfrac{ J_{\mathcal{T}_{ij}^{n}}}{J_{\mathcal{T}_{ij}^{n+1}} } =  1-\dfrac{1}{J_{\mathcal{T}_{ij}^{n+1}}} \left(J_{\mathcal{T}_{ij}^{n}} - J_{\mathcal{T}_{ij}^{n+1}}\right)
 \geq 1-\dfrac{1}{c_0 \Delta _{ij}^{n+1}}\Delta t  |\nabla\cdot {\boldsymbol w^n}| J_{\mathcal{T}_{ij}^{n}}\geq 0.
\end{align*}
Hence, by the CFL constraint (\ref{CFLcondition}) as well as the non-negative values  $\widetilde h^-_{i_\beta,j_1}$, $\widetilde h^-_{i_\beta,j_k}$, $\delta_{ij}^1$, $\widetilde h^-_{i_1,j_\alpha}$, $\widetilde h^-_{i_k,j_\alpha}$, $\delta_{ij}^2$, the cell average of $\bar h_{ij}^{n+1}$ can also be  non-negative. The proof is completed. Notice that for the modified third order SSP-RK method, we can prove the positivity-preserving property in a similar way, thus we omit it here.

\end{appendix}

\subsection*{Acknowledgements}
Research of Yinhua Xia is supported by NSFC grant No 11871449.
Research of Yan Xu is supported by NSFC grant No. 12071455.

\addcontentsline{toc}{section}{References}
\bibliographystyle{abbrv}
\normalem
\bibliography{shortreference}
\end{document}